\documentclass[11pt, leqno]{amsart} 
\usepackage{amssymb,amscd,amsfonts,amsbsy}
\usepackage{latexsym}
\usepackage{exscale}
\usepackage{amsmath,amsthm,amsfonts}
\usepackage{mathrsfs}
\usepackage{xcolor} 
\usepackage[colorlinks=true,linkcolor=blue,citecolor=red,urlcolor=red, 
]{hyperref} 
\usepackage{esint} 
\usepackage{pifont}

\usepackage[utf8]{inputenc}

\calclayout

\allowdisplaybreaks

\theoremstyle{plain}
\newtheorem{theorem}[equation]{Theorem}
\newtheorem{lemma}[equation]{Lemma}
\newtheorem{corollary}[equation]{Corollary}
\newtheorem{proposition}[equation]{Proposition}

\theoremstyle{definition}
\newtheorem{definition}[equation]{Definition}

\theoremstyle{remark}
\newtheorem{remark}[equation]{Remark}

\numberwithin{equation}{section} 

\def\diam{\operatorname{diam}}
\def\div{\operatorname{div}}
\def\dist{\operatorname{dist}}
\def\supp{\operatorname{supp}}

\def\Cap{\operatorname{Cap}}
\def\BMO{\operatorname{BMO}}
\def\VMO{\operatorname{VMO}}

\def\Lip{\operatorname{Lip}}
\def\loc{{\operatorname{loc}}}

\def\osc{\operatorname{osc}}

\def\E{\mathscr{E}}

\def\H{\mathcal{H}}
\def\Z{\mathbb{Z}}

\def\R{\mathbb{R}}
\def\re{\mathbb{R}}
\def\pom{{\partial\Omega}}

\def\Rn{\mathbb{R}^n}
\def\W{\mathcal{W}}

\def\w{\omega}

\newcommand{\RR}{\mathbb R}
\newcommand{\NN}{\mathbb N}
\newcommand{\ZZ}{\mathbb Z}

\newcommand{\abs}[1]{\left\lvert #1 \right\rvert}
\newcommand{\norm}[1]{\left\lVert #1 \right\rVert}

\newcommand{\ree}{{\RR^{n+1}}}
\newcommand{\dc}{\dot{\mathscr{C}}}
\newcommand{\dcv}{\dot{\mathscr{C}}^{\varphi}}
\newcommand{\dcstar}{\dot{\mathscr{C}}^{*,\varphi}}
\newcommand{\Ephip}{{\mathscr{E}^{\varphi, p}}}

\renewcommand{\emptyset}{\mbox{\textup{\O}}}

\newcommand\restr[2]{\ensuremath{\left.#1\right|_{#2}}}

\def\Xint#1{\mathchoice
	{\XXint\displaystyle\textstyle{#1}}%
	{\XXint\textstyle\scriptstyle{#1}}%
	{\XXint\scriptstyle\scriptscriptstyle{#1}}%
	{\XXint\scriptscriptstyle\scriptscriptstyle{#1}}%
	\!\int}
\def\XXint#1#2#3{{\setbox0=\hbox{$#1{#2#3}{\int}$ }
		\vcenter{\hbox{$#2#3$ }}\kern-.585\wd0}}
\def\barint{\Xint-}

\newcommand{\bariint}{\barint\mkern-11.5mu\barint}

\begin{document}

	\author[M. Cao]{Mingming Cao}
	\address{Mingming Cao\\
		Instituto de Ciencias Matem\'aticas CSIC-UAM-UC3M-UCM\\
		Con\-se\-jo Superior de Investigaciones Cient{\'\i}ficas\\
		C/ Nicol\'as Cabrera, 13-15\\
		E-28049 Ma\-drid, Spain} \email{mingming.cao@icmat.es}

	\author[P. Hidalgo-Palencia]{Pablo Hidalgo-Palencia}
	\address{Pablo Hidalgo-Palencia \\
		Instituto de Ciencias Matem\'aticas CSIC-UAM-UC3M-UCM\\
		Consejo Superior de Investigaciones Cient{\'\i}ficas\\
		C/ Nicol\'as Cabrera, 13-15\\
		E-28049 Ma\-drid, Spain;
		and 
		Departamento de Análisis Matemático y Matemática Aplicada \\
		Universidad Complutense de Madrid \\
		Plaza de Ciencias 3 \\
		E-28040 Madrid, Spain} \email{pablo.hidalgo@icmat.es}

	\author[J.M. Martell]{Jos\'e Mar{\'\i}a Martell}
	\address{Jos\'e Mar{\'\i}a Martell\\
		Instituto de Ciencias Matem\'aticas CSIC-UAM-UC3M-UCM\\
		Consejo Superior de Investigaciones Cient{\'\i}ficas\\
		C/ Nicol\'as Cabrera, 13-15\\
		E-28049 Ma\-drid, Spain} \email{chema.martell@icmat.es}

	\author[C. Prisuelos-Arribas]{Cruz Prisuelos-Arribas}
	\address{Cruz Prisuelos-Arribas
		\\
		Universidad de Alcalá de Henares
		\\
		Departamento de Física y Matemáticas
		\\
		Campus universitario 
		\\E-28805 Alcalá de Henares (Madrid), Spain
	} \email{cruz.prisuelos@uah.es}

	\author[Z. Zhao]{Zihui Zhao}
	\address{Zihui Zhao
		\\
		Department of Mathematics
		\\
		Johns Hopkins University
		\\
		Baltimore, MD 21218, USA} 
	\email{zhaozh@jhu.edu}

	\thanks{The first and second authors are respectively supported by grants RYC2021-032600-I and CEX2019-000904-S-20-3, both funded by  MCIN/AEI/ 10.13039/501100011033. The first, second, and third authors acknowledge financial support from MCIN/AEI/ 10.13039/501100011033 grants CEX2019-000904-S and PID2019-107914GB-I00. The third author also acknowledges that the research leading to these results has received funding from the European Research Council under the European Union's Seventh Framework Programme (FP7/2007-2013)/ ERC agreement no. 615112 HAPDEGMT. The fourth author was partially supported by the Spanish Ministry of Science and Innovation grant MTM PID2019-107914GB-I00. The last author was partially supported by NSF grants DMS-1361823, DMS-1500098, DMS-1664867, DMS-1902756 and by the Institute for Advanced Study.
	 Part of this work was carried out while the last author was visiting the  Instituto de Ciencias Matem\'aticas (ICMAT), and all authors participated in a conference organized by the Universidad del Pa\'is Vasco (UPV/EHU).  The authors express their gratitude to these institutions. The last author is also grateful to Tatiana Toro for inspiring conversations about this work.	
} 
	\date{\today}

	\makeatletter
	\@namedef{subjclassname@2020}{\textup{2020} Mathematics Subject Classification}
	\makeatother
	\subjclass[2020]{35J25, 26A16, 35B65, 31B05, 31B25, 42B37, 42B35}


	\keywords{
		Well-posedness, 
		Dirichlet boundary value problem, 
		Uniform elliptic operators, 
		Elliptic measure, 
		Capacity density condition,  
		H\"{o}lder spaces, 
		Carleson measure estimates, 
		1-sided chord-arc domains}

	\begin{abstract}
		In this paper we study the Dirichlet problem for real-valued second order divergence form elliptic operators with boundary data in  H\"{o}lder spaces. Our context is that of open sets $\Omega \subset \mathbb{R}^{n+1}$, $n \ge 2$, satisfying the capacity density condition, without any further topological assumptions. Our main result states that if $\Omega$ is either bounded, or unbounded with unbounded boundary, then the corresponding Dirichlet boundary value problem is well-posed; when $\Omega$ is unbounded with bounded boundary, we establish that solutions exist, but they fail to be unique in general. These results are optimal in the sense that solvability of the Dirichlet problem in H\"{o}lder spaces is shown to imply the capacity density condition. 
		
		As a consequence of the main result, we present a characterization of the H\"{o}lder spaces in terms of the boundary traces of solutions, and obtain well-posedness of several related Dirichlet boundary value problems.
		
		All the results above are new even for 1-sided chord-arc domains, and can be extended to generalized H\"{o}lder spaces associated with a natural class of growth functions. 
	\end{abstract}

	\title[Elliptic operators in rough sets]
	{Elliptic operators in rough sets, and the Dirichlet problem with boundary data in H\"{o}lder spaces}

	\maketitle
	\tableofcontents

	\section{Introduction}\label{sec:intro} 
	The past 50 years have witnessed remarkable progress in the study of Dirichlet boundary value problems for elliptic operators. As already noted, the nature of these problems strongly depends on both the geometry of the domains and the choice of function spaces from which the boundary datum is selected. The geometrical features of the domains determine the construction of the solutions and their behavior near the boundary, while the choice of function spaces measures the size of the solutions and the way their boundary traces can be understood, and also determines if the problem is well-posed or not. 
	
	In the (classical) Dirichlet problem for the Laplacian in the upper half-space with a continuous boundary datum $f$, one seeks a harmonic function $u \in \mathscr{C}^{\infty}(\R^{n+1}_+) \cap \mathscr{C}(\overline{\R^{n+1}_+})$ satisfying $\restr{u}{\Rn}=f$. As noted in \cite[p.~42 and p.~158]{Helms}, the solution to this classical Dirichlet problem is not unique even if the boundary datum $f$ is a bounded continuous function in $\Rn$. Typically, the uniqueness can be ensured by specifying the behavior of $u(x, t)$ as $t \to \infty$. Indeed, the uniqueness was  established in \cite{ST} for harmonic functions $u \in \mathscr{C}^{\infty}(\R^{n+1}_+) \cap \mathscr{C}(\overline{\R^{n+1}_+})$ satisfying $u(x) = o(|x| \sec^{\gamma} \theta)$ as $|x| \to \infty$ (where $\theta:=\arccos(x_n/|x|)$ and $\gamma \in \R$ is arbitrary), by proving a Phragm\'{e}n-Lindel\"{o}f principle under the latter growth condition. For more related work in this regard see also \cite{Sie, Yos, YM}.

	Much attention has also been paid to the case of the $L^p$-Dirichlet problem in the upper half-space $\R^{n+1}_+$ for variable coefficient elliptic operators in divergence form $L=-\div(A\nabla)$ with $t$-independent coefficients $A(x, t)=A(x)$ for $(x, t) \in \R^{n+1}_+$. For a complex bounded coefficient matrix $A$,  the method of layer potentials was used to establish the solvability of the Dirichlet problem with $L^2(\Rn)$ data, for a small complex $L^{\infty}$ type perturbation of a real symmetric matrix in \cite{AAAHK}, which was extended to the case with $L^p(\Rn)$ data in \cite{HMiM}, and then to the case with $L^p(\Rn)$, $\BMO(\Rn)$, and $\dot{\mathscr{C}}^{\alpha}(\Rn)$ data, for a small complex Carleson type perturbation of a real symmetric matrix in \cite{HMaM}. Here, we mention that this kind of work crucially relies on the De Giorgi-Nash-Moser theory. On the other hand, the harmonic/elliptic measure is a powerful tool to deal with elliptic operators with real bounded coefficient matrices. Indeed, Hofmann, Kenig, Mayboroda,  and Pipher \cite{HKMP} proved that for each $t$-independent coefficients matrix $A$ in $\R^{n+1}_+$, there exists some $p \in (1, \infty)$ such that the Dirichlet problem with $L^p(\Rn)$-boundary data is solvable, which complemented the previous result by Kenig, Koch, Pipher, and Toro \cite{KKPT} in dimension $n=1$. Conversely, counterexamples in \cite{KKPT} showed that for each $p \in (1, \infty)$, there exists a uniformly elliptic coefficient matrix $A$ for which the Dirichlet problem is not solvable for $L^p(\Rn)$-boundary data. A significant portion of that theory in \cite{HKMP} was recently extended to the degenerate elliptic case by Hofmann, Le, and Morris \cite{HLM}, compared with the result due to Auscher, Ros\'{e}n, and Rule \cite{ARR} for $L^2(\Rn)$ data.

	As regards the $L^p$-Dirichlet problem in rough domains,  its solvability is closely linked with the absolute continuity properties of elliptic measures. Indeed, the importance of the quantitative absolute continuity of the elliptic measure with respect to the surface measure comes from the fact that $\w_L \in RH_p(\sigma)$ (the reverse H\"older class with respect to the surface measure $\sigma$) is equivalent to the $L^{p'}(\sigma)$-solvability of the Dirichlet problem (see, e.g. \cite[Theorem~1.5]{FKP} and \cite{HL}). In this direction, Dahlberg \cite{D77} first established its solvability for the Laplacian with $L^p$ data and $p \in [2, \infty)$ in Lipschitz domains. This result was extended to $\BMO_1$ domains by Jerison and Kenig \cite{JK}, to chord-arc domains by David and Jerison \cite{DJ}, and further studied in $1$-sided chord-arc domains by Hofmann et al. \cite{AHMNT, HM, HMU}. Under some optimal background hypothesis (an open set satisfying the corkscrew condition with an $n$-dimensional Ahlfors-David regular boundary), it was culminated in the recent results of \cite{AHMMT}, which gives a geometric characterization of the solvability of the $L^p$-Dirichlet problem for the Laplacian and for some finite $p$ via the rectifiability of the boundary plus a  local accessibility by non-tangential paths to some pieces of the boundary. It is worth pointing out that all the aforementioned results are restricted to the $n$-dimensional boundaries of domains in $\R^{n+1}$. Some analogues have been obtained in \cite{DFM, DM, MZ} on lower-dimensional sets. 
	
	Other types of boundary data have also been studied in the literature. For example, Fabes et al. \cite{FJN} dealt with some boundary value problems in the upper-half space with data in BMO and Morrey-Campanato spaces, which were extended to more general domains and operators in the works \cite{DK, FN, Zhao}. Shen \cite{Shen} also established the well-posedness for elliptic systems in Lipschitz domains with boundary data in Morrey spaces, when $p=2$. Recently, Martell et al. \cite{MMMM19} have taken advantage of the properties of the Poisson kernel associated with constant complex elliptic systems \cite{ADN59, ADN64} to obtain well-posedness of the Dirichlet boundary value problem in the upper-half space with data in BMO, VMO and Morrey-Campanato spaces. Some extensions are even possible to more general Banach spaces \cite{MMMM16} that include variable exponent Lebesgue spaces, Lorentz spaces, Zygmund spaces, and weighted versions of them in \cite{CMM, M+5}.

	The main goal of this paper is to study the well-posedness of the $\dot{\mathscr{C}}^{\alpha}$-Dirichlet boundary value problem
	\begin{equation}\label{eq:problem}
	\begin{cases}
	u \in W^{1, 2}_{\loc} (\Omega), \\
	Lu = 0 \text{ in the weak sense in } \Omega, \\
	u \in \dot{\mathscr{C}}^{\alpha}(\Omega), \\
	\restr{u}{\pom} = f \in \dot{\mathscr{C}}^{\alpha}(\pom). 
	\end{cases}
	\end{equation}
	Hereafter, given an arbitrary set $E \subset \ree$ and $\alpha \in (0,1)$, the homogeneous H\"{o}lder space of order $\alpha$ on $E$ is defined as 
	\begin{align*}
	\dot{\mathscr{C}}^{\alpha}(E)
	:= \big\{u: E \longrightarrow \re: \|u\|_{\dot{\mathscr{C}}^{\alpha}(E)}<\infty \big\} , 
	\end{align*}
	where $\|\cdot\|_{\dot{\mathscr{C}}^{\alpha}(E)}$ stands for the semi-norm 
	\begin{align*}
	\|u\|_{\dot{\mathscr{C}}^{\alpha}(E)} 
	:= \sup_{\substack{X,Y \in E \\ X \neq Y}} \frac{|u(X)-u(Y)|}{|X-Y|^{\alpha}}. 
	\end{align*}
	The operator $L$ is a second order uniformly elliptic operator in divergence form defined as $Lu=-\div(A\nabla u)$, with $A$ being a bounded real (non-necessarily symmetric) uniformly elliptic matrix: 
	\begin{align*}
	\Lambda^{-1} |\xi|^{2} \leq A(X) \xi \cdot \xi 
	\quad\text{ and } \quad
	|A(X) \xi \cdot\eta|\leq \Lambda |\xi|\,|\eta|,
	\end{align*}
	for all $\xi, \eta \in\mathbb{R}^{n+1}$, for almost every $X \in \Omega$, and for some $\Lambda \geq 1$.
	
	A particular case of \eqref{eq:problem} is $\Omega=\R^{n+1}_+$, which was treated in \cite[Theorem~1.4]{MMM} and \cite[Theorem~1.21]{MMMM19} for homogeneous second-order strongly elliptic systems of differential operators with constant complex coefficients. In the current paper, we would like to investigate what geometrical features of the domain will guarantee the well-posedness of \eqref{eq:problem}. The method used in \cite{MMM, MMMM19} is invalid in rough domains since one cannot ensure the existence and good properties of the associated Poisson kernel. Recently, the area has witnessed the availability and powerfulness of the method of harmonic/elliptic measure in studying the Dirichlet boundary value problem in rough domains (cf. \cite{AHMMT, AHMNT, CHMT, HLMN, HM, HMMTZ, HMU}).  Motivated by these, we are interested in finding solutions, given by the associated harmonic/elliptic measure, which also satisfy the natural bounds. We also consider the problem of uniqueness and investigate under what circumstances such elliptic measure solutions are indeed unique. 
	

	Our main results establish that provided $\alpha$ is sufficiently small (depending only on the allowable parameters) we can always find a solution to problem \eqref{eq:problem} expressed in terms of the boundary data and the associated elliptic measure. On the other hand, the uniqueness of such solution depends on the topological features of the domain and its boundary, more specifically, on their boundedness. We have two possibilities. In the first one, $\Omega$ is either bounded or it is unbounded with unbounded boundary, and the $\dot{\mathscr{C}}^{\alpha}$-Dirichlet problem is well-posed, that is, the associated elliptic measure solution is the unique solution. In the second scenario, $\Omega$ unbounded with bounded boundary, solutions cease to be unique in general and one can find different solutions to the same boundary value problem, written in terms of the elliptic measure. Our proof is very robust so that we can work in open sets without connectivity assumptions. More precisely, our first result establishes the following:

	\begin{theorem}\label{thm:Ca-well}
		Let $\Omega \subset \R^{n+1}$, $n \geq 2$, be an open set satisfying the capacity density condition (cf. Definition $\ref{def:CDC}$) so that either {\bf $\Omega$ is bounded or $\Omega$ is unbounded with $\pom$ being unbounded}. Let $L = -\div (A\nabla)$ be a real (non-necessarily symmetric) elliptic operator, and let $\omega_L$ be the associated elliptic measure. Then there exists $\alpha_0 \in (0, 1)$\footnote[2]{
			\label{note1}One can take $\alpha_0=\min\{\alpha_1,\alpha_2\}$, where $\alpha_1\in (0,1)$ is the exponent in the De Giorgi/Nash estimates (cf.~Lemma~\ref{lem:DGN}) and $\alpha_2\in (0,1)$ is the exponent in the H\"{o}lder continuity at the boundary (cf.~Lemma~\ref{lem:Holder}).
		} 
		(depending only on $n$, the CDC constant, and the ellipticity constant of $L$) such that for every $\alpha \in (0, \alpha_0)$, the $\dot{\mathscr{C}}^{\alpha}$-Dirichlet problem \eqref{eq:problem} is well-posed. More specifically, there is a unique solution which is given by 
		\begin{equation}\label{eq:well-solution}
		u(X) = \int_\pom f(y) \, d\omega_L^X(y), \qquad X \in \Omega.
		\end{equation}
		Moreover, $u \in W^{1, 2}_{\loc}(\Omega) \cap \dot{\mathscr{C}}^{\alpha}(\overline{\Omega})$ and satisfies $\restr{u}{\pom} = f$, and there is a constant $C_{\alpha}$ (depending only on $n$, the CDC constant, the ellipticity constant of $L$, and $\alpha$) such that 
		\begin{equation*}
		\|f\|_{\dot{\mathscr{C}}^{\alpha}(\pom)}
		\leq \|u\|_{\dot{\mathscr{C}}^{\alpha}(\Omega)}
		\leq C_{\alpha} \|f\|_{\dot{\mathscr{C}}^{\alpha}(\pom)}.
		\end{equation*}
	\end{theorem}
	\begin{remark}
		Compared to previous work on bounded domains (see, for example, \cite[Theorem 6.44]{HKM}), we remark that the constant $C_{\alpha}$ above (as well as in the next result) depends on the CDC constant of $\Omega$, but not on $\diam(\pom)$.
	\end{remark}
	
	Our second result concerns the case $\Omega$ being unbounded with $\pom$ bounded (e.g. $\Omega = \R^{n+1} \setminus B$, where $B$ is a ball in $\R^{n+1}$). In this scenario, solutions of \eqref{eq:problem} fail to be unique in general. 
	
	\begin{theorem}\label{thm:Ca-ill} 
		Let $\Omega \subset \R^{n+1}$, $n \geq 2$, be an open set satisfying the capacity density condition so that {\bf $\Omega$ is unbounded with $\pom$ being bounded}. Let $L = -\div (A\nabla)$ be a real (non-necessarily symmetric) elliptic operator, and let $\omega_L$ be the associated elliptic measure. Then there exists $\alpha_0 \in (0, 1)$\textsuperscript{\ref{note1}} (depending only on $n$, the  CDC constant, and the ellipticity constant of $L$) such that for every $\alpha \in (0, \alpha_0)$, for every $f \in \dot{\mathscr{C}}^{\alpha}(\pom)$, and for every $y_0 \in \pom$, if we define 
		\begin{equation} \label{eq:ill-solution}
		u(X) = f(y_0) + \int_{\pom} (f(y)-f(y_0)) \, d\w_L^X(y), \qquad X \in \Omega, 
		\end{equation}
		then $u \in W^{1,2}_{\loc}(\Omega) \cap \dot{\mathscr{C}}^{\alpha}(\overline{\Omega})$, satisfies $u|_{\pom}=f$ and solves the $\dot{\mathscr{C}}^{\alpha}$-Dirichlet problem for $L$ in $\Omega$ as formulated in \eqref{eq:problem}. Moreover, there exists $C_{\alpha}$ (depending on $n$, the CDC constant, the ellipticity constant of $L$, and $\alpha$) such that 
		\begin{equation}\label{eq:ill-fuf}
		\|f\|_{\dot{\mathscr{C}}^{\alpha}(\pom)}
		\leq \|u\|_{\dot{\mathscr{C}}^{\alpha}(\Omega)}
		\leq C_{\alpha} \|f\|_{\dot{\mathscr{C}}^{\alpha}(\pom)}.
		\end{equation}
		Nonetheless, the $\dot{\mathscr{C}}^{\alpha}$-Dirichlet problem for $L$ in $\Omega$ is ill-posed since \eqref{eq:problem} always has more than one solution if $f$ is non-constant. 
	\end{theorem}

	\begin{remark}
		We would like to call the reader's attention to the (apparently different) expressions \eqref{eq:well-solution} and \eqref{eq:ill-solution} for solutions to the $\dot{\mathscr{C}}^{\alpha}$-Dirichlet problem. We first observe that in the scenario of Theorem $\ref{thm:Ca-well}$ (once we know that $u$ as set in \eqref{eq:well-solution} is well-defined via an absolutely convergent integral, which is indeed part of the proof), one can use the fact that $\omega_L^X(\pom)=1$ for every $X\in\Omega$, to see that  \eqref{eq:well-solution} and  \eqref{eq:ill-solution} give rise to the same function $u$. On the other hand, in the scenario of Theorem~\ref{thm:Ca-ill} we have $0<\omega_L^X(\pom)<1$ for every $X\in\Omega$, hence \eqref{eq:well-solution} and \eqref{eq:ill-solution} are the same solution if and only if $f(y_0)=0$, so they are different in general. Moreover, if $u$ were chosen as in \eqref{eq:well-solution}, then the second estimate in \eqref{eq:ill-fuf} would fail. To see this, take $f\equiv 1$ on $\pom$ so that $f\in \dot{\mathscr{C}}^{\alpha}(\pom)$ with $\|f\|_{\dot{\mathscr{C}}^{\alpha}(\Omega)}=0$. If we set $u_1$ as in  \eqref{eq:well-solution}, then $u_1(X)=\omega_L^X(\pom)$ for every $X\in\Omega$. If the second estimate in \eqref{eq:ill-fuf} were true for $u_1$, then we would have $\|u_1\|_{\dot{\mathscr{C}}^{\alpha}(\Omega)}=0$, hence $u_1$ must be constant. But we know that $u_1\in\mathscr{C}(\overline{\Omega})$ with $u_1|_{\pom}=f\equiv 1$, hence $u_1(X)=\omega_L^X(\pom)=1$ for every $X\in \Omega$, and this  contradicts the fact that $\Omega$ is unbounded with $\pom$ bounded. Note on the other hand that if we define $u_2$ as in \eqref{eq:ill-solution}, then $u_2\equiv 1$ as desired and in that case $u_2$ clearly verifies \eqref{eq:ill-fuf}. 
	\end{remark}

	As an immediate consequence of Theorems \ref{thm:Ca-well} and \ref{thm:Ca-ill}, we have the following. 
	\begin{corollary}
	Let $\Omega \subset \ree$, $n\ge 2$, be an open set satisfying the capacity density condition and $L=-\div(A\nabla)$ be a real (non-necessarily symmetric) uniformly elliptic operator. Then, the $\dot{\mathscr{C}}^{\alpha}$-Dirichlet problem  for $L$ in $\Omega$ formulated in \eqref{eq:problem} is well-posed if and only if either \textbf{$\Omega$ is bounded} or \textbf{$\Omega$ is unbounded with $\pom$ being unbounded}.
	\end{corollary}

	It is important to emphasize that the results above are optimal in the sense that the capacity density condition is the minimal assumption needed to obtain well-posedness of the problem \eqref{eq:problem}. In \cite{A}, Aikawa proved that for the Laplacian in a bounded regular domain $\Omega$, solvability of the Dirichlet problem in H\"older spaces (along with the bound \eqref{eq:bounds_CDC}) already implies that $\Omega$ satisfies the capacity density condition.
	Our next result, which extends \cite{A} in several ways, is stated as follows: 
	
	\begin{theorem} \label{th:necessity_CDC}
		Let $\Omega \subset \ree$, $n \ge 2$, be an open set, and let $L = -\div (A\nabla)$ be a real (non-necessarily symmetric) elliptic operator. Let $\alpha \in (0, 1)$. Assume that for every $f \in \dot{\mathscr{C}}^{\alpha}(\pom)$, the function $u$ defined by \eqref{eq:well-solution} in the case $\Omega$ is bounded or $\Omega$ is unbounded with $\pom$ being unbounded, or defined by \eqref{eq:ill-solution} in the case $\Omega$ is unbounded with $\pom$ bounded, is a solution\footnote[3]{ Here and elsewhere, whenever we assume that \eqref{eq:well-solution} (or \eqref{eq:ill-solution}) is a solution to \eqref{eq:problem}, we also assume implicitly that the integral in \eqref{eq:well-solution} (or \eqref{eq:ill-solution}) is absolutely convergent.} 
		to \eqref{eq:problem} satisfying 
		\begin{equation} \label{eq:bounds_CDC}
		\|u\|_{\dot{\mathscr{C}}^{\alpha}(\Omega)}
		\leq 
		C_{\alpha} \|f\|_{\dot{\mathscr{C}}^{\alpha}(\pom)}. 
		\end{equation}
		Then $\Omega$ satisfies the capacity density condition, with constants depending only on $n$, the ellipticity of $L$, $\alpha$, and $C_\alpha$.
	\end{theorem}

	This result (along with the previous statements) is a quantitative analogue of the main results in \cite{LSW}. Indeed, we show that
	solving the Dirichlet problem with data in Hölder spaces relies on the CDC property which is of a geometrical nature (this should be compared with Wiener regularity for the continuous Dirichlet problem). Besides, we obtain that solving the Dirichlet problem with data in Hölder spaces for some given elliptic operator implies such solvability for any operator in this class (with a possibly different exponent depending on ellipticity).	
	{
	Moreover, we can prove that for any regular open set $\Omega$ (cf. Section~\ref{sec:optimal})  and for any elliptic operator $L$, the following holds: solutions defined by \eqref{eq:well-solution} or \eqref{eq:ill-solution} satisfy the estimate \eqref{eq:bounds_CDC}, if and only if, non-negative solutions with vanishing boundary data enjoy boundary H\"older continuity (see the precise statement in Lemma \ref{lem:Holder}), with \emph{almost} the same exponent $\alpha$ (see Remark~\ref{remark:equivalences}).
	}
	
	On the other hand, Theorem~\ref{thm:Ca-well} allows us to provide a characterization of the H\"{o}lder spaces in terms of the boundary traces of solutions. To discuss this, we need to introduce some notation. We first observe that $\|\cdot\|_{\dot{\mathscr{C}}^{\alpha}(\pom)}$ is only a semi-norm, and as a matter of fact, $\|f\|_{\dot{\mathscr{C}}^{\alpha}(\pom)}=0$ if and only if $f$ is constant everywhere in $\pom$. In view of this, given two functions $f,g$ defined on $\pom$, we say that $f\sim g$ provided $f-g$ is constant in $\pom$. This induces an equivalence relation and for every function $f: \pom \to \R$ we introduce the equivalence class of $f$ as $[f] :=\{g: \pom \to \R: g \thicksim f\}$, which ultimately defines the quotient space $\dot{\mathscr{C}}^{\alpha}(\pom)/\re := \{[f]: f \in \dot{\mathscr{C}}^{\alpha}(\pom)\}$. And defining
	$
	\big\|[f] \big\|_{\dot{\mathscr{C}}^{\alpha}(\pom)/\re}
	:=
	\|f\|_{\dot{\mathscr{C}}^{\alpha}(\pom)}$
	for $f\in\dot{\mathscr{C}}^{\alpha}(\pom)$, we get a genuine norm on this quotient space. It is not hard to show that $\dot{\mathscr{C}}^{\alpha}(\pom)/\re$ equipped with the norm $\big\|[\cdot]\big\|_{\dot{\mathscr{C}}^{\alpha}(\pom)}$ is complete, hence a Banach space. 
	
	Let us also introduce the space
	\[
	\dot{\mathscr{C}}_{L}^{\alpha}(\Omega)
	:=
	\big\{u\in W^{1,2}_{\loc}(\Omega)\cap \dot{\mathscr{C}}^{\alpha}(\Omega): Lu=0 \text{ in the weak sense in $\Omega$}
	\big\}.
	\]
	This is clearly a linear space on which $\|\cdot\|_{\dot{\mathscr{C}}^{\alpha}(\Omega)}$ is also a semi-norm with null-space $\re$. Much as before we can introduce
	$
	\big\|[u] \big\|_{\dot{\mathscr{C}}_L^{\alpha}(\Omega)/\re}
	:=
	\|u\|_{\dot{\mathscr{C}}^{\alpha}(\Omega)},
	$
	which defines a genuine norm on the quotient space $\dot{\mathscr{C}}^{\alpha}_L(\Omega)/\re$.
	
	\begin{corollary}\label{cor:fatou}
		Let $\Omega \subset \ree$, $n\ge 2$, be an open set satisfying the capacity density condition so that either \textbf{$\Omega$ is bounded} or \textbf{$\Omega$ is unbounded with $\pom$ being unbounded}. Let $L=-\div(A\nabla)$ be  a real (non-necessarily symmetric) elliptic operator. There exists $\alpha_0 \in (0, 1)$\textsuperscript{\ref{note1}} (depending only on $n$,  the CDC constant, and the ellipticity constants of $L$) such that for every $\alpha \in (0, \alpha_0)$, there exists $C_\alpha\ge 1$ (depending on the same parameters and additionally on $\alpha$) with the property that
		\begin{equation}\label{eqn:corol:Fatou}
		u \in \dot{\mathscr{C}}_L^{\alpha}(\Omega) 
		\Longrightarrow
		\left\{
		\begin{array}{l}
		u\in \dot{\mathscr{C}}^{\alpha}(\overline{\Omega}), \text{ hence } u|_{\pom}\in \dot{\mathscr{C}}^{\alpha}(\pom),
		\\[0.2cm]
		\displaystyle u(X)= \int_{\partial \Omega}\big(u|_{\pom}\big)(y)\,d\omega_L^X(y),\quad X\in\Omega,
		\\[0.4cm]
		\big\|u|_{\pom}\big\|_{\dot{\mathscr{C}}^{\alpha}(\pom)}
		\le 
		\|u\|_{\dot{\mathscr{C}}^{\alpha}(\Omega)}
		\le
		C_\alpha \big\|u|_{\pom}\big\|_{\dot{\mathscr{C}}^{\alpha}(\pom)}.
		\end{array}
		\right.
		\end{equation}
		In addition, 
		\begin{equation}\label{eqn:ident}
		\dot{\mathscr{C}}^{\alpha}(\pom)
		=
		\big\{
		u|_{\pom}: u\in  \dot{\mathscr{C}}_{L}^{\alpha}(\Omega)
		\big\}.
		\end{equation}
		Moreover, $\dot{\mathscr{C}}^{\alpha}_L(\Omega)/\re$ equipped with the norm $\big\|[\cdot] \big\|_{\dot{\mathscr{C}}^{\alpha}(\Omega)/\re}$ is a Banach space and the operator
		\begin{equation}\label{eqn:map}
		\dot{\mathscr{C}}^{\alpha}_L(\Omega)/\re\ni [u] \longmapsto \big[u|_{\pom}\big]\in \dot{\mathscr{C}}^{\alpha}(\pom)/\re
		\end{equation}
		is a well-defined isomorphism between Banach spaces. 
	\end{corollary}

	
	As applications of the above results, we will also establish well-posedness results for Dirichlet boundary value problems associated with different functional spaces, both for solutions and boundary data. These include Morrey-Campanato spaces, and spaces of solutions satisfying a fractional Carleson measure condition, which we will show to be closely related to the $\dot{\mathscr{C}}^\alpha$ spaces.

	
	The paper is organized as follows.
	In Section \ref{sec:pre}, we present some preliminaries, definitions, and auxiliary results that will be used throughout the paper. 
	The proofs of the main results, that is, those of well-posedness and ill-posedness of the $\dot{\mathscr{C}}^{\alpha}$-Dirichlet problem (Theorems~\ref{thm:Ca-well}, \ref{thm:Ca-ill}, and also Corollary~\ref{cor:fatou}), are presented in Sections \ref{sec:well} and \ref{sec:ill}, respectively. 
	Next, Section \ref{sec:optimal} is devoted to proving optimality of the capacity density condition (Theorem~\ref{th:necessity_CDC}). 
	Lastly, in Section~\ref{sec:applications} we explore some applications of the previous results to Morrey-Campanato spaces and spaces of functions satisfying a fractional Carleson measure condition.

	We mention that all the results above can be extended to generalized H\"{o}lder and Morrey-Campanato spaces under certain assumptions on the growth function (see Section~\ref{subsec:growth} for more details). Our proofs will be carried out in this general scenario below.


	\section{Preliminaries and auxiliary results}\label{sec:pre}
	
	\subsection{Notation}
	
	\begin{list}{$\bullet$}{\leftmargin=0.4cm  \itemsep=0.2cm}
		
		\item We use the letters $c$, $C$ to denote harmless positive constants, not necessarily the same at each occurrence, which depend only on dimension and the constants appearing in the hypotheses of the theorems (which we refer to as the ``allowable parameters''). We shall also sometimes write $a\lesssim b$ and $a\approx b$ to mean, respectively, that $a\leq C b$ and $0<c\leq a/b\leq C$, where the constants $c$ and $C$ are as above, unless explicitly noted to the contrary. 
		
		\item We shall use lower case letters $x,y,z$, etc., to denote points on $\partial\Omega$, and capital letters $X,Y,Z$, etc., to denote generic points in $\re^{n+1}$ (especially those in $\Omega$).
		
		
		\item The open $(n+1)$-dimensional Euclidean ball of radius $r$ will be denoted $B(x,r)$. A ``surface ball'' is denoted $\Delta(x,r):=B(x,r)\cap \partial\Omega$. Given a Euclidean ball $B=B(X,r)$, we will denote $r_B := r$, and its concentric dilate by a factor of $\kappa>0$ will be denoted by $\kappa B=B(X,\kappa r)$. 
		
		\item For $X\in\re^{n+1}$, we set $\delta(X):=\dist(X,\partial\Omega)$. 
		
		\item We denote by $\abs{\cdot}$ the ($n+1$)-dimensional Lebesgue measure in $\ree$, by $\H^n$ the $n$-dimensional Hausdorff measure, and by $\sigma:=\H^n |_{\partial \Omega}$ the surface measure on $\partial \Omega$. 
		
		\item For a Borel set $A\subset \partial\Omega$ with $0<\sigma(A)<\infty$, we write $\fint_{A}f\,d\sigma := \sigma(A)^{-1} \int_A f \,d\sigma$. Given a function $u \in L^1_{\loc}(\ree)$ and a Euclidean ball $B \subset \ree$ with $|B|>0$, we set $\bariint_B u \, dX := \frac{1}{|B|} \iint_B u \, dX$.
		
	\end{list}

	\subsection{Basic geometric framework}
	\begin{definition}[\bf Ahlfors-David  regular, ADR]\label{def:ADR}
		We say that a closed set $E \subset \ree$ is \textit{$n$-dimensional Ahlfors-David regular} (ADR for short) if there is some uniform constant $C_1>1$ such that
		\begin{equation*}
		C_1^{-1}\, r^n \leq \mathcal{H}^n(E\cap B(x,r)) \leq C_1\, r^n,\qquad\forall x\in E, \quad 0<r<\diam(E).
		\end{equation*}
	\end{definition}
	
	We recall the definition of the capacity of a set.  Given an open set $D\subset \ree$ (where we recall that we always assume that $n\ge 2$) and a compact set $K\subset D$  we define the capacity of $K$ relative to $D$ as
	\begin{align*}
	\Cap(K, D) := \inf\bigg\{\iint_{D} |\nabla v(X)|^2 dX:\, \, 
	v\in \mathscr{C}^{\infty}_{c}(D),\, v(x)\geq  1 \mbox{ in }K\bigg\}.
	\end{align*}
	
	\begin{definition}[\textbf{Capacity density condition, CDC}]\label{def:CDC}
		An open set $\Omega$ is said to satisfy the \textit{capacity density condition} (CDC for short) if there exists a uniform constant $c_1>0$ such that
		\begin{equation*}
		\frac{\Cap(\overline{B(x,r)}\setminus \Omega, B(x,2r))}{\Cap(\overline{B(x,r)}, B(x,2r))} \geq c_1
		\end{equation*}
		for all $x\in \partial\Omega$ and $0<r<\diam(\pom)$.
	\end{definition}

	The CDC is also known as the uniform 2-fatness as studied by Lewis in \cite{Lew}. Using \cite[Example 2.12]{HKM} one has that
	\begin{equation*}
	\Cap(\overline{B(x, r)}, B(x,2r))\approx r^{n-1}, \qquad \mbox{for all $x\in\ree$ and $r>0$},
	\end{equation*}
	and hence the CDC is a quantitative version of the Wiener regularity. In particular, if $\Omega$ satisfies the CDC, every $x\in\pom$ is Wiener regular. It is easy to see that the exterior Corkscrew condition (see Definition~\ref{def:cks}) implies the CDC. Also, it was proved in \cite[Section 3]{Zhao} and \cite[Lemma 3.27]{HLMN} that 
	\begin{equation}\label{eq:ADR-CDC}
	\begin{array}{c}
	\text{an open set with ADR boundary satisfies the capacity density condition} 
	\\ 
	\text{with constant $c_1$ depending only on $n$ and the ADR constant.}
	\end{array}
	\end{equation}

	\subsection{PDE framework} 
	Next, we recall several facts concerning the elliptic measure. To set the stage, let $\Omega\subset\re^{n+1}$ be an open set. Throughout we consider
	elliptic operators $L$ of the form $Lu=-\div(A\nabla u)$ with $A(\cdot)=(a_{i,j}(\cdot))_{i,j=1}^{n+1}$ being a real (non-necessarily symmetric) matrix such that $a_{i,j}\in L^{\infty}(\Omega)$ and there exists $\Lambda\geq 1$ such that the following uniform ellipticity condition holds
	\begin{align*}
	\Lambda^{-1} |\xi|^{2} \leq A(X) \xi \cdot \xi 
	\quad\text{ and }\quad
	|A(X) \xi \cdot\eta|\leq \Lambda |\xi|\,|\eta|
	\end{align*}
	for all $\xi,\eta \in\mathbb{R}^{n+1}$ and for almost every $X\in\Omega$. 

	We say that $u$ is a weak solution to $Lu=0$ in $\Omega$ provided that $u\in W_{\rm loc}^{1,2}(\Omega)$ satisfies
	\[
	\iint_\Omega A(X)\nabla u(X)\cdot \nabla\phi(X) dX=0  \quad\mbox{whenever}\,\, \phi\in \mathscr{C}^{\infty}_{c}(\Omega).
	\]
	Associated with $L$ one can construct an elliptic measure $\{\omega_L^X\}_{X\in\Omega}$ (see \cite{HMT} for full details). If $\Omega$ satisfies the CDC then it follows that all boundary points are Wiener regular and hence for a given $f\in \mathscr{C}_c(\partial\Omega)$ we can define
	\[
	u(X):=\int_{\partial\Omega} f(y) \, d\w^{X}_{L}(y), \quad \mbox{whenever}\, \, X \in \Omega,
	\]
	and $u:=f$ on $\pom$ and obtain that $u\in W^{1,2}_{\loc}(\Omega) \cap \mathscr{C}(\overline{\Omega})$ and  $Lu=0$ in the weak sense in $\Omega$. Moreover, if $f\in \Lip_c(\pom)$ then $u \in W^{1,2}(\Omega)$. When $\Omega$ is bounded or $\Omega$ is unbounded with $\pom$ unbounded, one has that $\w_L$ is a probability, that is, $\w_L^X(\pom)=1$ for every $X \in \Omega$. On the other hand, when $\Omega$ is unbounded with $\pom$ bounded (e.g. the complement of a ball), then $\w_L$ is not a probability, in fact $0<\w_L^X(\pom)<1$ for every $X \in \Omega$.

	We present two basic and useful estimates for solutions below, see for instance \cite{HMT}. 
	\begin{lemma}[De Giorgi/Nash]\label{lem:DGN} 
		Let $L=-\div(A\nabla)$ be a real (non-necessarily symmetric) uniformly elliptic operator. There exist $C_1 \ge 1$ and $\alpha_1 \in (0, 1)$ (depending only on dimension and on the ellipticity constant of $L$) such that for every ball $B=B(X_0, R) \subset \R^{n+1}$ and for every $u \in W^{1,2}(2B)$ so that $Lu=0$ in the weak sense in $2B$, one can redefine $u$ in a set of null measure (abusing the notation, the new function is called again $u$) in such a way that $u \in \dc^{\alpha_1}(\overline{B})$ (hence $u \in \mathscr{C}(\overline{B}) \cap L^{\infty}(B)$) satisfying 
		\begin{align*}
		\sup_{X \in \overline{B}} |u(X)| + \sup_{\substack{X, Y \in \overline{B} \\ X \neq Y}} \frac{|u(X)-u(Y)|}{(|X-Y|/R)^{\alpha_1}} 
		\le C_1 \bigg(\bariint_{2B} |u(Y)|^2 \, dY\bigg)^{\frac12}. 
		\end{align*}
	\end{lemma}

	\begin{lemma}[Boundary H\"{o}lder continuity]\label{lem:Holder} 
		Let $\Omega \subset \ree$, $n \ge 2$, be an open set satisfying the CDC, and let $L=-\div(A\nabla)$ be a real (non-necessarily symmetric) uniformly elliptic operator. There exist $C_2 \ge 1$ and $\alpha_2 \in (0, 1)$ (depending only on dimension, the  CDC constant, and on the ellipticity constant of $L$) such that for every ball $B=B(x, r)$ with $x \in \pom$, $0<r<\diam(\pom)$, and $\Delta=B \cap \pom$; and for every $0 \le u \in W^{1,2}_{\loc}(2B \cap \Omega) \cap \mathscr{C}(\overline{2B \cap \Omega})$ with $u \equiv 0$ in $2\Delta$ satisfying $Lu=0$ in the weak sense in $2B \cap \Omega$, there holds 
		\begin{align*}
		u(X) \le C_2 \bigg(\frac{|X-x|}{r}\bigg)^{\alpha_2} \sup_{\overline{B \cap \Omega}} u, \quad\forall X \in B \cap \Omega. 
		\end{align*}
	\end{lemma}
	
	
	\subsection{Growth functions and auxiliary lemmas} \label{subsec:growth}
	In the rest of the paper, we will be working with functions that generalize $\varphi(t)=t^{\alpha}$ with $\alpha \in (0, 1)$, hence generalizing the $\dot{\mathscr{C}}^\alpha$ spaces.
	\begin{definition}[{\bf Growth function and $\mathcal{G}_\beta$ class}]
		Call a given function $\varphi : (0, +\infty) \longrightarrow (0, + \infty)$ a \textit{growth function} if $\varphi$ is non-decreasing and $\varphi(t) \to 0$ as $t \to 0^+$. 
		Moreover, we say that the growth function $\varphi$ belongs to the class $\mathcal{G}_\beta$ for some $\beta \in (0, 1)$ if there exists some $C_\varphi$ such that 
		\begin{equation} \label{eq:ass-phi}
		\int_{0}^{t} \varphi(s) \frac{ds}{s} + t^{\beta} \int_t^\infty \frac{\varphi(s)}{s^{\beta}} \frac{ds}{s} 
		\leq C_\varphi \, \varphi(t), \qquad\forall t>0.
		\end{equation}
	\end{definition}
	
	\begin{definition}[{\bf Generalized Hölder spaces}]
		Let $\varphi$ be a growth function. Define the homogeneous $\varphi$-H\"{o}lder space on a set $E \subset \RR^{n+1}$ as 
			\begin{equation*}
			\dcv(E) 
			:=
			\bigg\{ u : E \to \RR : \|u\|_{\dot{\mathscr{C}}^\varphi(E)} := \sup_{\substack{X, Y \in E \\ X \neq Y}} \frac{|u(X) - u(Y)|}{\varphi(|X-Y|)}
			< \infty \bigg\}.
			\end{equation*}
	\end{definition}

	Observe that $\|\cdot\|_{\dcv(E)}$ is only a semi-norm, and as a matter of fact, $\|f\|_{\dcv(E)}=0$ if and only if $f$ is a constant everywhere in $E$. In view of this, given two functions $f, g$ defined in $E$, we say that $f \thicksim g$ provided $f-g$ is a constant in $E$. This induces an equivalence relation and for every function $f: E \to \R$ we introduce the equivalence class of $f$ as 
	$[f] :=\{g: E \to \R: g \thicksim f\}$.
	For any $f \in \dc(E)$ we can now define $\|[f]\|_{\dcv(E) / \R} := \|f\|_{\dcv(E)}$, which defines a genuine norm on the quotient space $\dcv(E)/ \R :=\{[f]: f \in \dcv(E) \}$. It is not hard to show that $\dcv(E)/\R$ equipped with the norm $\|\cdot\|_{\dcv(E)/\R}$ is complete, hence a Banach space.

	Let us denote, for $\alpha \in (0, 1)$ and $\varphi$ a growth function, 
	\begin{equation} \label{eq:Qalpha}
	Q_\alpha \varphi (t) 
	:=
	t^\alpha \int_t^\infty \frac{\varphi(s)}{s^\alpha} \frac{ds}{s}.
	\end{equation}
	Observe that $Q_\alpha \varphi$ is a well-defined function whenever $Q_\alpha \varphi(1)<\infty$, because then $Q_\alpha \varphi(t)<\infty$ for any $t>0$. 
	We would like to observe that $Q_\alpha$ is a generalization of the so-called Hardy operators, see \cite[Chapter 1]{BS}.

\begin{lemma}\label{lem:Qalpha}
Let $\varphi$ be a growth function and $\alpha\in (0,1)$ be such that $Q_\alpha \varphi(1)<\infty$. Then the following hold:
	\begin{list}{{\rm (\theenumi)}}{\usecounter{enumi}\leftmargin=1cm \labelwidth=1cm \itemsep=0.2cm \topsep=.2cm \renewcommand{\theenumi}{\alph{enumi}}}
	
	\item \label{item:Qalpha_growth} $Q_\alpha \varphi$ is a growth function.
	
	\item \label{item:Qalpha_limit} $t\longmapsto t^{-\alpha}\,Q_\alpha \varphi(t)$ is a decreasing function that decays to 0 as $t\to\infty$.

	\item \label{item:Qalpha_doubling} $Q_\alpha \varphi$ is doubling. Concretely, $Q_\alpha \varphi(2\,t)\le 2^\alpha Q_\alpha \varphi(t)$ for every $t>0$.

	\item \label{item:varphi_Qalpha} $\varphi(t) \le \alpha\, Q_\alpha\varphi(t) \leq Q_\alpha \varphi(t)$ for any $t > 0$.
	
	\item \label{item:sum_varphi_Qalpha} $\displaystyle t^\alpha\,\sum_{k=0}^\infty 	\frac{\varphi(2^{k}\,t)}{(2^{k}\,t)^{\alpha}}\lesssim_\alpha Q_\alpha\varphi(t)$ for any $t>0$.
	
	\item \label{item:Qalpha_decreasing} $Q_\alpha \varphi$ is decreasing in $\alpha$, that is, 
	$Q_{\alpha} \varphi
	\le
	Q_{\alpha'} \varphi$ whenever $0<\alpha'<\alpha<1$.
	\end{list}

\end{lemma}

\begin{proof}
	After changing variables, we have 
	\begin{equation} \label{eq:Qalpha_change_vars}
	Q_\alpha \varphi(t)
	=
	\int_1^\infty \frac{\varphi(t\tau)}{\tau^\alpha} \frac{d\tau}{\tau}.
	\end{equation}
	Therefore, since $\varphi$ is non-decreasing, $Q_\alpha(t)$ is also non-decreasing on $t$. Moreover, letting $t \to 0^+$ with the Dominated Convergence Theorem (because $Q_\alpha \varphi(1) < \infty$), and using that $\varphi(t) \to 0$ as $t \to 0^+$, it follows that $Q_\alpha \varphi(t) \to 0$ as $t \to 0^+$. This shows \eqref{item:Qalpha_growth}. To show \eqref{item:Qalpha_limit}, invoke the Dominated Convergence Theorem in \eqref{eq:Qalpha}, again using $Q_\alpha \varphi(1) < \infty$. In turn, property \eqref{item:Qalpha_doubling} follows from the fact that $\varphi$ is positive:
	\begin{equation*}
	Q_\alpha \varphi(2t) 
	=
	2^\alpha t^\alpha \int_{2t}^\infty \frac{\varphi(s)}{s^\alpha} \frac{ds}{s}
	\leq 
	2^\alpha t^\alpha \int_t^\infty \frac{\varphi(s)}{s^\alpha} \frac{ds}{s}
	=
	2^\alpha Q_\alpha \varphi(t).
	\end{equation*}
	
	Property \eqref{item:varphi_Qalpha} is an easy consequence of the monotonicity of $\varphi$:
	\begin{equation*}
	Q_\alpha \varphi(t) 
	\geq 
	t^\alpha \varphi(t) \int_t^\infty \frac{1}{s^\alpha} \frac{ds}{s}
	=
	\frac{\varphi(t)}{\alpha},
	\end{equation*}
	and we may also obtain \eqref{item:sum_varphi_Qalpha} as follows:
	\begin{equation*}
	\sum_{k=0}^\infty \frac{\varphi(2^{k}\,t)}{(2^{k}\,t)^{\alpha}}
	\approx 
	\sum_{k=0}^\infty \frac{\varphi(2^{k}\,t)}{(2^{k}\,t)^{\alpha}} \int_{2^kt}^{2^{k+1}t} \frac{ds}{s}
	\lesssim 
	\sum_{k=0}^\infty \int_{2^kt}^{2^{k+1}t} \frac{\varphi(s)}{s^\alpha} \frac{ds}{s}
	=
	t^{-\alpha} Q_\alpha(t).
	\end{equation*}
	Lastly, \eqref{item:Qalpha_decreasing} follows at once from \eqref{eq:Qalpha_change_vars}.
\end{proof}

\begin{lemma} \label{lem:extension}
	Let $\varphi$ be a growth function. Assume that $\varphi$ is doubling, i.e., there exists $C$ such that $\varphi(2t) \leq C \varphi(t)$ for every $t > 0$. Then, for each set $E \subset \RR^{n+1}$ and function $u \in \dot{\mathscr{C}}^\varphi(E)$, there exists an extension ---which we call again $u$--- to the closure of $E$, so that $u \in \dot{\mathscr{C}}^\varphi(\overline{E})$ and the following estimate holds:
	\begin{equation*}
	\|u\|_{\dot{\mathscr{C}}^\varphi(E)}
	\leq 
	\|u\|_{\dot{\mathscr{C}}^\varphi(\overline{E})}
	\leq 
	C \|u\|_{\dot{\mathscr{C}}^\varphi(E)}.
	\end{equation*}
\end{lemma}
\begin{proof}
	Note first that if $u \in \dot{\mathscr{C}}^\varphi(E)$, then $u$ is uniformly continuous in $E$. Hence, we can extend $u$ continuously to a function $v$ on $\overline{E}$. Let $y \neq z \in \overline{E}$ and pick $(y_k)_k, (z_k)_k \subset E$ with $y_k \to y$ and $z_k \to z$. Compute now
	\begin{multline*}
	|v(y) - v(z)| 
	=
	\lim_{k \to \infty} |u(y_k) - u(z_k)| 
	\leq 
	\|u\|_{\dot{\mathscr{C}}^\varphi(E)} \limsup_{k \to \infty} \varphi(|y_k - z_k|)
	\\ \leq 
	\|u\|_{\dot{\mathscr{C}}^\varphi(E)} \varphi(2|y - z|)
	\leq 
	C \|u\|_{\dot{\mathscr{C}}^\varphi(E)} \varphi(|y - z|)
	\end{multline*}
	using the doubling property of $\varphi$. This finishes the proof.
\end{proof} 

	The following result generalizes \cite[Lemma 2.1]{MMM}, and follows easily from Lemma~\ref{lem:Qalpha}:
	\begin{lemma}\label{lem:pro-phi}
		Let $\varphi$ be a growth function satisfying 
		\begin{equation} \label{eq:assumption}
		Q_\beta \varphi(t)=t^{\beta} \int_t^\infty \frac{\varphi(s)}{s^{\beta}} \frac{ds}{s} 
		\leq 
		C_\varphi \, \varphi(t), \quad \forall t > 0, \quad\text{ for some $\beta \in (0, 1)$}. 
		\end{equation}
		Then the following hold: 
		
		\begin{list}{{\rm (\theenumi)}}{\usecounter{enumi}\leftmargin=1cm \labelwidth=1cm \itemsep=0.2cm \topsep=.2cm \renewcommand{\theenumi}{\alph{enumi}}}
			\item\label{list-1} For every $0 < t_1 \leq t_2 < \infty$ we have 
			$
			t_2^{-\beta} \varphi(t_2)
			\leq
			\beta C_\varphi \, t_1^{-\beta} \varphi(t_1)
			$.
			
			\item\label{list-2} For every $t > 0$ we have 
			$
			\varphi(2t)
			\leq 2^{\beta} \beta C_\varphi \, \varphi(t)
			\leq 2 C_\varphi \, \varphi(t)
			$.
			
			\item\label{list-22} For every $t_1, t_2 > 0$ we have 
			$
			\varphi(t_1+t_2)
			\leq 2 C_\varphi \, (\varphi(t_1)+\varphi(t_2)) 
			$. 
			
			\item\label{list-3} $t^{-\beta}\varphi(t) \to 0$ as $t \to \infty$.
		\end{list}
	\end{lemma}
	
	\begin{proof}
 
		Using Lemma~\ref{lem:Qalpha} (concretely, \eqref{item:varphi_Qalpha}, \eqref{item:Qalpha_limit}, and \eqref{eq:assumption}), \eqref{list-1} follows.	
		From this, we get \eqref{list-2} just plugging $t_2 = 2t$ and $t_1=t$. Moreover, \eqref{list-2} implies that for all $t_1, t_2>0$, 
		\begin{align*}
		\varphi(t_1 + t_2)
		\le \varphi(2\max\{t_1, t_2\})
		\le 2C_{\varphi} \, \varphi(\max\{t_1, t_2\})
		\le 2C_{\varphi} \, (\varphi(t_1)+\varphi(t_2)), 
		\end{align*}
		which shows \eqref{list-22}. Lastly, \eqref{list-3} follows at once from Lemma~\ref{lem:Qalpha} (concretely, \eqref{item:varphi_Qalpha} and \eqref{item:Qalpha_limit}).
	\end{proof} 
	
	\begin{remark} \label{rem:growth_stronger}
		Let us note that whenever $\varphi$ satisfies \eqref{eq:assumption}, it follows that $Q_\beta \varphi(1) < \infty$ and $\varphi$ is doubling, hence the hypotheses of Lemmas~\ref{lem:Qalpha} and \ref{lem:extension} are also satisfied.
	\end{remark} 
	
	\begin{remark}\label{rem:alpha}
		Given $\alpha \in (0, 1)$, the function $\varphi(t) := t^{\alpha}$ satisfies the assumption \eqref{eq:assumption} for every $\beta>\alpha$, because $Q_\beta \varphi(t) \approx t^\alpha$. This is the particular case we are most interested in. 
	\end{remark}  
	
	
	The following lemma will be key in the forthcoming arguments.
	
	\begin{lemma}\label{lem:8r4r}
		Let $\Omega \subset \ree$, $n \ge 2$, be an open set satisfying the CDC, and let $L=-\div(A\nabla)$ be a real (non-necessarily symmetric) uniformly elliptic operator. Then there exists $C$ (depending only on dimension, the CDC constant, and on the ellipticity constant of $L$) such that, for any $x \in \pom$ and $r > 0$, 
		\begin{equation} \label{eq:elliptic_annuli}
		\omega_L^X (\pom \setminus \Delta(x, 4r))
		\leq C \bigg(\frac{|X - x|}{r} \bigg)^{\alpha_2}, \qquad X \in B(x, r) \cap \Omega,
		\end{equation}
		where $\alpha_2 \in (0, 1)$ is the exponent in Lemma~\ref{lem:Holder}. 		 
		Furthermore, for each growth function $\varphi$, there exists some $C$ (depending only on dimension, the CDC constant and the ellipticity constant of $L$) such that
		\begin{equation} \label{eq:domination_mod_continuity}
		\int_\pom \varphi(|y-x|) d\w_L^X(y)
		\leq 
		C Q_{\alpha_2} \varphi(\delta(X)) 
		, 
		\quad \forall X \in \Omega, x \in \pom \text{ such that } |X-x| \leq 3 \delta(X).
		\end{equation}	
		Therefore, if $\varphi$ satisfies \eqref{eq:assumption} for some $0 < \beta \leq \alpha_2$ there exists some $C$ (depending only on dimension, the CDC constant, the ellipticity constant of $L$, and $C_\varphi$) such that 
		\begin{equation}\label{fwaffcwr}
		\int_\pom \varphi(|y-x|) d\omega_L^X(y) 
		\leq C \varphi(\delta(X)), 
		\quad \forall X \in \Omega, x \in \pom \text{ such that } |X-x| \leq 3 \delta(X).
		\end{equation} 
	\end{lemma}
	\begin{remark}\label{rmk:bdrHolderdecay2}
		We remark that even though we assume that $\Omega$ satisfies the CDC, we only need the fact that $\Omega$ satisfies the boundary H\"older continuity property stated in Lemma \ref{lem:Holder} and regularity of the boundary (cf. Section~\ref{sec:optimal}).
	\end{remark}
	
	\begin{proof}
		
		Given $j \gg 1$, choose $\phi_j \in \mathscr{C}_c^\infty(\ree)$ such that $\phi_j \equiv 1$ in $B(x, 2^jr) \setminus B(x, 4r)$, $\phi_j \equiv 0$ in $B(x, 2r) \cup (\pom \setminus B(x, 2^{j+1}r))$, and $0 \leq \phi_j \leq 1$ everywhere. Hence
		\begin{equation*}
		\omega_L^X (\Delta(x, 2^jr) \setminus \Delta(x, 4r))
		\leq 
		\int_\pom \phi_j \,d\omega_L^X
		=:
		v_j(X),
		\qquad 
		X \in \Omega.
		\end{equation*}
		By Boundary H\"older continuity (cf. Lemma~\ref{lem:Holder}, which we can use because $\Omega$ satisfies the CDC, whence $\pom$ is Wiener regular), we have, for every $X\in B(x,r) \cap \Omega$,
		\begin{equation*} 
		v_j(X) \lesssim \left(\frac{|X-x| }{r} \right)^{\alpha_2} \sup_{\Omega} v_j \leq \left(\frac{|X-x| }{r} \right)^{\alpha_2}, 
		\end{equation*}
		because $0 \leq v_j \leq 1$ since $0 \leq \phi_j \leq 1$.
		Taking limits as $j \to \infty$ and using the Monotone Convergence Theorem, \eqref{eq:elliptic_annuli} follows.


		To proceed, let $\varphi$ be a growth function. Let $X \in \Omega$ and $x \in \pom$ be such that $|X-x| \leq 3 \delta(X)$. 
		Write $\Delta_k:=\Delta(x, 2^k\delta(X))$ for any $k \in \Z$. Then we may compute 
		\begin{multline*}
		\int_\pom \varphi(|y-x|) d\w_L^X(y) 
		= 
		\int_{\Delta_4} \varphi(|y-x|) d\w_L^X(y)
		+ \sum_{k=4}^{\infty} \int_{\Delta_{k+1} \setminus \Delta_k} \varphi(|y-x|) d\w_L^X(y)
		\\ 
		\lesssim \w^X_L(\Delta_4) \varphi(16\delta(X)) 
		+ \sum_{k=4}^{\infty} \varphi(2^{k+1}\delta(X)) \bigg(\frac{3\delta(X)}{2^{k-2}\delta(X)} \bigg)^{\alpha_2}
		\lesssim 
		Q_{\alpha_2} \varphi(\delta(X)),
		\end{multline*}
		where we have used that $\varphi$ is non-decreasing, $\omega_L^X(\Delta_4) \leq \omega_L^X(\pom) \leq 1$, and Lemma~\ref{lem:Qalpha}.		
		This shows \eqref{eq:domination_mod_continuity}. Lastly, \eqref{fwaffcwr} follows at once from \eqref{eq:domination_mod_continuity}, Lemma~\ref{lem:Qalpha}, $\beta \leq \alpha_2$, and \eqref{eq:assumption}. 
	\end{proof}


	\section{Well-posedness of the Dirichlet problem in Hölder spaces}\label{sec:well}
	In this section, we will prove Theorem \ref{thm:Ca-well}. We will in fact obtain a more general version (see Remark~\ref{rem:alpha}), using growth functions $\varphi \in \mathcal{G}_\beta$.

	\begin{theorem}\label{thm:well}
		Let $\Omega \subset \R^{n+1}$, $n \geq 2$, be an open set satisfying the CDC so that either {\bf $\Omega$ is bounded or $\Omega$ is unbounded with $\pom$ being unbounded}. Let $L = -\div (A\nabla)$ be a real (non-necessarily symmetric) elliptic operator, and let $\omega_L$ be the associated elliptic measure. Then there exists $\beta \in (0, 1)$\textsuperscript{\ref{note1}} (depending only on $n$, the CDC constant, and the ellipticity constant of $L$) such that for every growth function $\varphi \in \mathcal{G}_\beta$, the $\dcv$-Dirichlet problem
		\begin{equation}\label{eq:C-varphi}
		\begin{cases}
		u \in W^{1, 2}_{\loc} (\Omega), \\
		Lu = 0 \text{ in the weak sense in } \Omega, \\
		u \in \dot{\mathscr{C}}^\varphi(\Omega), \\
		\restr{u}{\pom} = f \in \dcv(\pom)
		\end{cases}
		\end{equation}
		is well-posed. More specifically, there is a unique solution which is given by 
		\begin{equation}\label{eq:solution}
		u(X) = \int_\pom f(y) \, d\omega_L^X(y), \qquad X \in \Omega.
		\end{equation}
		Moreover, $u \in W_{\loc}^{1, 2}(\Omega) \cap \dot{\mathscr{C}}^\varphi (\overline{\Omega})$ and satisfies $\restr{u}{\pom} = f$, and there is a constant $C$ (depending only on $n$, the CDC constant, the ellipticity constant of $L$, and $C_\varphi$) such that 
		\begin{equation}\label{eq:fuf}
		\|f\|_{\dot{\mathscr{C}}^\varphi(\pom)}
		\leq \|u\|_{\dot{\mathscr{C}}^\varphi(\Omega)}
		\leq C \|f\|_{\dot{\mathscr{C}}^\varphi(\pom)}.
		\end{equation}
	\end{theorem}
	
	
	\subsection{Proof of Theorem~\ref{thm:well}: Existence}\label{sect:existence}
	We first show that $u$ defined in \eqref{eq:solution} solves \eqref{eq:C-varphi}. If $f$ is a constant on $\pom$, then it follows from \eqref{eq:solution} and the fact that in the current case $\w_L$ is a probability that $u \equiv f(y_0)$ for any $y_0 \in \pom$. Then $u$ satisfies \eqref{eq:C-varphi} and \eqref{eq:fuf}. Thus, without loss of generality, by normalizing, we may assume that $\|f\|_{\dcv(\pom)} = 1$.

	\textbf{Step 1: $u$ is a well-defined solution.} 
	When $\pom$ is bounded, $f\in \dcv(\pom)$ implies that $f$ is continuous and bounded defined in the compact set $\pom$, hence the integral in \eqref{eq:solution} is certainly well-defined. When $\pom$ is unbounded, we need to consider smooth cut-offs to show that the integral in \eqref{eq:solution} is absolutely convergent. To that end, cover $\Omega$ with a sequence of balls of the form $\Omega = \bigcup_{i=0}^\infty B(X_i, \delta(X_i) / 4)$. We first work locally inside any of these Whitney balls, say $B(X_0, \delta(X_0)/4)$. Consider $\widehat{x}_0 \in \pom$ such that $\delta(X_0) = |X_0 - \widehat{x}_0|$.
	
	Choose now $\phi \in \mathscr{C}_c^{\infty}(\R)$ so that $\mathbf{1}_{[0, 1]}(t) \le \phi(t) \le \mathbf{1}_{[0, 2]}(t)$ for every $t \ge 0$. For any $j \in \Z$, if we write $\phi_j^{X_0}(X):=\phi(2^{-j}|X-\widehat{x}_0|)$, then $\phi_j^{X_0} \in \mathscr{C}_c^{\infty}(\ree)$ and $\mathbf{1}_{B(\widehat{x}_0, 2^j)}(X) \le \phi_j^{X_0}(X) \le \mathbf{1}_{B(\widehat{x}_0, 2^{j+1})}(X)$ for every $X \in \ree$. Observe that for each $j \in \Z$, 
	\[
	|\phi^{X_0}_j(y) - \phi^{X_0}_{j+1}(y)|  
	\le 
	\mathbf{1}_{\Delta(\widehat{x}_0, 2^{j+2}) \setminus \Delta(\widehat{x}_0, 2^j)}(y),
	\qquad 
	y \in \pom. 
	\]
	
	If we define 
	\[
	u^{X_0}_j(X) 
	:= 
	f(\widehat{x}_0) + \int_\pom (f(y)-f(\widehat{x}_0)) \phi_j^{X_0}(y) \, d\w_L^X(y) ,
	\qquad 
	X \in \Omega, 
	\]
	then for any $X \in B(X_0, \delta(X_0)/2)$ and $j \in \ZZ$ such that $2^j > 4 \delta(X_0)$, 
	\begin{align}\label{eq:ujuj}
	|u^{X_0}_j(X) - u^{X_0}_{j+1}(X)|
	& \leq 
	\int_\pom |f(y) - f(\widehat{x}_0)| \big|\phi^{X_0}_j(y) - \phi^{X_0}_{j+1}(y)\big| d\omega_L^X(y)
	\\\nonumber
	& \leq 
	\int_{\Delta(\widehat{x}_0, 2^{j+2}) \setminus \Delta(\widehat{x}_0, 2^j)} |f(y) - f(\widehat{x}_0)| d\omega_L^X(y)
	\\\nonumber
	& \le 
	\int_{\Delta(\widehat{x}_0, 2^{j+2}) \setminus \Delta(\widehat{x}_0, 2^j)} \varphi(|y-\widehat{x}_0|) \, d\omega_L^X(y)
	\\\nonumber
	& \leq 
	\varphi(2^{j+2}) \, \w_L^X (\Delta(\widehat{x}_0, 2^{j+2}) \setminus \Delta(\widehat{x}_0, 2^j))
	\\\nonumber
	& \approx 
	\varphi(2^{j+2}) \, \w_L^{X_0} (\Delta(\widehat{x}_0, 2^{j+2}) \setminus \Delta(\widehat{x}_0, 2^j))
	\\\nonumber
	& \lesssim 
	\varphi(2^{j+2}) \bigg(\frac{|X_0-\widehat{x}_0|}{2^{j-1}} \bigg)^{\alpha_2}
	\\\nonumber 
	& \lesssim \varphi(2^{j+2}) \bigg(\frac{\delta(X_0)}{2^{j+2}}\bigg)^{\alpha_2}, 
	\end{align}
	where we used that $\varphi$ is non-decreasing, Harnack's inequality (since $X \in B(X_0, \delta(X_0)/2)$), and Lemma~\ref{lem:8r4r}. 
	Telescoping for $k' > k$ with $2^k>4\delta(X_0)$, we obtain from \eqref{eq:ujuj} and Lemma~\ref{lem:Qalpha}  
	\eqref{item:sum_varphi_Qalpha}, \eqref{item:Qalpha_decreasing} and \eqref{item:Qalpha_limit}, assuming that $\beta \leq \alpha_2$: 
	\begin{multline} \label{eq:X_0_telescope}
	\big|u^{X_0}_k(X) - u^{X_0}_{k'}(X) \big|
	\leq 
	\sum_{j=k}^{k'-1} \big|u^{X_0}_j(X)-u^{X_0}_{j+1}(X)\big|
	\lesssim 
	\sum_{j=k}^\infty 
	\varphi(2^{j+2}) \bigg(\frac{\delta(X_0)}{2^{j+2}} \bigg)^{\alpha_2}
	\\
	=
	\delta(X_0)^{\alpha_2} \sum_{j=0}^\infty \frac{\varphi(2^{j} 2^{k+2})}{(2^j 2^{k+2})^{\alpha_2}} 
	\lesssim
	\delta(X_0)^{\alpha_2} (2^{k+2})^{-\alpha_2} Q_{\alpha_2} \varphi(2^{k+2})
	\underset{k \to \infty}{\longrightarrow}
	0,
	\end{multline} 
	hence $\{u^{X_0}_j\}_j$ is a Cauchy sequence uniformly in $B(X_0, \delta(X_0)/2)$. Since $u^{X_0}_j \in \mathscr{C}(B(X_0, \delta(X_0) / 2))$, we deduce that there exists $u^{X_0}_{\infty} \in \mathscr{C}(B(X_0, \delta(X_0)/2))$ such that $u^{X_0}_j \to u^{X_0}_\infty$ uniformly in $B(X_0, \delta(X_0) / 2)$ as $j \to \infty$. On the other hand, Caccioppoli's inequality implies 
	\begin{multline*}
	\iint_{B(X_0, \delta(X_0) / 4)} \big|\nabla u^{X_0}_j - \nabla u^{X_0}_{j'} \big|^2 dX 
	\lesssim 
	(\delta(X_0))^{-2} \iint_{B(X_0, \delta(X_0) / 2)} \big|u^{X_0}_j - v^{X_0}_{j'}\big|^2 dX 
	\\ \lesssim 
	(\delta(X_0))^{n-1} \sup_{B(X_0, \delta(X_0) / 2)} \big|u^{X_0}_j-u^{X_0}_{j'}\big|^2,
	\end{multline*}
	thus $\{\nabla u^{X_0}_j \}_j$ is a Cauchy sequence in $L^2(B(X_0, \delta(X_0) / 4))$. This gives easily that $u^{X_0}_\infty \in W^{1, 2}(B(X_0, \delta(X_0)/4))$ and $u^{X_0}_j \to u^{X_0}_j$ in $W^{1, 2}(B(X_0, \delta(X_0)/4))$ as $j \to \infty$.
	
	We now turn to prove that $u^{X_0}_\infty$ coincides with $u$ (as defined in \eqref{eq:solution}) for $X \in B(X_0, \delta(X_0)/4)$. Recall that 
	\begin{equation*}
	u(X)
	=
	\int_\pom f(y) \,d\omega_L^X(y)
	=
	f(\widehat{x}_0) + \int_\pom (f(y) - f(\widehat{x}_0)) \,d\omega_L^X(y)
	\end{equation*}
	because the elliptic measure is a probability in this case, and the integral in the previous equation is well defined since it is absolutely convergent:
	\begin{align*}
	\int_\pom |f(y) - f(\widehat{x}_0)| \,d\omega^X_L(y) 
	\leq \int_\pom \varphi(|y - \widehat{x}_0|) \,d\omega^{X}_L(y)
	\lesssim 
	\varphi(\delta(X))
	< \infty
	\end{align*}
	by Lemma~\ref{lem:8r4r}, where we also use the fact that $|X-\widehat{x}_0| \leq \frac54 \delta(X_0) \leq \frac53 \delta(X) $. In particular, for any $X \in B(X_0, \delta(X_0)/4)$ and $j \in \Z$ with $2^j > 5\delta(X_0)$ we invoke Lemma \ref{lem:8r4r}, and Harnack's inequality later, to conclude that  
	\begin{align*}
	\big|u(X) - u^{X_0}_j(X)\big|
	&\leq 
	\int_{\pom} |f(y) - f(\widehat{x}_0)| \big|1-\phi^{X_0}_j(y)\big| \, d\w_L^X(y) 
	\\ \nonumber & \le 
	\int_{\pom \setminus \Delta(\widehat{x}_0, 2^j)} |f(y) - f(\widehat{x}_0)| \, d\w_L^X(y) 
	\\ \nonumber & \leq
	\sum_{k=j}^{\infty} \varphi(2^{k+1}) \w_L^X (\Delta(\widehat{x}_0, 2^{k+1}) \setminus \Delta(\widehat{x}_0, 2^k)) 
	\\ \nonumber & \approx 
	\sum_{k=j}^{\infty} \varphi(2^{k+1}) \w_L^{X_0} (\Delta(\widehat{x}_0, 2^{k+1}) \setminus \Delta(\widehat{x}_0, 2^k)) 
	\\ \nonumber
	&\lesssim 
	\sum_{k=j}^{\infty} \varphi(2^{k+1}) \left( \frac{\delta(X_0)}{2^{k-2}} \right)^{\alpha_2},
	\end{align*}	
	which converges to 0 as $j \to \infty$ by Lemma~\ref{lem:Qalpha} (the computation is analogous to that in \eqref{eq:X_0_telescope}), provided $\beta \leq \alpha_2$. Therefore, we conclude that $u(X) = u^{X_0}_\infty(X)$ for all $X \in B(X_0, \delta(X_0)/4)$. 
	
	We can proceed in a similar fashion in $B(X_i, \delta(X_i) / 4)$ for each $i \in \NN$ to see that there exists some $u^{X_i}_\infty \in \mathscr{C}(B(X_i, \delta(X_i) / 4)) \cap W^{1, 2}(B(X_i, \delta(X_i) / 4))$ that coincides with $u$ in $B(X_i, \delta(X_i) / 4)$. This yields easily that $u \in \mathscr{C}(\Omega) \cap W^{1, 2}_{\loc}(\Omega)$  since $\Omega = \bigcup_{i=0}^\infty B(X_i, \delta(X_i) / 4)$.
	
	Now we are going to show that $u$ is a weak solution. For that purpose, pick $\psi \in \mathscr{C}_c^\infty(\Omega)$ and decompose $\psi = \sum_{i = 0}^\infty \psi_i$ for some $\psi_i \in \mathscr{C}_c^\infty (\Omega)$ with $\supp \psi_i \subset B(X_i, \delta(X_i) / 4)$ (this can be easily done with a partition of unity subordinated to the open cover of $\Omega$ we have been using). In fact, since $\psi$ is compactly supported, actually we only need finitely many terms and will say $\psi = \sum_{i = 0}^N \psi_i$ for some large enough $N$, so that $\supp \psi \subset \bigcup_{i=0}^N B(X_i, \delta(X_i)/4)$, to highlight that there are no convergence issues. We can then compute
	\begin{multline*}
	\iint_\Omega A \nabla u \cdot \nabla \psi  
	= 
	\sum_{i = 0}^N \iint_\Omega A \nabla u \cdot \nabla \psi_i
	=
	\sum_{i = 0}^N \iint_{B(X_i, \delta(X_i)/4)} A \nabla u \cdot \nabla \psi_i
	\\ =
	\sum_{i = 0}^N \iint_{B(X_i, \delta(X_i)/4)} A \nabla u^{X_i}_\infty \cdot \nabla \psi_i
	=
	\sum_{i = 0}^N \lim_{j \to \infty} \iint_{B(X_i, \delta(X_i)/4)} A \nabla u^{X_i}_j \cdot \nabla \psi_i
	\\ =
	\sum_{i = 0}^N \lim_{j \to \infty} \iint_\Omega A \nabla u^{X_i}_j \cdot \nabla \psi_i
	= 
	0
	\end{multline*}
	because $Lu^{X_i}_j = 0$ in the weak sense for any $i, j \in \NN$. With all these, it is now clear that $u$ is indeed a well defined weak solution in $W^{1, 2}_\loc(\Omega) \cap \mathscr{C}(\Omega)$. 
	
	\textbf{Step 2: $u \in \dcv(\Omega)$.} 
	Let $X, Y \in \Omega$. Without loss of generality we may assume that $\delta(X) \leq \delta(Y)$. Denote by $\widehat{x} \in \pom$ a point for which $|X-\widehat{x}| = \delta(X)$, and $\widehat{y} \in \pom$ satisfying $|Y-\widehat{y}| = \delta(Y)$. Let us distinguish some cases.
	
	\textbf{Case 1:} $|X-Y| < \delta(Y)/4$. In this case $X \in B(Y, \delta(Y)/4)$, and De Giorgi/Nash (Lemma~\ref{lem:DGN}) estimate for $u-f(\widehat{y})$ yields 
	\begin{align} \label{eq:uX-uY-case1}
	|u(X) - u(Y)|
	& = |(u(X) - f(\widehat{y})) - (u(Y) - f(\widehat{y}))|
	\\ 
	& \lesssim \bigg(\frac{|X-Y|}{\delta(Y)}\bigg)^{\alpha_1} 
	\bigg(\bariint_{B(Y, \delta(Y) / 2)} |u(Z) - f(\widehat{y})|^2 dZ \bigg)^{\frac12}
	\nonumber
	\\ 
	& \leq 	
	\bigg(\frac{|X-Y|}{\delta(Y)} \bigg)^{\alpha_1} 
	\bigg(\bariint_{B(Y, \delta(Y) / 2)} \bigg(\int_\pom |f(z) - f(\widehat{y})| d\w_L^Z(z) \bigg)^2 dZ \bigg)^{\frac12}
	\nonumber
	\\ 
	& \lesssim \bigg( \frac{|X-Y|}{\delta(Y)} \bigg)^{\alpha_1} 
	\bigg(\bariint_{B(Y, \delta(Y) / 2)} \bigg(\int_\pom \varphi(|z - \widehat{y}|) \, d\w_L^Z(z) \bigg)^2 dZ \bigg)^{\frac12}.
	\nonumber
	\end{align}
	Now, for $Z \in B(Y, \delta(Y) / 2)$, we clearly have $\delta(Z) \leq 2\delta(Y)$, hence we can use Lemma~\ref{lem:8r4r} and Lemma~\ref{lem:Qalpha}, to obtain
	\begin{equation*}
	\int_\pom \varphi(|z - \widehat{y}|) \, d\w_L^Z(z)
	\lesssim 
	Q_{\alpha_2} \varphi(\delta(Z))
	\leq 
	Q_{\alpha_2} \varphi(2\delta(Y))
	\lesssim 
	Q_{\alpha_2} \varphi(\delta(Y)).
	\end{equation*}
	This, back in \eqref{eq:uX-uY-case1}, and setting $\alpha_0 := \min \{\alpha_1, \alpha_2\}$,
	\begin{equation*}
	|u(X) - u(Y)|
	\lesssim 
	\bigg(\frac{|X-Y|}{\delta(Y)} \bigg)^{\alpha_1} Q_{\alpha_2} \varphi(\delta(Y))
	\leq 
	\bigg(\frac{|X-Y|}{\delta(Y)} \bigg)^{\alpha_0} Q_{\alpha_0} \varphi(\delta(Y))
	\leq 
	 Q_{\alpha_0} \varphi(|X-Y|),
	\end{equation*}
	where we have used Lemma~\ref{lem:Qalpha} repeatedly.
%
	Therefore, assuming that $\beta \leq \alpha_0$ and that $\varphi$ satisfies \eqref{eq:ass-phi}, we obtain 
	\begin{equation*}
	|u(X) - u(Y)|
	\lesssim 
	\varphi(|X-Y|).
	\end{equation*} 	
	
	\textbf{Case 2:} $\delta(Y) \leq 4 |X-Y|$ and $|\widehat{x} - \widehat{y}| \leq 10 |X-Y|$. 
	In this situation,  
	using Lemmas~\ref{lem:8r4r} and \ref{lem:Qalpha}, it follows that
	\begin{multline*}
	|u(X) - f(\widehat{x})|
	\leq \int_\pom |f(z) - f(\widehat{x})| d\omega_L^X(z)
	\leq \int_\pom \varphi(|z-\widehat{x}|) d\omega_L^X(z) 
	\\ \lesssim 
	Q_{\alpha_2} \varphi (\delta(X)) 
	\leq 
	Q_{\alpha_2} \varphi (4|X-Y|)
	\lesssim 
	Q_{\alpha_2} \varphi(|X-Y|).
	\end{multline*}
	A similar computation shows 
	\begin{equation*}
	|u(Y) - f(\widehat{y})| \lesssim Q_{\alpha_2} \varphi(|X-Y|).
	\end{equation*} 
	Then the triangle inequality and Lemma~\ref{lem:Qalpha} yield
	\begin{multline*}
	|u(X) - u(Y)|
	\leq |u(X) - f(\widehat{x})| + |f(\widehat{x}) - f(\widehat{y})| + |f(\widehat{y}) - u(Y)|
	\\ \lesssim 
	Q_{\alpha_2} \varphi(|X-Y|) + \varphi(10|X-Y|)
	\lesssim 
	Q_{\alpha_2} \varphi(|X-Y|)
	.
	\end{multline*}	
	Therefore, assuming that $\beta \leq \alpha_2$ and by \eqref{eq:ass-phi}, we obtain
	\begin{equation*}
	|u(X) - f(\widehat{x})|
	\lesssim \varphi(|X-Y|).
	\end{equation*} 
	
	\textbf{Case 3:} $\delta(Y) \leq 4|X-Y| $ and $10 |X-Y| < |\widehat{x} - \widehat{y}|$. 
	In this scenario, we have 
	\begin{multline*}
	10 |X-Y|
	<
	|\widehat{x} - \widehat{y}|
	\leq 
	|\widehat{x} - X| + |X-Y| + |Y-\widehat{y}|
	\\ =
	\delta(X) + |X-Y| + \delta(Y)
	\leq 
	2 \delta(Y) + |X-Y|
	\leq 
	9 |X-Y|,
	\end{multline*} 
	which is a contradiction, hence this case cannot occur.

	In either of the situations we have shown that $|u(X) - u(Y)| \lesssim \varphi(|X-Y|)$ with constants independent of the chosen points $X$ and $Y$, hence $u \in \dcv(\Omega)$  
	provided $\beta \leq \alpha_0 = \min \{\alpha_1, \alpha_2\}$. 
	
	\medskip 
	\textbf{Step 3: $u \in \dot{\mathscr{C}}^{\varphi}(\overline{\Omega})$, $\restr{u}{\pom} = f$, and \eqref{eq:fuf} holds.}  
	By Lemma \ref{lem:extension} (along with Remark~\ref{rem:growth_stronger}), 
	$u$ can be extended to $\overline{\Omega}$ (abusing the notation we denote such an extension by $u$), satisfying $u \in \dot{\mathscr{C}}^{\varphi}(\overline{\Omega})$ and \eqref{eq:fuf}.
	We only need to check that $\restr{u}{\pom} = f$.
	
	Let $X \in \Omega$ and $y \in \pom$. Pick $\widehat{x} \in \pom$ such that $|X-\widehat{x}| = \delta(X)$. Then it follows from Lemma \ref{lem:8r4r} 
	(concretely from \eqref{eq:domination_mod_continuity}) that 
	\begin{multline*}
	|u(X) - f(y)|
	\leq |u(X) - f(\widehat{x})| + |f(\widehat{x}) - f(y)| 
	\leq \int_\pom |f(z) - f(\widehat{x})| d\omega_L^X(z) + |f(\widehat{x}) - f(y)|
	\\ 
	\leq \int_\pom \varphi(|z - \widehat{x}|) d\omega_L^X(z) + \varphi(|\widehat{x} - y|)
	\lesssim 
	Q_{\alpha_2} \varphi(\delta(X)) + \varphi(|\widehat{x} - y|),
	\end{multline*} 
	which converges to 0 if we let $X$ approach $y$, recalling that $\varphi(t) \to 0$ as $t \to 0^+$ (for the first term, see Lemma~\ref{lem:Qalpha}, which states that $Q_{\alpha_2}\varphi$ is a growth function).  
	With this we are finally done, showing the existence of solutions in Theorem~\ref{thm:well}.
	 
	\begin{remark} \label{rem:different_spaces_existence}
		The reader may review the previous proof (along with Lemma~\ref{lem:Qalpha}) to observe that we have actually shown that, given $f \in \dot{\mathscr{C}}^\varphi(\pom)$, the associated solution $u$ constructed by \eqref{eq:solution} satisfies $u \in \dot{\mathscr{C}}^{Q_{\alpha_0}\varphi}(\overline{\Omega})$ with $\alpha_0 = \min \{\alpha_1, \alpha_2\}$, provided that $\varphi$ is merely a growth function for which $Q_{\alpha_0}\varphi(1) < \infty$ (this always happens whenever $\varphi \in \mathcal{G}_\beta$, see Remark~\ref{rem:growth_stronger}). 
%
	\end{remark} 
	
	
	\subsection{Proof of Theorem~\ref{thm:well}: Uniqueness}\label{sect:unique}
	Given $f \in \dcv(\pom)$, suppose that $u_1$ and $u_2$ are two solutions of \eqref{eq:C-varphi}. Write $u:=u_1-u_2$, which solves \eqref{eq:C-varphi} with $f \equiv 0$. We are going to show that $u \equiv 0$. When $\Omega$ is bounded, $u \equiv 0$ follows from the fact that $u \in \dcv(\overline{\Omega}) \subset \mathscr{C}(\overline{\Omega})$ and the classical maximum principle. 
	
	We next consider the case $\Omega$ unbounded with $\pom$ unbounded. If $\|u\|_{\dcv(\Omega)} = 0$, then $u$ is a constant and this constant must be zero since $u \in \mathscr{C}(\overline{\Omega})$ with $\restr{u}{\pom}=0$. Therefore, by homogeneity we may assume that $\|u\|_{\dcv(\Omega)} = 1$ and $\restr{u}{\pom} = 0$. 
	
	We present a fundamental estimate which will be used repeatedly afterwards. Given $\tau \in (0, 1)$ and an arbitrary point $Y \in \Omega$, we write $B_Y^{\tau} := B(Y, \tau \delta(Y))$ and compute 
	\begin{multline}\label{eq:average_ball_half}
	\bariint_{B_Y^{\tau}} |u - u_{B_Y^{\tau}}|
	\leq 
	\bariint_{B_Y^{\tau}} \bariint_{B_Y^{\tau}} |u(X) - u(Z)| dZ dX 
	\leq 
	\bariint_{B_Y^{\tau}} \bariint_{B_Y^{\tau}} \varphi(|X-Z|) dZ dX
	\\ \leq 
	\bariint_{B_Y^{\tau}} \bariint_{B_Y^{\tau}} \varphi(2\tau \delta(Y)) dZ dX
	=
	\varphi(2\tau \delta(Y)) 
	\leq
	\varphi(2\delta(Y)).
	\end{multline}

	Now fix $X \in \Omega$ and let $\widehat{x} \in \pom$ be such that $|X - \widehat{x}| = \delta(X)$. Construct a sequence of points $\{X_j\}_j \subset \Omega$ inductively: $X_0 = X$ and $X_{j+1} = (X_j + \widehat{x}) / 2$ for $j \geq 0$. Write $B_j := B(X_j, \delta(X_j)/2)$ for every $j \ge 0$. We have  
	\begin{equation}\label{eq:uuB}
	|u(X)| \leq \bigg|u(X) - \bariint_{B_0} u(Y)\, dY\bigg| + \bigg|\bariint_{B_0} u(Y) \, dY\bigg|.
	\end{equation}
	Considering the first term, in a similar fashion as in \eqref{eq:average_ball_half} we compute
	\begin{align} \label{eq:estimate_u_minus_B_0}
	\bigg|u(X) - \bariint_{B_0} u(Y) dY\bigg|
	\leq 
	\bariint_{B_0} |u(X) - u(Y)| dY 
	\leq 
	\bariint_{B_0} \varphi(\delta(X)) dY
	=
	\varphi(\delta(X)).
	\end{align}
	On the other hand, for each $N \geq 0$, 
	\begin{equation}\label{eq:u-sum}
	\bigg|\bariint_{B_0} u(Y) \, dY\bigg|
	\leq \bariint_{B_0} |u(Y) - u_{B_0}| \, dY + \sum_{j = 0}^N |u_{B_j} - u_{B_{j+1}}| + |u_{B_{N+1}}|.
	\end{equation}
	And by definition of our sequence of balls $B_j$, we have that $B_{N+1} \subset B \Big(\widehat{x}, \dfrac32 \dfrac{\delta(X)}{2^{N+1}}\Big)$ since $\delta(X_{N+1}) = \dfrac{\delta(X)}{2^{N+1}}$. This, jointly with the fact that $u \in \mathscr{C}(\overline{\Omega})$ and $\restr{u}{\pom} \equiv 0$, yields
	\begin{equation*}
	u_{B_{N+1}} = \bariint_{B_{N+1}} u(Y) \, dY \longrightarrow 0, \quad\text{ as } N \to \infty.
	\end{equation*}
	This allows us to take limits in \eqref{eq:u-sum} to achieve
	\begin{equation} \label{eq:telescopic}
	\bigg|\bariint_{B_0} u(Y) \, dY\bigg|
	\leq 
	\bariint_{B_0} |u(Y) - u_{B_0}| \, dY + \sum_{j = 0}^\infty |u_{B_j} - u_{B_{j+1}}|.
	\end{equation}
	If we define the fattened balls $\widetilde{B}_j := B(X_j, \frac34 \delta(X_j))$, one can easily see that $B_j \cup B_{j+1} \subset \widetilde{B}_j$ for any $j \geq 0$. Invoking \eqref{eq:average_ball_half} with $\tau=\frac34$, we can estimate each one of the previous summands: 
	\begin{multline*}
	|u_{B_j} - u_{B_{j+1}}|
	\leq |u_{B_j} - u_{\widetilde{B}_j}| + |u_{B_{j+1}} - u_{\widetilde{B}_j}|
	\leq \bariint_{B_j} |u - u_{\widetilde{B}_j}| + \bariint_{B_{j+1}} |u - u_{\widetilde{B}_j}| 
	\\  
	\lesssim \bariint_{\widetilde{B}_j} |u - u_{\widetilde{B}_j}|
	\leq \varphi(2\delta(X_j))
	= \varphi (2^{-j+1} \delta(X)).
	\end{multline*}
	Inserting this into \eqref{eq:telescopic}, we have 
	\begin{multline} \label{eq:average_B_0}
	\bigg|\bariint_{B_0} u(Y) \, dY\bigg|
	\lesssim 
	\sum_{j = 0}^{\infty} \varphi (2^{-j+1} \delta(X))
	\approx 
	\sum_{j = 0}^{\infty} \varphi (2^{-j+1} \delta(X)) \int_{2^{-j+1}\delta(X)}^{2^{-j+2}\delta(X)} \frac{ds}{s}
	\\ \leq
	\sum_{j = 0}^{\infty} \int_{2^{-j+1} \delta(X)}^{2^{-j+2} \delta(X)} \varphi(s) \frac{ds}{s}
	=
	\int_{0}^{4 \delta(X)} \varphi(s) \frac{ds}{s}.
	\end{multline}		
	On the other hand, again in a similar spirit to Lemma~\ref{lem:Qalpha}, we may compute 
	\begin{equation} \label{eq:domination_varphi_integral_2}
	\varphi(\delta(X)) 
	\approx 
	\varphi(\delta(X)) \int_{\delta(X)}^{2\delta(X)} \frac{ds}{s}
	\leq 
	\int_{\delta(X)}^{2\delta(X)} \varphi(s) \frac{ds}{s}
	\leq 
	\int_0^{4\delta(X)} \varphi(s) \frac{ds}{s}.
	\end{equation}
	Hence, estimating \eqref{eq:uuB} with \eqref{eq:estimate_u_minus_B_0} (and later using \eqref{eq:domination_varphi_integral_2}) and \eqref{eq:average_B_0}, we end up getting 
	\begin{equation*} 
	|u(X)| \lesssim \int_{0}^{4 \delta(X)} \varphi(s) \frac{ds}{s}, \quad\forall X \in \Omega,
	\end{equation*}
	so taking into account that $\varphi \in \mathcal{G}_\beta$, we may use \eqref{eq:ass-phi} to obtain 
	\begin{equation*}
	|u(X)| \lesssim \varphi(4\delta(X)), \quad\forall X \in \Omega,
	\end{equation*} 	
	This in turn implies that 
	\begin{equation}\label{eq:uX-estimate}
	\sup_{B(x, r) \cap \Omega} |u| \lesssim \varphi(4r), \quad\forall x \in \pom, r>0, 
	\end{equation}
	where the implicit constant is independent of $x$ and $r$.
	
	Now fix $X_0 \in \Omega$ and pick $\widehat{x}_0 \in \pom$ such that $|X_0-\widehat{x}_0|=\delta(X_0)$. 
	Invoking \eqref{eq:uX-estimate} and Boundary H\"{o}lder continuity from Lemma~\ref{lem:Holder} applied to $u$ (with $\restr{u}{\pom} \equiv 0$), we get, for $R$ big enough so that $X_0 \in B(\widehat{x}_0, R/2)$: 
	\begin{equation*}
	|u(X_0)| 
	\lesssim 
	\bigg( \frac{|X_0 - \widehat{x}_0|}{R} \bigg)^{\alpha_2} \sup_{B(\widehat{x}_0, R) \cap \overline{\Omega}} |u| 
	\lesssim 
	\delta(X_0)^{\alpha_2} \frac{\varphi(4R)}{R^{\alpha_2}}
	\lesssim 
	\delta(X_0)^{\alpha_2} \frac{Q_{\alpha_2}\varphi(4R)}{(4R)^{\alpha_2}}
	\underset{R \to \infty}{\longrightarrow} 
	0 
	\end{equation*}
	by Lemma~\ref{lem:Qalpha}, since $\beta \leq \alpha_2$ and $\varphi$ satisfies \eqref{eq:ass-phi}.  
	Since $X_0$ was chosen arbitrarily, this shows that $u \equiv 0$, hence finishing the proof of the uniqueness part in Theorem~\ref{thm:well}.
	 
	\begin{remark} \label{rem:different_spaces_uniqueness}
		We may also generalize this uniqueness result in the spirit of Remark~\ref{rem:different_spaces_existence}. In fact, the reader may check from the above proof that we have shown the following: if $\psi$ is a growth function for which
		\begin{equation} \label{eq:ass_uniqueness}
		\int_0^1 \psi(s) \frac{ds}{s} < \infty
		\qquad 
		\text{and}
		\qquad 
		R^{-\alpha_2} \int_0^R \psi(s) \frac{ds}{s} 
		\underset{R \to \infty}{\longrightarrow}
		0,
		\end{equation}
		then there is uniqueness of solutions $u$ to the version of the problem \eqref{eq:C-varphi} with $u \in \dot{\mathscr{C}}^\psi(\Omega)$. Note that the first condition in \eqref{eq:ass_uniqueness} guarantees that the integral in the second condition in \eqref{eq:ass_uniqueness} is well-defined.
		
		Following Remark~\ref{rem:different_spaces_existence}, let us consider the case $\psi := Q_\beta \varphi$ for some $\beta \in (0, 1)$ and $\varphi$ a growth function. Then, it turns out that the conditions in \eqref{eq:ass_uniqueness} (so that one has uniqueness of solutions to \eqref{eq:C-varphi} with $u \in \dot{\mathscr{C}}^{Q_\beta \varphi}(\Omega)$) are fulfilled if $\beta \leq \alpha_2$, $\int_0^1 \varphi(s) \frac{ds}{s} < \infty$, and $Q_\beta \varphi(1) < \infty$ are all satisfied. Indeed, if such properties hold, using Fubini's theorem we obtain:
		\begin{equation*}
		\int_0^1 Q_\beta \varphi(s) \frac{ds}{s}
		\approx 
		\int_0^1 \varphi(s) \frac{ds}{s} + Q_\beta \varphi(1)
		< 
		\infty.
		\end{equation*}
		Next, to verify the second condition in \eqref{eq:ass_uniqueness}, use again Fubini's theorem to obtain that for every $R > 1$:
		\begin{multline*}
		R^{-\alpha_2}{\int_0^R Q_\beta \varphi(s) \frac{ds}{s}}
		\approx 
				R^{-\alpha_2}{\int_0^1 \varphi(s) \frac{ds}{s}}
		+ 
				R^{-\alpha_2} {\int_1^R \varphi(s) \frac{ds}{s}}
		+ 
				R^{\beta -\alpha_2}{\int_R^\infty \frac{\varphi(s)}{s^\beta} \frac{ds}{s}}
				\\
		=:
		I(R) + II(R) + III(R).
		\end{multline*}
		The hypotheses $\beta \leq \alpha_2$, $\int_0^1 \varphi(s) \frac{ds}{s} < \infty$ and $Q_\beta \varphi(1) < \infty$ yield $\lim_{R \to \infty} I(R) + III(R) = 0$. To deal with $II(R)$,  we use Fubini's theorem, that $\beta \leq \alpha_2$ and $Q_\beta \varphi(1) < \infty$ to arrive at
		\begin{multline*}
			\sum_{k=1}^\infty \left( 2^{-k\alpha_2} \int_1^{2^k} \varphi(s) \frac{ds}{s}\right)
			\lesssim
			\sum_{k=1}^\infty \left( 2^{-k\alpha_2} \sum_{j = 1}^k \varphi(2^j) \right)
			\lesssim 
			\sum_{j=1}^\infty 2^{-j\alpha_2} \varphi(2^j)
			\\
			\leq 
			\sum_{j=1}^\infty 2^{-j\beta} \varphi(2^j)
			\lesssim 
			\int_1^\infty \frac{\varphi(s)}{s^\beta} \frac{ds}{s}
			=
			Q_\beta \varphi(1) 
			<
			\infty.
		\end{multline*}
		This readily implies that  $\lim_{k \to \infty} 2^{-k\alpha_2} \int_1^{2^k} \varphi(s) \frac{ds}{s} = 0$, hence  $\lim_{R \to \infty} II(R) = 0$.
	\end{remark} 
	
	\subsection{Proof of Corollary~\ref{cor:fatou}}
	Let $u\in W^{1,2}_{\loc}(\Omega)\cap\dot{\mathscr{C}}^{\alpha}(\Omega)$ satisfy $Lu=0$ in the weak sense in $\Omega$. By Lemma~\ref{lem:extension} (along with Remark~\ref{rem:growth_stronger}), $u$ extends uniquely to a function defined in $\overline{\Omega}$, which is called again $u$, so that $u\in  \dot{\mathscr{C}}^{\alpha}(\overline{\Omega})$ and
	$\|u|_{\pom}\|_{\dot{\mathscr{C}}^{\alpha}(\pom)} \le  \|u\|_{\dot{\mathscr{C}}^{\alpha}(\overline{\Omega})} \lesssim \|u\|_{\dot{\mathscr{C}}^{\alpha}(\Omega)}$. In particular, $u$ solves \eqref{eq:problem} with $f=u|_{\pom}$. Theorem~\ref{thm:Ca-well} readily implies \eqref{eqn:corol:Fatou}. 
	
	To see \eqref{eqn:ident} let $f\in \dot{\mathscr{C}}^{\alpha}(\Omega)$ and invoke  Theorem~\ref{thm:Ca-well} to solve \eqref{eq:problem}. Then, $u\in \dot{\mathscr{C}}_{L}^{\alpha}(\Omega)$ and $f=u|_{\pom}$. Thus, $f=u|_{\pom}$ belongs to the set on the right hand side of \eqref{eqn:ident}. The converse inclusion follows trivially from Lemma~\ref{lem:extension}.
	
	To continue, we note that \eqref{eqn:corol:Fatou} and \eqref{eqn:ident} easily imply that the linear mapping \eqref{eqn:map} is a well-defined bounded operator. Note also that if $[u|_{\pom}]=0$ then $u|_{\pom}$ is constant, say $c_0$. This, \eqref{eqn:corol:Fatou}, and the fact that elliptic measure is a probability in the current scenario imply that $u$ must be identically to $c_0$. Eventually, $[u]=0$ and we have shown that the operator in \eqref{eqn:map} is injective. It is additionally surjective by Theorem~\ref{thm:Ca-well} and invoking again \eqref{eqn:corol:Fatou} the converse of the mapping \eqref{eqn:map} is also bounded. Finally, given that $\big\|[\cdot]\big\|_{\dot{\mathscr{C}}^{\alpha}(\pom)/\re}$ is complete we readily obtain that $\dot{\mathscr{C}}^{\alpha}_L(\Omega)/\re$ equipped with the norm $\big\|[\cdot] \big\|_{\dot{\mathscr{C}}^{\alpha}(\Omega)/\re}$ is a Banach space. This completes the proof of Corollary~\ref{cor:fatou}.


	\section{Ill-posedness of the Dirichlet problem in Hölder spaces}\label{sec:ill}
	The goal of this section is to show Theorem \ref{thm:Ca-ill}. This will be done using Remark \ref{rem:alpha} and the following more general result. 
	
	\begin{theorem}
		Let $\Omega \subset \R^{n+1}$, $n \geq 2$, be an open set satisfying the CDC so that {\bf $\Omega$ is unbounded with $\pom$ being bounded}. Let $L = -\div (A\nabla)$ be a real (non-necessarily symmetric) elliptic operator, and let $\omega_L$ be the associated elliptic measure. Then there exists $\beta \in (0, 1)$ (depending only on $n$, the  CDC constant, and the ellipticity constant of $L$) such that for every growth function $\varphi \in \mathcal{G}_\beta$, for every $f \in \dot{\mathscr{C}}^{\varphi}(\pom)$, and for every $y_0 \in \pom$, if we define 
		\begin{equation} \label{eq:uX_y0}
		u(X) = f(y_0) + \int_{\pom} (f(y)-f(y_0)) \, d\w_L^X(y), \qquad X \in \Omega, 
		\end{equation}
		then $u \in W^{1,2}_{\loc}(\Omega) \cap \dcv(\overline{\Omega})$, satisfies $u|_{\pom}=f$, and solves the $\dcv$-Dirichlet problem for $L$ in $\Omega$ as formulated in \eqref{eq:C-varphi}. Moreover, there exists $C$ (depending on $n$, the CDC constant, the ellipticity constant of $L$, and $C_\varphi$) such that 
		\begin{equation*}
		\|f\|_{\dot{\mathscr{C}}^\varphi(\pom)}
		\leq \|u\|_{\dot{\mathscr{C}}^\varphi(\Omega)}
		\leq C \|f\|_{\dot{\mathscr{C}}^\varphi(\pom)}.
		\end{equation*}
		Nonetheless, the $\dcv$-Dirichlet problem for $L$ in $\Omega$ is ill-posed because \eqref{eq:C-varphi} always has more than one solution if $f$ is non-constant. 
	\end{theorem}
	
	\begin{proof} 
		As in the proof of Theorem~\ref{thm:well}, we may assume $\|f\|_{\dcv(\pom)} = 1$ because the case $f$ being constant is trivial. 
		As before, let $\beta \leq \alpha_0 := \min \{\alpha_1, \alpha_2\}$, where $\alpha_1\in (0,1)$ is the exponent in Lemma~\ref{lem:DGN} and $\alpha_2\in (0,1)$ is the exponent in Lemma~\ref{lem:Holder}.
		
		Fix some arbitrary $y_0 \in \pom$. We first show that the function $u$ defined in \eqref{eq:uX_y0} is a solution to the problem \eqref{eq:C-varphi}. Much as in the proof of Theorem~\ref{thm:well}, we proceed in a series of steps. Let us denote
		\begin{equation*}
		u(X; z) := f(z) + \int_\pom (f(y) - f(z)) \,d\omega_L^X(y), \quad X \in \Omega, \, \,  z \in \pom.  
		\end{equation*}
		
		\textbf{Step 1: $u$ is a well defined solution}. The reader may observe that in this case we are defining $u$ via an integration of a bounded and continuous function in the compact set $\pom$, hence, the integral is a well defined weak solution with $u \in W_{\loc}^{1,2}(\Omega) \cap \mathscr{C}(\overline{\Omega})$ and $\restr{u}{\pom}=f$. 
		
		\medskip 
		
		\textbf{Step 2: $u \in \dcv(\Omega)$.} 
		Since $\varphi \in \mathcal{G}_\beta$ for $\beta \leq \alpha_0$, by Lemma~\ref{lem:Qalpha} it suffices to show that
		\begin{equation*}
		|u(X; y_0) - u(Y; y_0)|
		\lesssim 
		Q_{\alpha_0} \varphi(|X-Y|),
		\qquad X, Y \in \Omega.
		\end{equation*}  
		For that purpose, fix $X, Y \in \Omega$. Without loss of generality we may assume that $\delta(X) \leq \delta(Y)$. Let $\widehat{x}, \widehat{y} \in \pom$ be points such that $|X-\widehat{x}| = \delta(X)$ and $|Y-\widehat{y}| = \delta(Y)$, respectively. Let us distinguish some cases. 
		
		\textbf{Case 1: $|X-Y| < \delta(Y) / 4$ and $\diam(\pom) < \delta(Y)$.} We first note that for any $Z \in \Omega$, since $\omega_L^Z(\pom) \leq 1$,
		\begin{equation*}
		|u(Z; y_0) - f(y_0)| 
		\le \int_\pom |f(y) - f(y_0)| \,d\omega_L^Z(y)
		\le \int_\pom \varphi(|y- y_0|) \,d\omega_L^Z(y)
		\le \varphi(\diam(\pom)). 
		\end{equation*}
		With this estimate and De Giorgi/Nash (see Lemma~\ref{lem:DGN}) we deduce that 
		\begin{align*}
		|u(X; y_0) - u(Y; y_0)|
		& =|(u(X; y_0) -f(y_0)) - (u(Y; y_0) - f(y_0))|
		\\ 
		& \lesssim 	\bigg(\frac{|X-Y|}{\delta(Y)} \bigg)^{\alpha_1} 
		\bigg(\bariint_{B(Y, \delta(Y) / 2))} |u(Z; y_0) - f(y_0)|^2 dZ \bigg)^{\frac12}
		\\ 
		& \lesssim 	\bigg(\frac{|X-Y|}{\delta(Y)} \bigg)^{\alpha_1} \varphi(\diam(\pom)).
		\end{align*} 
		Then, using the monotonicity of $\varphi$, and Lemma~\ref{lem:Qalpha} repeatedly, we obtain 
		\begin{equation*}
		|u(X; y_0) - u(Y; y_0)| 
		\lesssim 
		\bigg(\frac{|X-Y|}{\delta(Y)} \bigg)^{\alpha_1} \varphi(\delta(Y)) 
		\leq  
		\bigg(\frac{|X-Y|}{\delta(Y)} \bigg)^{\alpha_0} Q_{\alpha_0}\varphi(\delta(Y)) 
		\lesssim 
		Q_{\alpha_0}\varphi(|X-Y|),  
		\end{equation*} 
		
		\textbf{Case 2: $|X-Y| < \delta(Y) / 4$ and $\delta(Y) \le \diam(\pom)$.}  In this case $X \in B(Y, \delta(Y)/4)$, so we will use interior estimates. We can first compute 
		\begin{align}\label{UUXY}
		|u(X; y_0) - u(Y; y_0)| 
		& =
		\bigg| \int_\pom (f(y) - f(y_0)) \,d\omega_L^X(y) - \int_\pom (f(y) - f(y_0)) \,d\omega_L^Y(y) \bigg|
		\\ \nonumber
		& \leq 
		\bigg| \int_\pom (f(y) - f(\widehat{y})) \,d\omega_L^X(y) - \int_\pom (f(y) - f(\widehat{y})) \,d\omega_L^Y(y) \bigg|
		\\ \nonumber
		& \quad + \bigg| \int_\pom (f(\widehat{y}) - f(y_0)) \,d\omega_L^X(y) - \int_\pom (f(\widehat{y}) - f(y_0)) \,d\omega_L^Y(y) \bigg|
		\\ \nonumber
		& \leq
		|u(X; \widehat{y}) - u(Y; \widehat{y})| + \varphi(\diam(\pom)) |\omega_L^X(\pom) - \omega_L^Y(\pom)|.
		\end{align}
		Let us estimate each term separately. For the first term, we proceed as in Case 1 in the proof of Theorem~\ref{thm:well} (using De Giorgi/Nash, Harnack, Lemma~\ref{lem:8r4r}, 
		$\alpha_0 = \min \{ \alpha_1, \alpha_2 \}$ and Lemma~\ref{lem:Qalpha}):
		\begin{align*}
		|u(X; \widehat{y}) - u(Y; \widehat{y})|
		& = 
		\bigg| \int_\pom (f(y) - f(\widehat{y})) \,d\omega_L^X(y) - \int_\pom (f(y) - f(\widehat{y})) \,d\omega_L^Y(y) \bigg|
		\\ & \lesssim 	
		\bigg(\frac{|X-Y|}{\delta(Y)} \bigg)^{\alpha_1} 
		\bigg(\bariint_{B(Y, \delta(Y) / 2))} \bigg(\int_\pom |f(z) - f(\widehat{y})| d\w_L^Z(z) \bigg)^2 dZ \bigg)^{\frac12}
		\\ 
		& \approx \bigg(\frac{|X-Y|}{\delta(Y)} \bigg)^{\alpha_1} 
		\bigg( \bariint_{B(Y, \delta(Y) / 2))} \bigg(\int_\pom |f(z) - f(\widehat{y})| d\w_L^Y(z) \bigg)^2 dZ \bigg)^{\frac12}
		\\ 
		& = \bigg( \frac{|X-Y|}{\delta(Y)} \bigg)^{\alpha_1} \int_\pom |f(z) - f(\widehat{y})| d\w_L^Y(z)
		\\ 
		& \leq \bigg(\frac{|X-Y|}{\delta(Y)} \bigg)^{\alpha_1} \int_\pom \varphi(|z - \widehat{y}|) \, d\w_L^Y(z)
		\\ 
		& \lesssim \bigg( \frac{|X-Y|}{\delta(Y)} \bigg)^{\alpha_1} Q_{\alpha_2}\varphi(\delta(Y))
		\\ 
		& \leq \bigg( \frac{|X-Y|}{\delta(Y)} \bigg)^{\alpha_0} Q_{\alpha_0}\varphi(\delta(Y))
		\\
		& \leq Q_{\alpha_0} \varphi(|X-Y|). 
		\end{align*} 
		Therefore it is only left to estimate the second term in \eqref{UUXY}. We use first De Giorgi/Nash (Lemma~\ref{lem:DGN}), then Harnack's inequality applied to the positive solution $1-\w_L^Z(\pom)$ and Boundary H\"{o}lder continuity (cf. Lemma \ref{lem:Holder}):
		\begin{align*}
		|\omega_L^X(\pom) - \omega_L^Y(\pom)|
		& =
		|(1-\omega_L^X(\pom)) - (1-\omega_L^Y(\pom))| 
		\\ & \lesssim 
		\bigg( \frac{|X-Y|}{\delta(Y)} \bigg)^{\alpha_1} 
		\bigg( \bariint_{B(Y, \delta(Y) / 2)} |1-\omega_L^Z(\pom)|^2 dZ \bigg)^{1/2} 
		\\ & \approx 
		\bigg( \frac{|X-Y|}{\delta(Y)} \bigg)^{\alpha_1} (1-\omega_L^Y(\pom))
		\\ & \lesssim 
		\bigg( \frac{|X-Y|}{\delta(Y)} \bigg)^{\alpha_1} \bigg( \frac{\delta(Y)}{\diam(\pom)} \bigg)^{\alpha_2} 
		\sup_{Z \in B(\widehat{y}, \diam(\pom))} (1 - \omega_L^Z(\pom)) 
		\\ & \leq 
		\bigg( \frac{|X-Y|}{\delta(Y)} \bigg)^{\alpha_1} \bigg( \frac{\delta(Y)}{\diam(\pom)} \bigg)^{\alpha_2}
		\\ & \leq 
		\bigg( \frac{|X-Y|}{\delta(Y)} \bigg)^{\alpha_0} \bigg( \frac{\delta(Y)}{\diam(\pom)} \bigg)^{\alpha_0}
		\\
		& =\bigg( \frac{|X-Y|}{\diam(\pom)} \bigg)^{\alpha_0},
		\end{align*} 
		where we have also used that $\alpha_0 =\min\{\alpha_1, \alpha_2\}$. As a result, by Lemma~\ref{lem:Qalpha},
		\begin{equation*}
		\varphi(\diam(\pom)) |\omega_L^X(\pom) - \omega_L^Y(\pom)| 
		\lesssim 
		Q_{\alpha_0}\varphi(\diam(\pom)) \bigg( \frac{|X-Y|}{\diam(\pom)} \bigg)^{\alpha_0} 	
		\lesssim
		Q_{\alpha_0}\varphi(|X-Y|).
		\end{equation*}  
		
		\textbf{Case 3: $\delta(Y) / 4 \leq |X - Y|$ and $|X-Y| \leq \diam(\pom)$.} As in the proof of Theorem~\ref{thm:well}, we may suppose $|\widehat{x} - \widehat{y}| \leq 10 |X-Y|$  and write 
		\begin{equation} \label{eq:uX-uY_triangle}
		|u(X; y_0) - u(Y; y_0)| 
		\leq 
		|u(X; y_0) - u(X; \widehat{x}) - u(Y; y_0) + u(Y; \widehat{y})| + |u(X; \widehat{x}) - u(Y; \widehat{y})|.
		\end{equation}
		Let us first estimate the second term. Using the triangle inequality we get 
		\begin{equation}\label{UXY}
		|u(X; \widehat{x}) - u(Y; \widehat{y})| 
		\leq 
		|u(X; \widehat{x}) - f(\widehat{x})| + |u(Y; \widehat{y}) - f(\widehat{y})| + |f(\widehat{x}) - f(\widehat{y})|.
		\end{equation} 
		The first term in this sum can be estimated much as before using Lemmas~\ref{lem:8r4r} and \ref{lem:Qalpha}: 
		\begin{multline*}
		|u(X; \widehat{x}) - f(\widehat{x})| 
		\leq
		\int_\pom |f(y) - f(\widehat{x})| \,d\omega_L^X(y)
		\leq 
		\int_\pom \varphi(|y-\widehat{x}|) \,d\omega_L^X(y)
		\\ \lesssim
		Q_{\alpha_2} \varphi(\delta(X))
		\leq 
		Q_{\alpha_2} \varphi(4 |X-Y|)
		\lesssim 
		Q_{\alpha_2} \varphi(|X-Y|)
		\leq 
		Q_{\alpha_0} \varphi(|X-Y|), 
		\end{multline*}
		because $\alpha_0 \leq \alpha_2$. 
		The same ideas apply to the second term in \eqref{UXY}. In turn, the last term is easy to estimate with Lemma~\ref{lem:Qalpha}: 
		\begin{equation*}
		|f(\widehat{x}) - f(\widehat{y})| 
		\leq
		\varphi(|\widehat{x} - \widehat{y}|)
		\leq 
		\varphi(10 |X-Y|)
		\lesssim 
		Q_{\alpha_0} \varphi (10|X-Y|)
		\lesssim 
		Q_{\alpha_0} \varphi(|X-Y|).
		\end{equation*} 
		We turn now our attention to the first term in \eqref{eq:uX-uY_triangle} and use again the triangle inequality to introduce some additional boundary terms (along with the monotonicity of $\varphi$):
		\begin{align} \label{eq:uX-uY_several_triangles}
		|u(X; y_0)  - u(X; \widehat{x}) - u(Y; y_0) + u(Y; \widehat{y})| 
		&
		\\ & \hspace{-6cm} 
		= | (f(\widehat{y}) - f(\widehat{x})) (1-\omega_L^X(\pom)) 
		+ (f(\widehat{y}) - f(y_0)) (\omega_L^X(\pom) - \omega_L^Y(\pom)) | 
		\nonumber \\ & \hspace{-6cm} \leq
		|f(\widehat{x}) - f(\widehat{y}))| + |f(\widehat{y}) - f(y_0)| |\omega_L^X(\pom) - \omega_L^Y(\pom)| 
		\nonumber \\ & \hspace{-6cm} \leq 
		\varphi(|\widehat{x} - \widehat{y}|) + \varphi(|\widehat{y} - y_0|) |\omega_L^X(\pom) - \omega_L^Y(\pom)| 
		\nonumber \\ & \hspace{-6cm} \leq 
		\varphi(10|X-Y|) + \varphi(\diam(\pom)) |\omega_L^X(\pom) - \omega_L^Y(\pom)|. \nonumber  
		\end{align}
		
		Next, the Boundary H\"{o}lder continuity from Lemma~\ref{lem:Holder} yields 
		\begin{align*}
		|\omega_L^X(\pom) - \omega_L^Y(\pom)| 
		& \leq
		|1-\omega_L^X(\pom)| + |1-\omega_L^Y(\pom)| 
		\\ & \lesssim 
		\bigg( \frac{\delta(X)}{\diam(\pom)} \bigg)^{\alpha_2} \sup_{Z \in B(\widehat{x}, \diam(\pom))} (1 - \omega_L^Z(\pom)) 
		\\ & \;\;\;\; 
		\quad+ \bigg( \frac{\delta(Y)}{\diam(\pom)} \bigg)^{\alpha_2} \sup_{Z \in B(\widehat{y}, \diam(\pom))} (1 - \omega_L^Z(\pom)) 
		\\ & \leq
		\bigg( \frac{\delta(X)}{\diam(\pom)} \bigg)^{\alpha_2} + \bigg( \frac{\delta(Y)}{\diam(\pom)} \bigg)^{\alpha_2} 
		\\ & \lesssim 
		\bigg( \frac{|X-Y|}{\diam(\pom)} \bigg)^{\alpha_2},
		\end{align*}
		since $\omega_L^Z(\pom) \in [0, 1]$ for any $Z$. Thus, returning to \eqref{eq:uX-uY_several_triangles} we conclude that  
		\begin{multline*}
		|u(X; y_0) - u(Y; y_0) - u(X; \widehat{x}) + u(Y; \widehat{y})| 
		\\ 
		\lesssim 
		\varphi(10|X-Y|) + \varphi(\diam(\pom)) \bigg( \frac{|X-Y|}{\diam(\pom)} \bigg)^{\alpha_0}
		\lesssim 
		Q_{\alpha_0}\varphi(|X-Y|)
		\end{multline*}
		using $\alpha_0 \leq \alpha_2$, Lemma~\ref{lem:Qalpha}, and  $|X-Y| \leq \diam(\pom)$. Collecting all the estimates, the proof in Case 3 is complete.
		
		\textbf{Case 4: $\delta(Y) / 4 \leq |X - Y|$ and $|X-Y| > \diam(\pom)$}. This case is easy:
		\begin{multline*}
		|u(X; y_0) - u(Y; y_0)| 
		=
		\bigg| \int_\pom (f(y) - f(y_0)) \,d\omega_L^X(y) - \int_\pom (f(y) - f(y_0)) \,d\omega_L^Y(y) \bigg|
		\\ 
		\leq
		\int_\pom \varphi(|y-y_0|) \,d\omega_L^X(y) + \int_\pom \varphi(|y-y_0|) \,d\omega_L^Y(y)
		\\ 
		\lesssim 
		\varphi(\diam(\pom))
		\leq 
		\varphi(|X-Y|)
		\leq 
		Q_{\alpha_0} \varphi(|X-Y|),
		\end{multline*}
		where we have used the monotonicity of $\varphi$, the fact that $\omega_L^Z(\pom) \leq 1$ for any $Z \in \Omega$, and Lemma~\ref{lem:Qalpha}.

		Collecting all the cases above, we conclude that $u \in \dot{\mathscr{C}}^{Q_{\alpha_0}\varphi}(\Omega)$.  By Lemma~\ref{lem:Qalpha} and the fact that $\varphi \in \mathcal{G}_\beta$ with $\beta \leq \alpha_0$ one can easily see that $Q_{\alpha_0}\varphi\approx \varphi$, hence $u \in \dot{\mathscr{C}}^\varphi(\Omega)$. 
		
		\medskip
		
		\textbf{Step 3: $u \in \dcv(\overline{\Omega})$ and \eqref{eq:fuf} holds}. 
		This follows at once from the previous steps and Lemma \ref{lem:extension} (along with Remark~\ref{rem:growth_stronger}).  
		
		\textbf{Step 4: Lack of uniqueness.} As a final remark, let us show that the constructed solutions are different whenever $f$ is not constant. Indeed, if $f$ is not constant, we can pick $x_0, y_0 \in \pom$ for which $f(x_0) \neq f(y_0)$. If we consider the solutions $u(\cdot; x_0)$ and $u(\cdot; y_0)$, it turns out that for any $X \in \Omega$ we have
		\begin{multline*}
		u(X; x_0) - u(X; y_0)
		=
		f(x_0) + \int_\pom (f(y) - f(x_0)) \,d\omega_L^X(y) - f(y_0) - \int_\pom (f(y) - f(y_0)) \,d\omega_L^X(y)
		\\ =
		(f(x_0) - f(y_0)) (1-\omega_L^X(\pom)) 
		\neq 
		0, 
		\end{multline*}
		since the elliptic measure is not a probability in this situation. This shows the ill-posedness of the problem in the case when $f$ is not constant, and the proof is complete.  
	\end{proof}

	\begin{remark}
		As the reader may check from the above proof, generalizations of this result to more general classes of growth functions $\varphi$, in the spirit of Remarks~\ref{rem:different_spaces_existence} and \ref{rem:different_spaces_uniqueness}, are also possible. 
	\end{remark}


	\section{Optimality of the capacity density condition}\label{sec:optimal}
	
	In this section we prove Theorem~\ref{th:necessity_CDC}. Our proof extends \cite{A} and contains several novelties and new technical difficulties because we are working with a general elliptic operator $L$ rather than the Laplacian, and $\Omega$ is not assumed to be connected nor bounded.
	In fact, the case when $\Omega$ is unbounded with $\pom$ being bounded has some additional subtle points to be taken into account.

	Unlike in the other sections, for the sake of clarity and of a cleaner exposition, we are going to show Theorem~\ref{th:necessity_CDC} only for the $\dot{\mathscr{C}}^\alpha$ spaces, not for the more general $\dot{\mathscr{C}}^{\varphi}$, where $\varphi$ is a growth function satisfying \eqref{eq:ass-phi}. It is not hard to check that our proof can be adapted to the $\dot{\mathscr{C}}^{\varphi}$ setting with small modifications, and we give a sketch of such proof in Subsection~\ref{subsec:CDC_growth}.

	Let us start with some auxiliary results. We first introduce a notion of regularity of points in the boundary, which will be used frequently in the subsequent proofs.
	
	\begin{definition}[\textbf{Regular set}]\label{def:regular}
		Let $\Omega \subset \ree$, $n \ge 2$, be an open set, and let $L=-\div(A\nabla)$ be a real (non-necessarily symmetric) elliptic operator.
		We say that $\Omega$ is \textit{(Wiener) regular for $L$} if for any function $f: \pom \longrightarrow \R$ which is uniformly continuous and bounded, the associated solution $u$ defined by \eqref{eq:well-solution} 
		satisfies $u \in \mathscr{C}(\overline{\Omega})$.
	\end{definition}

	\begin{remark}
		Note that in the case $\Omega$ is unbounded with $\pom$ being bounded, if $\Omega$ is regular for $L$, then the solutions defined by \eqref{eq:ill-solution} are also $\mathscr{C}(\overline{\Omega})$. We will use this fact repeatedly from now on. The short proof is as follows: indeed, any solution defined by \eqref{eq:ill-solution} can be rewritten as
		\begin{equation*}
		u(X; y_0) = f(y_0) (1-\omega_L^X(\pom)) + \int_\pom f(y) \, d\omega_L^X(y),
		\end{equation*}
		where $X \mapsto \int_\pom f(y) \, d\omega_L^X(y)$ belongs to $\mathscr{C}(\overline{\Omega})$ by regularity of $\Omega$, and $X \mapsto \omega_L^X(\pom)$ clearly belongs to $\mathscr{C}(\overline{\Omega})$ because it is the solution given by \eqref{eq:well-solution} associated with the boundary datum constantly 1, using again that $\Omega$ is regular. Thus, $u(\cdot\,; y_0) \in \mathscr{C}(\overline{\Omega})$, as desired.

	\end{remark}

	Note that Definition~\ref{def:regular} coincides with the classical one, widely known to be equivalent to the Wiener criterion, which is in fact a qualitative version of the capacity density condition that we want to show in this section. Since the operator $L$ will always be clear by the context, we will often abuse notation simply saying \textit{$\Omega$ is regular}. This abuse of notation is also motivated by the well-known fact that the regularity of $\Omega$ does not depend on the operator $L$ (see \cite{LSW}).

	Before proving Theorem~\ref{th:necessity_CDC}, let us first present a somehow qualitative version of it. 
	
	\begin{lemma} \label{lem:regular}
		Let $\Omega \subset \ree$, $n \ge 2$, be an open set, and let $L=-\div(A\nabla)$ be a real (non-necessarily symmetric) elliptic operator. Let $\alpha \in (0, 1)$. Assume that for every $f \in \dot{\mathscr{C}}^{\alpha}(\pom)$, the function $u$ constructed in  \eqref{eq:well-solution} in the case $\Omega$ is bounded or $\Omega$ is unbounded with $\pom$ being unbounded, or constructed in \eqref{eq:ill-solution} in the case $\Omega$ is unbounded with $\pom$ being bounded, is well-defined via an absolutely convergent integral and moreover is a solution to \eqref{eq:problem} satisfying \eqref{eq:bounds_CDC}. Then $\Omega$ is regular for $L$.
	\end{lemma}
	
	\begin{remark}
		We want to point out that our assumption above rules out the possibility that $\pom$ contains a polar set (trivial boundary points). For example, consider $\Omega$ to be the unit ball with the origin removed, i.e., $\Omega = B \setminus \{0\}$. Then it is not true that for any $f\in \dot{\mathscr{C}}^\alpha(\pom) = \dot{\mathscr{C}}^\alpha(\partial B \cup\{0\})$, we can find a solution $u$ to \eqref{eq:problem}, which in particular assumes the boundary datum pointwise, because any such $u$ must also be a classical solution to $L$ in the unit ball, and thus by the uniqueness of classical solutions we would have that $f(0)=u(0) = \int_{\partial B} f \, d\omega^0_{L,B}$. See also \cite[Theorem 1]{A} for comparison.
	\end{remark}
	
	\begin{proof}
		Let $f:\pom \longrightarrow \R$ be uniformly continuous and bounded. By a simple adaptation of \cite[Lemma 6.8]{H}, there exists a sequence $\{f_j\} \subset \dot{\mathscr{C}}^\alpha(\pom)$ such that $\norm{f - f_j}_{L^\infty(\pom)} \longrightarrow 0$ as $j \to \infty$. Denote by $u_j$ and $u$ the functions defined by \eqref{eq:well-solution} in the case $\Omega$ is bounded or $\Omega$ is unbounded with $\pom$ unbounded, or defined by \eqref{eq:ill-solution} with some fixed $y_0\in \pom$ in the case $\Omega$ is unbounded with $\pom$ bounded, associated with $f_j$ and $f$, respectively. We know that $u_j$ is well-defined by hypothesis and the same occurs for $u$ since $f$ is uniformly continuous and bounded. Since $f_j \in \dot{\mathscr{C}}^\alpha(\pom)$, by \eqref{eq:bounds_CDC} we know that $u_j \in \dot{\mathscr{C}}^\alpha(\Omega)$. Therefore it can be extended (denoting again the extension by $u_j$) to $u_j \in \dot{\mathscr{C}}^\alpha(\overline{\Omega}) \subset \mathscr{C}(\overline{\Omega})$, and moreover, since $u_j$ is a solution to \eqref{eq:problem} with boundary datum $f_j$, we also have that $u_j = f_j$ on $\pom$.
		
		Write 
		\begin{equation*}
		\widehat{u}(X) :=
		\begin{cases}
		u(X) \quad \text{ if } X \in \Omega, \\
		f(X) \quad \text{ if } X \in \pom. 
		\end{cases}
		\end{equation*}
		Our goal is to show $\widehat{u} \in \mathscr{C}(\overline{\Omega})$. To this end, fix $\varepsilon > 0$ and choose $N \gg 1$ such that $\norm{f - f_N}_{L^\infty(\pom)} < \varepsilon / 9$. Since $u_N \in \dot{\mathscr{C}}^\alpha(\overline{\Omega})$, take $\delta > 0$ (which depends on $N$ and $\varepsilon$) such that $\abs{X - Y} < \delta$ implies $\abs{u_N(X) - u_N(Y)} < \varepsilon / 3$ for any $X, Y \in \overline{\Omega}$. Then, for any $X, Y \in \overline{\Omega}$ such that $\abs{X - Y} < \delta$, we may compute 
		\begin{equation} \label{eq:split_epsilon3}
		\abs{\widehat{u}(X) - \widehat{u}(Y)}
		\leq 
		\abs{\widehat{u}(X) - u_N(X)} + \abs{u_N(X) - u_N(Y)} + \abs{u_N(Y) - \widehat{u}(Y)}
		=:
		I + II + III.
		\end{equation}
		The second term satisfies $II < \varepsilon / 3$. In turn, for the first term, we can distinguish two cases. If $X \in \Omega$ and $\Omega$ is bounded or $\Omega$ is unbounded with $\pom$ bounded, \eqref{eq:well-solution} allows us to compute 
		\begin{equation*}
		\abs{\widehat{u}(X) - u_N(X)}
		=
		\abs{u(X) - u_N(X)}
		\leq 
		\int_\pom |f - f_N| \, d\omega_L^X
		\leq 
		\norm{f - f_N}_{L^\infty(\pom)}
		< 
		\frac{\varepsilon}{9},
		\end{equation*}
		and if $\Omega$ is unbounded with $\pom$ bounded, a similar computation using \eqref{eq:ill-solution} shows 
		\begin{equation*}
		\abs{\widehat{u}(X) - u_N(X)}
		=
		\abs{u(X) - u_N(X)}
		\leq 
		3\norm{f - f_N}_{L^\infty(\pom)}
		< 
		\frac{\varepsilon}{3}.
		\end{equation*} 
		Whereas if $X \in \pom$ we have 
		\begin{equation*}
		\abs{\widehat{u}(X) - u_N(X)}
		=
		\abs{f(X) - f_N(X)}
		\leq 
		\norm{f - f_N}_{L^\infty(\pom)}
		< 
		\frac{\varepsilon}{9}.
		\end{equation*}
		In any case, there holds $I < \varepsilon / 3$. In the same fashion, $III < \varepsilon / 3$. In view of \eqref{eq:split_epsilon3},  these yield $\abs{\widehat{u}(X) - \widehat{u}(Y)} < \varepsilon$, which finally shows $\widehat{u} \in \mathscr{C}(\overline{\Omega})$ (indeed, we have shown that $\widehat{u}$ is uniformly continuous), as desired. 
	\end{proof}

	The following lemma will help us deal with the case of unbounded open sets (without any connectivity assumption), which is one of the main novelties of this section in comparison to the work of Aikawa \cite{A} in bounded domains. 
	
	\begin{lemma} \label{lem:decay_unbounded}
		Let $\Omega \subset \ree$, $n \geq 2$, be a regular open set, and let $L=-\div(A\nabla)$ be a real (non-necessarily symmetric) elliptic operator. Then there exists some $C > 0$ (depending only on $n$ and the ellipticity of $L$) such that for any $a \in \pom$ and $r > 0$ we have 
		\begin{equation} \label{eq:decay_unbounded}
		\omega_L^X(\pom \cap B(a, r)) \leq C \left( \frac{|X-a|}{r} \right)^{1-n} ,
		\qquad 
		X \in \Omega \setminus B(a, 2r).
		\end{equation} 
		As a consequence, in the case that $\Omega$ is unbounded, for any $a \in \pom$ and $r > 0$, we have 
		\begin{equation*}
		\omega_L^X(\pom \cap B(a, r)) \longrightarrow 0
		\quad 
		\text{ as } 
		\quad 
		X \to \infty.
		\end{equation*}
	\end{lemma}
	\begin{proof}
		The second assertion clearly follows from the first one, so let us focus on that first one.
		Note that we may assume $\Omega\setminus B(a,2r) \neq \emptyset$, for otherwise \eqref{eq:decay_unbounded} is an empty statement. Write $\Gamma_L(\cdot, \cdot)$ for the fundamental solution to $L$. It is well known that, for some $c_0 \in (0, 1)$ depending only on $n$ and the ellipticity of $L$,
		\begin{equation} \label{eq:fund_sol}
		c_0 \abs{X-Y}^{1-n} \leq \Gamma_L(X, Y) \leq c_0^{-1} \abs{X-Y}^{1-n}, 
		\qquad 
		X \neq Y.
		\end{equation}
		
		Using the definition of the elliptic measure in unbounded domains (and in bounded domains this is trivially true), we have, for any $X \in \Omega$,
		\begin{equation}\label{LXB-1} 
		\omega_L^X(B(a, r) \cap \pom)
		=
		\lim_{N \to \infty} \omega_{L, \Omega \cap B(a, N)}^X(B(a, r) \cap \pom),
		\end{equation}	
		where (here and elsewhere) $\omega_{L, D}$ refers to the elliptic measure relative to the set $D$; as opposed to $\omega_L$, which is always constructed relative to $\Omega$.
		
		Now, for $N \gg 2r$, choose $\phi \in \mathscr{C}^\infty_c(\ree)$ so that $\mathbf{1}_{B(a, r)} \leq \phi \leq \mathbf{1}_{B(a, 2r)}$. Denote by
		\begin{equation*}
		u_\phi(X) := \int_{\partial (\Omega \cap B(a, N))} \phi \, d\omega_{L, \Omega \cap B(a, N)}^X
		\end{equation*}
		the solution associated with $\phi$, which is continuous up to the boundary by the regularity of $\Omega$.
		
		For $X \in \Omega \cap \partial B(a, 2r)$, it follows from \eqref{eq:fund_sol} that 
		\begin{equation*}
		u_\phi(X) \leq 1 \leq \frac{\Gamma_L(X, a)}{c_0(2r)^{1-n}}.
		\end{equation*}
		Moreover, for $x \in \pom \cap \overline{B(a, N)} \setminus B(a, 2r)$ or $x \in \Omega \cap \partial B(a, N)$, the continuity of $u_\phi$ up to the boundary gives 
		\begin{equation*}
		u_\phi(x) = \phi(x) = 0 < \frac{\Gamma_L(x, a)}{c_0(2r)^{1-n}}.
		\end{equation*}
		Therefore, by the maximum principle in $(\Omega \cap B(a, N)) \setminus B(a, 2r)$, we obtain
		\begin{equation} \label{LXB-2}
		\omega_{L, \Omega \cap B(a, N)}^X(B(a, r) \cap \pom)
		\leq 
		u_\phi(X)
		\leq 
		\frac{\Gamma_L(X, a)}{c_0(2r)^{1-n}}
		\leq
		\frac{|X-a|^{1-n}}{c_0^2 (2r)^{1-n}}, 
		\end{equation}
		for all $X \in (\Omega \cap B(a, N)) \setminus B(a, 2r)$, where \eqref{eq:fund_sol} was used in the last step. As a consequence, \eqref{eq:decay_unbounded} follows from \eqref{LXB-1} and \eqref{LXB-2}.
	\end{proof}

	As a byproduct of the preceding lemma, we get the absolute continuity of the elliptic measure with respect to codimension-two Hausdorff measure, namely sets of lower dimensions are invisible to the elliptic measure. Although we will not use it in the sequel, we consider it to be interesting on its own. 
	
	\begin{corollary} 
		Let $\Omega \subset \ree$, $n \geq 2$, be a regular open set, and let $L=-\div(A\nabla)$ be a real (non-necessarily symmetric) elliptic operator. Then $\omega_L^X \ll \mathcal{H}^{n-1}$ for any $X \in \Omega$. 
	\end{corollary}
	\begin{proof}
		Fix $X \in \Omega$. Let $E \subset \pom$ such that $\mathcal{H}^{n-1}(E) = 0$. It is well known that  
		\begin{equation*}
		\liminf_{\delta \to 0} \bigg\{\sum_{i=1}^\infty r_i^{n-1} : E \subset \bigcup_{i=1}^\infty B(x_i, r_i), 2r_i < \delta \bigg\} 
		= 
		0.
		\end{equation*}	
		This means that given $\varepsilon > 0$, there exists some $\delta > 0$ small enough (we may assume that $\delta < \delta(X)/2$) such that for some $x_i \in \ree$ and $r_i \in (0, \delta/2)$, we have $E \subset \bigcup_i B(x_i, r_i)$ and $\sum_i r_i^{n-1} < \varepsilon$. Without loss of generality assume that $B(x_i, r_i) \cap \pom \neq \emptyset$. Note that then $E \subset \bigcup_i B(y_i, 2r_i)$ for some $y_i \in \pom$. Therefore, using the fact that $\delta(X) > 2\delta$ and  Lemma~\ref{lem:decay_unbounded}, we obtain  
		\begin{multline*}
		\omega_L^X(E)
		\leq 
		\sum_{i=1}^\infty \omega_L^X(\pom \cap B(y_i, 2r_i))
		\leq 
		C \sum_{i=1}^\infty \left( \frac{\abs{X-y_i}}{2r_i} \right)^{1-n} 
		\leq 
		\frac{2^{n-1} C}{\delta(X)^{n-1}} \sum_{i=1}^\infty r_i^{n-1}
		< 
		\frac{2^{n-1} C}{\delta(X)^{n-1}} \varepsilon. 
		\end{multline*}
		By the arbitrariness of $\varepsilon$, this indeed implies $\omega_L^X(E) = 0$. 
	\end{proof}

		The following estimate establishes the relationship between the capacitary potential and the elliptic measure, which plays an important ingredient to show CDC. 
		\begin{lemma}\label{lem:CP}
		Let $\Omega \subset \ree$, $n \geq 2$, be a regular open set, and let $L=-\div(A\nabla)$ be a real (non-necessarily symmetric) elliptic operator. Fix $a \in \pom$ and $0 < r < \diam(\pom)$. Denote $F := \overline{B(a, r)} \setminus \Omega$, and $v := \mathcal{R}_L(F; B(a, 2r))$ the capacitary potential (cf. \cite[Chapter 6]{HKM}) in $B(a, 2r)$ of the compact set $F$ with respect to the operator $L$. Then 
		\begin{align*}
		v(X) \geq 1 - \omega_{L, \Omega \cap B(a, r)}^X (\partial B(a, r) \cap \overline{\Omega}) 
		\quad\text{for any}\quad X \in \Omega \cap B(a, r).
		\end{align*}
		\end{lemma}
		
		\begin{proof}
		Write 
		\begin{equation*}
		u(X) := 1 - \omega_{L, \Omega \cap B(a, r)}^X (\partial B(a, r) \cap \overline{\Omega}) = \omega_{L, \Omega \cap B(a, r)}^X (\partial \Omega \cap B(a, r)).
		\end{equation*}
		Given $0<\eta<r$, choose $\phi_\eta \in \mathscr{C}^\infty_c(\ree)$ such that $0 \leq \phi_\eta \leq 1$, $\phi_\eta \equiv 1$ in $B(a, r-\eta)$, and $\phi_\eta \equiv 0$ outside $B(a, r)$. Set  
		\begin{equation*}
		u_\eta (X) := \int_{\partial (\Omega \cap B(a, r))} \phi_\eta \, d\omega_{L, \Omega \cap B(a, r)}^X
		\end{equation*}
		and the Dominated Convergence Theorem implies $u_\eta(X) \longrightarrow u(X)$ as $\eta \to 0$ for any $X \in \Omega \cap B(a, r)$. Hence, to conclude the proof it suffices to prove  
		\begin{align}\label{vueta}
		\text{$v \geq u_\eta$ in $\Omega \cap B(a, r)$ for any $\eta > 0$}.
		\end{align}
		Since $\Omega \cap B(a, r) \subset \Omega$ and $\Omega$ is regular, by the Wiener criterion $\Omega \cap B(a, r)$ is also regular. Then, the smoothness of $\phi_\eta$ implies that $u_\eta \in \mathscr{C}(\overline{\Omega \cap B(a, r)})$. 
		However $v$ may not be continuous or even defined pointwise near $\partial B(a,r) \cap \pom $, so we can not directly apply the classical maximum principle to prove \eqref{vueta}.
		
		By \cite[Lemma 7.9]{HKM}, we may extend $u_\eta$ by 0 to $B(a, 2r) \setminus F$, and if we denote by $\hat{u}_\eta$ the extension, $\hat{u}_\eta \in \mathscr{C}(\overline{B(a, 2r)\setminus F})$ is a weak subsolution. 
		By the weak maximum principle for subsolutions in $B(a, 2r) \setminus F$ (see \cite[Lemma 1.1.17]{K}), to prove \eqref{vueta} it suffices to show that $\hat{u}_\eta - v \leq 0$  
		in the $W^{1, 2}$-sense on $\partial (B(a, 2r) \setminus F)$. In fact, an easy modification of \cite[Definition 1.1.13]{K} lets us reduce to finding a sequence $\{f_j\} \subset \mathscr{C}^\infty (B(a, 2r) \setminus F) \cap W^{1, 2}(B(a, 2r) \setminus F)$ such that $f_j \to \hat{u}_\eta - v$ in $W^{1,2}(B(a,2r) \setminus F)$ and $f_j\leq 0$ in an open neighborhood $U_j$ of $\partial (B(a,2r) \setminus F)$. Let us next prove this.
		
		Pick $\Phi \in \mathscr{C}^\infty_c(\ree)$ such that $\mathbf{1}_{B(a, \frac32 r)} \le \Phi \le \mathbf{1}_{B(a, 2r)}$. Since $v$ is the solution to $Lv = 0$ in $B(a, 2r) \setminus F$ with data $\mathbf{1}_{\partial F}$, and $\Phi$ is a smooth extension of $\mathbf{1}_{\partial F}$, we have $v - \Phi \in W_0^{1, 2}(B(a, 2r) \setminus F)$. This means that there exists a sequence $\{\varphi_j\} \subset \mathscr{C}^\infty_c(B(a, 2r) \setminus F)$ such that $\varphi_j \to v-\Phi$ in $W^{1, 2}(B(a, 2r) \setminus F)$ as $j \to \infty$. 
		
		On the other hand, since $u_\eta - \phi_\eta \in W_0^{1, 2}(\Omega \cap B(a, r))$ because $u_\eta$ is a solution with smooth boundary data $\phi_\eta$, we can find $\psi_j \in \mathscr{C}^\infty_c(\Omega \cap B(a, r))$ so that $\psi_j \to u_\eta - \phi_\eta$ in $W^{1, 2}(\Omega \cap B(a, r))$. Denoting by $\hat{\psi}_j$ the extension of $\psi_j$ by 0, we get $\hat{\psi}_j \in \mathscr{C}^\infty_c(\ree)$ and $\hat{\psi}_j \to \hat{u}_\eta - \phi_\eta$ in $W^{1, 2}(B(a, 2r) \setminus F)$.   
		
		As a consequence, $\hat{\psi}_j + \phi_\eta - \varphi_j - \Phi \to \hat{u}_\eta - v$ in $W^{1, 2}(B(a, 2r) \setminus F)$ and $\hat{\psi}_j + \phi_\eta - \varphi_j - \Phi \in W^{1, 2}(B(a, 2r) \setminus F) \cap \mathscr{C}^\infty(B(a, 2r) \setminus F)$. Hence, to complete the proof, it remains to check that 
		\begin{align}\label{neigh}
		\text{$\hat{\psi}_j + \phi_\eta - \varphi_j - \Phi \leq 0$ in a neighborhood $U_j$ of $\partial (B(a, 2r) \setminus F)$.}
		\end{align} 
		Indeed, choose $U_j := \ree \setminus (\supp \varphi_j \cup \supp \hat{\psi}_j)$ so that $\varphi_j \equiv 0 \equiv \hat{\psi}_j$ in $U_j$. In $B(a, r) \cap \Omega$, we see that $\Phi \equiv 1$ and $\phi_\eta \leq 1$, so overall $\hat{\psi}_j + \phi_\eta - \varphi_j - \Phi \leq 0$ in $U_j \cap (B(a, r) \cap \Omega)$. Additionally, in $\R^{n+1} \setminus B(a, r)$, there holds $\phi_\eta \equiv 0$ and $\Phi \geq 0$, so $\hat{\psi}_j + \phi_\eta - \varphi_j - \Phi \leq 0$ in $U_j \setminus B(a, r)$. This justifies  \eqref{neigh} and the proof is complete.  
		\end{proof}

	Before showing Theorem~\ref{th:necessity_CDC} with the tools developed above, let us introduce some extra definitions to make the proof more understandable. The definitions below are indeed a generalization of those in \cite{A}. 
	
	\begin{definition}\label{def:HMD}
		Let $\Omega \subset \ree$, $n \ge 2$, be an open set. Let $L=-\div(A\nabla)$ be a real (non-necessarily symmetric) elliptic operator, and let $\w_L$ be the associated elliptic measure. 
		\begin{enumerate}
			\item Given $\alpha \in (0, 1)$, we say that $\Omega$ satisfies the \textit{global elliptic measure decay property} with exponent $\alpha$ (abbreviated to $\mathrm{GEMD}(\alpha)$) if there exists some $M > 0$ such that 
			\begin{equation}\label{cond:GEMD}
			\omega_L^X(\pom \setminus B(x, r)) 
			\leq 
			M \bigg( \frac{|X-x|}{r} \bigg)^\alpha,
			\end{equation}
			for all $x \in \pom$, $0 < r < \diam(\pom)$, and $X \in \Omega$.
			
			\item 
			Given $\alpha \in (0, 1)$, we say that $\Omega$ satisfies the \textit{local elliptic measure decay property} with exponent $\alpha$ (abbreviated to $\mathrm{LEMD}(\alpha)$) at $x \in \pom$ with scale $0 < r < \diam(\pom)$ if there exists some $M>0$ such that for every $X \in \Omega \cap B(x, r)$
			\begin{equation}\label{cond:LEMD}
			\omega_{L, \Omega \cap B(x, r)}^X(\overline{\Omega} \cap \partial B(x, r)) 
			\leq 
			M \bigg( \frac{|X-x|}{r} \bigg)^\alpha.
			\end{equation}
		\end{enumerate}
	\end{definition}
	
	{
	\begin{remark}
		We are going to show that the solvability of \eqref{eq:problem} along with \eqref{eq:bounds_CDC} implies $\mathrm{GEMD}(\alpha)$, as defined above. But in the proof of $\mathrm{GEMD}(\alpha) \implies \text{CDC}$, the only relevant case is the estimate \eqref{cond:GEMD} for $X$ close enough to $x$, say $X\in \Omega \cap B(x,\frac14 r)$  (in fact, upon changing the value of $M$, the $\mathrm{GEMD}(\alpha)$ condition is vacuous for $X\notin \Omega \cap B(x,\frac14 r)$).
		Then \eqref{cond:GEMD} is the same as \eqref{eq:elliptic_annuli} in Lemma \ref{lem:8r4r}, which is the key property allowing us to obtain the $\alpha$-H\"older-solvability in Sections \ref{sec:well} and \ref{sec:ill}. 		
		We also remark that, whenever $\Omega$ is regular, by the maximum principle, \eqref{cond:LEMD} implies \eqref{cond:GEMD} for $X \in \Omega \cap B(x, r)$. 
	\end{remark}

	\begin{remark}\label{rmk:LEMDbdrHolder}
		$\mathrm{LEMD}(\alpha)$ implies boundary H\"older continuity as in Lemma \ref{lem:Holder} with the same H\"older exponent $\alpha$ and constant. Indeed, fix any ball $B= B(x,r)$ with $x\in \pom$ and $0<r< \diam(\pom)$, and suppose that $0 \le u \in W^{1,2}_{\loc}(2B \cap \Omega) \cap \mathscr{C}(\overline{2B \cap \Omega})$ satisfies $Lu=0$ in the weak sense in $2B \cap \Omega$, with $u \equiv 0$ in $2\Delta= 2B \cap \pom$. Then, for every $X\in B\cap \Omega$, we have  
		\begin{equation*}
			u(X) = \int_{\partial(\Omega \cap B)} u|_{\partial(\Omega \cap B)} \, d\omega^X_{L, \Omega \cap B } \leq  \omega^X_{L, \Omega \cap B} (\overline{\Omega} \cap \partial B) \sup_{\overline{\Omega \cap B}} u
			\leq M \left( \frac{|X-x| }{r} \right)^{\alpha} \sup_{\overline{\Omega \cap B}} u.
		\end{equation*} 	 	
		We are going to prove in this section that $\dot{\mathscr{C}}^\alpha(\pom)$-solvability of \eqref{eq:problem} along with \eqref{eq:bounds_CDC} imply $\mathrm{LEMD}(\alpha)$.\footnote{To be more precise, our argument below shows $\mathrm{LEMD}(\alpha)$ for the pair $a$ and $r$ satisfying the first condition in \eqref{eq:geom_cond}. However, when the second condition in \eqref{eq:geom_cond} holds, we have $\overline{\Omega} \cap \partial B(a,r) = \emptyset$ and thus \eqref{cond:LEMD} holds trivially.} Therefore, these two conditions together yield the boundary H\"older continuity property with the same H\"older exponent $\alpha$. 
	\end{remark}
	}
	
	We can now proceed to prove the main result of this section. 
	It will be convenient to split the proof in two cases, depending whether or not the elliptic measure is a probability.
	
	\subsection{Proof of Theorem~\ref{th:necessity_CDC} in the case $\Omega$ is bounded or $\Omega$ is unbounded with $\pom$ being unbounded}\label{section:CDC-proof}
	
	Fix $a \in \pom$ and $0 < r < \diam(\pom)$. Our goal is to prove the CDC estimate 
	\begin{equation} \label{eq:CDC}
	\frac{\Cap(\overline{B(a,r)}\setminus \Omega, B(a,2r))}{\Cap(\overline{B(a,r)}, B(a,2r))} \geq c,
	\end{equation}
	with $c$ independent on $a$ and $r$.
	We proceed in a series of steps for clarity of the exposition. Recall that by Lemma~\ref{lem:regular} we may assume that $\Omega$ is regular.
	
	\textbf{Step 1: Solvability of \eqref{eq:problem} with bounds \eqref{eq:bounds_CDC} implies $\mathrm{GEMD}(\alpha)$.}
	
	Fix any $x \in \pom$ and $0 < \rho < \diam(\pom)$. Set, for $\xi \in \pom$ and $X \in \Omega$,
	\begin{equation*}
	\phi(\xi) := \min \bigg\{ \bigg( \frac{\abs{\xi - x}}{\rho} \bigg)^\alpha, 1 \bigg\} 
	\quad\text{ and }\quad 
	u_\phi(X) := \int_{\pom} \phi \, d\omega_L^X.
	\end{equation*} 
	The upper bound in \eqref{eq:bounds_CDC} gives 
	\begin{equation*}
	\frac{|u_\phi(X) - u_\phi(Y)|}{|X-Y|^\alpha} 
	\leq 
	C_\alpha \|\phi\|_{\dot{\mathscr{C}}^\alpha(\pom)}
	\leq 
	C_\alpha \rho^{-\alpha},
	\qquad 
	X, Y \in \Omega.
	\end{equation*}
	Letting $Y \to x$ and using the regularity of $\Omega$ (because $\phi$ is clearly bounded and uniformly continuous for it belongs to $\dot{\mathscr{C}}^\alpha(\pom)$), we obtain 
	\begin{equation} \label{eq:u_varphi}
	0\leq u_\phi(X) \leq C_\alpha \left( \frac{\abs{X-x}}{\rho} \right)^\alpha.
	\end{equation}
	Note that $\mathbf{1}_{\pom \setminus B(x, \rho)} (\xi) \leq \phi(\xi)$ for every $\xi \in \pom$. Hence, \eqref{eq:u_varphi} implies 
	\begin{equation*}
	\omega_L^X (\pom \setminus B(x, \rho)) 
	=
	\int_{\pom} \mathbf{1}_{\pom \setminus B(x, \rho)} \, d\omega_L^X 
	\leq 
	\int_{\pom} \phi \, d\omega_L^X
	=
	u_\phi(X)
	\leq
	C_\alpha \left( \frac{|X-x|}{\rho} \right)^\alpha.
	\end{equation*}
	This shows the $\mathrm{GEMD}(\alpha)$ as desired.

	
	\textbf{Step 2: $\mathrm{GEMD}(\alpha)$ implies the geometrical condition:}
	\begin{align} \label{eq:geom_cond}
	\text{ either } \pom \cap S(a; cr, r) \neq \emptyset \quad\text{ or }\quad S(a; cr, r) \subset \ree \setminus \overline{\Omega},
	\end{align}	
	for some constant $c \in (0, 1/2)$ independent of $a$ and $r$, where $S(a; cr, r)$ denotes the closed annulus
	\begin{equation*}
	S(a; cr, r) := \{ X \in \ree : cr \leq \abs{X-a} { \leq } r \}. 
	\end{equation*}

	We extend the proof of \cite[Lemma 1]{A} for the Laplacian to the case of general elliptic operators (we remark that the second case of \eqref{eq:geom_cond} does not occur in \cite{A} due to the connectedness assumption on $\Omega$). Fix $\kappa \in (0, 1/2)$ small enough so that the $\mathrm{GEMD}(\alpha)$ condition gives us 
	\begin{equation}\label{LXB}
	\omega_L^X(\pom \setminus B(a, r)) \leq 1/3,
	\qquad X \in \Omega \cap \overline{B(a, \kappa r)}.
	\end{equation} 
	Note that $\kappa$ is independent of $a$ and $r$.
	
	Let us prove \eqref{eq:geom_cond} with $c := \kappa (c_0^2/3)^{1/(n-1)} \in (0, \kappa)$, by contradiction. Indeed, assume $\pom \cap S(a; cr, r) = \emptyset$ and $S(a; cr, r) \not\subset \ree \setminus \overline{\Omega}$. The former implies (simply because $S(a; cr, r)$ is connected) that either $S(a; cr, r) \subset \Omega$ or $S(a; cr, r) \subset \ree \setminus \overline{\Omega}$. Since the latter case has already been ruled out, we are left with $S(a; cr, r) \subset \Omega$. Let us obtain a contradiction in this case.
	
	Now, let $v \in \mathscr{C}(\Omega) \cap W^{1, 2}(\Omega)$ be the solution to $Lv = 0$ in $\Omega$ with boundary data $\mathbf{1}_{B(a, r) \cap \pom}$, that is, $v(X) = \omega_L^X(B(a, r) \cap \pom)$ for $X \in \Omega$. Note that since $S(a; cr, r) \subset \Omega$, the connected components of $\pom$ either lie inside $B(a, cr)$ or outside of $\overline{B(a,r)}$. Thus we have that $\mathbf{1}_{B(a, r) \cap \pom}$ is uniformly continuous on $\pom$ and hence $v \in \mathscr{C}(\overline{\Omega})$ by the regularity of $\Omega$.

	For $X \in \Omega$ such that $\abs{X-a} = cr$, \eqref{eq:fund_sol} yields 
	\begin{equation*}
	\frac{\Gamma_L(X, a)}{c_0(cr)^{1-n}} \geq 1 \geq \omega_L^X(B(a, r) \cap \pom) = v(X).
	\end{equation*}	
	Moreover, in view of $S(a; cr, r) \subset \Omega$, any $x \in \pom \setminus B(a, cr)$ also satisfies $x \in \pom \setminus \overline{B(a, r)}$, and this yields
	\begin{equation*}
	\frac{\Gamma_L(x, a)}{c_0(cr)^{1-n}} > 0 = v(x).
	\end{equation*}		
	Moreover, in the case that $\Omega$ is unbounded, \eqref{eq:fund_sol} and Lemma~\ref{lem:decay_unbounded} give
	\begin{equation*}
	\dfrac{\Gamma_L(X, a)}{c_0(cr)^{1-n}} \longrightarrow 0 
	\quad\text{ and }\quad 
	\omega_L^X(B(a, r) \cap \pom) \longrightarrow 0
	\quad \text{as } X \to \infty.
	\end{equation*}
	With all the above estimates, the maximum principle in $\Omega \setminus \overline{B(a, cr)}$ yields
	\begin{equation}\label{GLX}
	\frac{\Gamma_L(X, a)}{c_0(cr)^{1-n}}
	\geq 
	\omega_L^X(B(a, r) \cap \pom),
	\qquad 
	X \in \Omega \setminus \overline{B(a, cr)}.
	\end{equation}	
	Now take $Y \in \ree$ such that $|Y-a|=\kappa r$. Then, since $0<c<\kappa<\frac12$, we have $Y \in S(a; cr, r)$, hence $Y \in \Omega \cap \overline{B(a, \kappa r)}$. Therefore, invoking \eqref{LXB},  \eqref{GLX}, and \eqref{eq:fund_sol}, we obtain 
	\begin{equation} \label{eq:contradiction}
	\frac13 
	\geq 
	\omega_L^Y(\pom \setminus B(a, r))
	\geq 
	1 - \frac{\Gamma_L(Y, a)}{c_0(cr)^{1-n}}
	\geq 
	1 - c_0^{-2}\left( \frac{\kappa}{c} \right)^{1-n}
	=
	\frac23,
	\end{equation}
	which is a contradiction. This finally shows that the dichotomy \eqref{eq:geom_cond} holds. 
	
	
	\textbf{Step 3: The second condition in \eqref{eq:geom_cond} implies CDC.}
	
	Let us briefly show how the condition $S(a; c r, r) \subset \ree \setminus \overline{\Omega}$ in \eqref{eq:geom_cond} implies \eqref{eq:CDC} for this pair of $a$ and $r$. Indeed, if $S(a; cr, r) \subset \ree \setminus \overline{\Omega}$, then we may find a ball 
	\[ B(X_0, cr/4) \subset S(a; cr, r) \subset \ree \setminus \overline{\Omega} \subset \ree \setminus \Omega. \] This implies, by the monotonicity of capacity, that 
	\begin{equation*}
	\frac{\Cap(\overline{B(a,r)}\setminus \Omega, B(a,2r))}{\Cap(\overline{B(a,r)}, B(a,2r))}
	\geq 
	\frac{\Cap(B(X_0, cr/4), B(X_0, 3r))}{\Cap(\overline{B(a,r)}, B(a,2r))}
	\approx_c
	1,
	\end{equation*}
	where in the last estimate we have used the well known fact that $\Cap(B(Y, R); B(Y, CR)) \approx_C R^{1-n}$ for any $Y \in \ree$, $R > 0$ and $C > 1$ (see for instance \cite[Chapter 2]{HKM}).	Since by the previous steps, $c$ does not depend on $a$ nor $r$, we have shown \eqref{eq:CDC} in the case that $S(a; c r, r) \subset \ree \setminus \overline{\Omega}$ in \eqref{eq:geom_cond}. We are therefore left with the case $\pom \cap S(a; cr, r) \neq \emptyset$, which we handle in the next steps.
	
	
	\textbf{Step 4: $\mathrm{GEMD}(\alpha)$ and the first condition in \eqref{eq:geom_cond} imply $\mathrm{LEMD}(\alpha)$ at $a$ with scale $r$}.
	Assuming that the $\mathrm{GEMD}(\alpha)$ condition is satisfied and that $\pom \cap S(a; cr, r) \neq \emptyset$, we claim that there exists $\varepsilon > 0$, independent on $a$ and $r$, such that 
	\begin{equation} \label{eq:claim_bourgain_like}
	\omega_L^X (\pom \setminus B(a, cr/2)) \geq \varepsilon
	\qquad 
	\text{for every } X \in \Omega \cap \partial B(a, r). 
	\end{equation}
	Assuming this momentarily, let us see how \eqref{eq:claim_bourgain_like} implies $\mathrm{LEMD}(\alpha)$ at $a$ and scale $r$, with uniform constants on $a$ and $r$. Choose $\phi \in \mathscr{C}^\infty(\ree)$ such that $\mathbf{1}_{\ree \setminus B(a, cr/2)} \leq \phi \leq \mathbf{1}_{\ree \setminus B(a, cr/4)}$. Let 
	\[
	u_\phi(X) := \int_{\pom} \phi \, d\omega_L^X
	\] 
	be the associated solution in $\Omega$. By regularity of $\Omega$, $u_\phi \in \mathscr{C}(\overline{\Omega})$. Moreover, the choice of $\phi$ along with \eqref{eq:claim_bourgain_like} yield  
	\begin{equation*}
	u_\phi(Y) 
	\geq 
	\omega_L^Y (\pom \setminus B(a, cr/2)) \geq \varepsilon, 
	\qquad 
	\text{for every } Y \in \Omega \cap \partial B(a, r). 
	\end{equation*} 
	Moreover, for $y \in \pom \cap \partial B(a, r)$, we have, by continuity of $u_\phi$ up to the boundary, $u_\phi (y) = \phi(y) = 1 \geq \varepsilon$. Thus we have shown 
	\begin{equation*}
	u_\phi(Y) 
	\geq 
	\varepsilon
	\qquad 
	\text{for every } Y \in \overline{\Omega} \cap \partial B(a, r). 
	\end{equation*}
	Hence, for any $X \in \Omega \cap B(a, r)$,
	\begin{multline} \label{eq:comp_claim_1}
	\omega_{L, \Omega \cap B(a, r)}^X (\overline{\Omega} \cap \partial B(a, r)) 
	=
	\int_{\partial(\Omega \cap B(a, r))} \mathbf{1}_{\overline{\Omega} \cap \partial B(a, r)}(y) \,d\omega_{L, \Omega \cap B(a, r)}^X (y)
	\\ 
	\leq 
	\int_{\partial(\Omega \cap B(a, r))} \varepsilon^{-1} u_\phi(y) \,d\omega_{L, \Omega \cap B(a, r)}^X (y)
	=
	\varepsilon^{-1} u_\phi(X),
	\end{multline}
	where in the last step we have used the uniqueness of classical solutions in $\Omega \cap B(a, r)$ and the fact that $u_\phi$ is a solution in $\Omega \cap B(a, r)$, continuous in the whole $\overline{\Omega}$. 
	
	On the other hand, the $\mathrm{GEMD}(\alpha)$ condition gives that for any $X \in \Omega \cap B(a, r)$,
	\begin{multline} \label{eq:comp_claim_2}
	u_\phi(X)
	=
	\int_{\pom} \phi \,d\omega_L^X
	\leq 
	\int_{\pom} \mathbf{1}_{\pom \setminus B(a, cr/4)} \, d\omega_L^X
	\\ =
	\omega_L^X (\pom \setminus B(a, cr/4))
	\leq 
	M \bigg( \frac{|X-a|}{cr/4} \bigg)^\alpha
	=
	\frac{4^\alpha M}{c^\alpha} \bigg( \frac{|X-a|}{r} \bigg)^\alpha.
	\end{multline}
	Therefore, \eqref{eq:comp_claim_1} and \eqref{eq:comp_claim_2} imply the desired $\mathrm{LEMD}(\alpha)$ at $a$ and scale $r$ with constant $4^\alpha M \varepsilon^{-1} c^{-\alpha}$, which is indeed independent on $a$ and $r$ by the previous steps. 
	
	It remains to prove our claim \eqref{eq:claim_bourgain_like}. The proof is similar to that of \cite[Lemma 2]{A}. Since $\pom \cap S(a; cr, r) \neq \emptyset$, there exist $\rho \in [cr, r)$ and $x^* \in \pom \cap \partial B(a, \rho)$. Morally speaking, the existence of such boundary point $x^*$ away from $B(a, cr/2)$ attracts the Brownian traveller and this is what gives us a non-zero lower bound in \eqref{eq:claim_bourgain_like}. The $\mathrm{GEMD}(\alpha)$ condition gives 
	\begin{equation} \label{eq:complement_GEMD}
	\omega_L^X (\pom \setminus B(x^*, \rho/4)) 
	\leq 
	M \bigg( \frac{|X - x^*|}{\rho/4} \bigg)^\alpha,
	\qquad 
	X \in \Omega \cap B(x^*, \rho/4).
	\end{equation}
	To proceed, choose $\psi \in \mathscr{C}^\infty(\ree)$ such that $\mathbf{1}_{\ree \setminus B(a, 3\rho/4)} \leq \psi \leq \mathbf{1}_{\ree \setminus B(a, \rho/2)}$, and let  
	\[
	u_\psi(X) := \int_{\pom} \psi \, d\omega_L^X, \qquad X \in \Omega, 
	\] 
	be the associated solution in $\Omega$, so that $u_\psi \in \mathscr{C}(\overline{\Omega})$ by regularity of $\Omega$. Then, noting that $\abs{x^* - a} = \rho$ implies $B(a, 3\rho/4) \subset \ree \setminus B(x^*, \rho/4)$, we use \eqref{eq:complement_GEMD} to obtain that for all $X \in \Omega \cap B(x^*, \rho/4)$, 
	\begin{equation*}
	\omega_L^X(\pom \cap B(a, 3\rho/4))
	\leq 
	\omega_L^X (\pom \setminus B(x^*, \rho/4)) 
	\leq 
	M \bigg( \frac{|X - x^*|}{\rho/4} \bigg)^\alpha,
	\end{equation*}	
	whence for some $0 < d < 1/12$ small enough, independent on $r, \rho$ and $x^*$, there holds 
	(recalling that $\psi \equiv 1$ on $\pom \setminus B(a, 3\rho/4)$)
	\begin{equation} \label{eq:measure_12}
	u_\psi(X) 
	\geq 
	\omega_L^X(\pom \setminus B(a, 3\rho/4))
	\geq
	1/2,
	\qquad 
	X \in \Omega \cap B(x^*, d\rho/4).
	\end{equation}
	
	Define  
	\begin{equation*}
	u(X) =
	\begin{cases}
	u_\psi(X) \quad & \text{ if } X \in \Omega, \\
	1 & \text{ if } X \notin \Omega. 
	\end{cases}
	\end{equation*}
	It is well known that $u$ is a supersolution to $Lu = 0$ in $\ree \setminus B(a, 3\rho/4)$ (see \cite[Lemma 7.9]{HKM}), and it is clear that $0 \leq u \leq 1$ in $\ree$. Also, \eqref{eq:measure_12} implies that $u \geq 1/2$ in $B(x^*, d\rho/4)$. By regularity of $\Omega$, it is not hard to see that $u \in \mathscr{C}(\ree \setminus B(a, 3\rho/4))$.  
	
	Take now any $x_1^* \in \partial B(a, \rho)$ with $\abs{x_1^* - x^*} < \rho/12$,  
	and denote $B := B(x_1^*, \frac{\rho}{12} + \frac{d\rho}{4})$. These clearly imply that $B(x^*, d\rho/4) \subset B$, and so $|\{ X \in B : u(X) \geq 1/2 \}| \geq |B(x^*, d\rho/4)| \approx |B|$, with constants depending only on $n$ and $d$. Moreover, using the facts that $|x_1^* - a| = \rho$ and $d < 1/12$, we easily get that $2B \subset \ree \setminus B(a, 3\rho/4)$. These allow us to apply (a trivial modification of) \cite[Theorem 4.9]{HanLin} to deduce
	\begin{equation*}
	u(X) \geq \delta,
	\qquad 
	X \in B(x_1^*, d\rho/4),
	\end{equation*}
	for some $\delta > 0$ independent on $x^*, x_1^*$ and $\rho$, because $B(x_1^*, d\rho/4) \subset \frac12 B$ since $d < 1/12$.
	Then repeat the argument using $x_1^*$ in place of $x^*$, and repeat until we have covered $\partial B(a, \rho)$ by a finite amount (depending only on $d$ and $n$) of balls $B(x^*, d\rho/4), B(x_1^*, d\rho/4), \ldots , B(x_k^*, d\rho/4)$, ultimately obtaining 
	\begin{equation*}
	u(X) \geq \varepsilon,
	\qquad 
	\text{ for all } X \in \partial B(a, \rho),
	\end{equation*}
	for some $\varepsilon > 0$ independent on $x^*$, $a$ and $\rho$. This implies, by definition of $u$,
	\begin{equation*}
	u_\psi(X) \geq \varepsilon,
	\qquad 
	X \in \Omega \cap \partial B(a, \rho).
	\end{equation*}
	Therefore, recalling that $cr \leq \rho$ and that $\psi \equiv 0$ on $\pom \cap B(a, \rho/2)$, 
	we obtain that for all $X \in  \Omega \cap \partial B(a, \rho)$, 
	\begin{equation} \label{eq:greater_eps}
	\omega_L^X (\pom \setminus B(a, cr/2)))
	\geq
	\omega_L^X(\pom \setminus B(a, \rho/2))
	\geq
	u_\psi(X) 
	\geq 
	\varepsilon.
	\end{equation}
	
{ 
Let $\Omega':=\Omega  \setminus \overline{B(a,\rho)}$. Note that $v(Y):= 	\omega^Y_L(B(a, cr/2))$, $Y\in\Omega$, clearly satisfies $v\in\mathscr{C}(\overline{\Omega'})$ since $B(a, cr/2)\subset B(a, \rho/2)$. Therefore, if we let 
\begin{equation} \label{eq:v_hat}
\widehat{v}(Y):=\int_{\pom'} v(z)\, d\omega^Y_{L, \Omega'} (z),
\qquad
Y\in\Omega',
\end{equation}
it follows that  $\widehat{v}\in\mathscr{C}(\overline{\Omega'})$ with $\widehat{v}\equiv v$ in $\pom'$. Since both $v$ and $\widehat{v}$ are $L$-solutions in $\Omega'$, the maximum principle yields that $\widehat{v}\equiv v$ in $\Omega'$ ---here we note that when $\Omega$ is unbounded both solutions decay at infinity by Lemma~\ref{lem:decay_unbounded}, hence we can legitimately invoke the maximum principle.
As a result, since $\rho<r$, for any $Y\in\Omega\cap \partial B(a,r)\subset\Omega'$ we obtain 
	\[ 
	\omega^Y_L(B(a, cr/2)) 
	= 
	\int_{\Omega \cap \partial B(a,\rho)} \omega_L^X(B(a, cr/2)) \, d\omega^Y_{L, \Omega \setminus B(a,\rho)} (X) \leq \sup_{X\in \Omega\cap \partial B(a,\rho)} \omega^X_L(B(a, cr/2)). \]
	This implies that for all $Y\in \Omega \cap \partial B(a,r)$, using \eqref{eq:greater_eps}:
	\[ \omega_L^Y (\pom \setminus B(a, cr/2)))  \geq \inf_{X\in \Omega\cap \partial B(a,\rho)} \omega_L^X(\pom\setminus B(a, cr/2)) \geq \varepsilon.  \] }
	This finishes the proof of the claim \eqref{eq:claim_bourgain_like}.	
	
	
	\textbf{Step 5: $\mathrm{LEMD}(\alpha)$ implies CDC.}
	This is an generalization of \cite[Lemma 3]{Ancona}. For the sake of convenience, set $F := \overline{B(a, r)} \setminus \Omega$ and $v := \mathcal{R}_L(F; B(a, 2r))$ the capacitary potential in $B(a, 2r)$ of the compact set $F$ with respect to the operator $L$. Then, Lemma \ref{lem:CP} along with the $\mathrm{LEMD}(\alpha)$ condition at $a$ and scale $r$ (obtained in Step 3) gives that for any $X \in \Omega \cap B(a, r/N)$, for $N \gg 1$, 
	\begin{equation*}
	v(X)
	\geq 
	1 - \omega_{L, \Omega \cap B(a, r)}^X (\partial B(a, r) \cap \overline{\Omega})
	\geq 
	1 - M \left( \frac{\abs{X-a}}{r} \right)^\alpha 
	\ge 
	1 - \frac{M}{N^\alpha}
	=: \eta > 0.
	\end{equation*}
	where $\eta$ and $N$ are independent of $a$ and $r$ by the previous steps. Moreover, since $B(a, r/N) \setminus \Omega \subset F$, we actually obtain  
	\begin{equation*}
	\eta^{-1} v(X) \geq 1, 
	\qquad 
	X \in B(a, r/N). 
	\end{equation*}
	By the definition of capacity and the ellipticity of $A$, we deduce 
	\begin{equation} \label{eq:cap1}
	\text{Cap}(B(a, r/N); B(a, 2r))
	\leq 
	\int_{B(a, 2r)} \abs{\nabla (\eta^{-1}v)}^2 dX
	\lesssim 
	\eta^{-2} \int_{B(a, 2r)} A \nabla v \cdot \nabla v \, dX.
	\end{equation}
	Observe that the results in \cite[Section 4]{LSW} (or \cite[Chapter 8]{HKM}) assert that $v =  \mathcal{R}_L(F; B(a, 2r))$ is precisely the function minimizing the energy $\int_{B(a, 2r)} A \nabla v \cdot \nabla v \, dX$ among those which are at least 1 on $F$. This means 
	\begin{multline} \label{eq:cap2}
	\int_{B(a, 2r)} A \nabla v \cdot \nabla v \, dX
	=
	\inf_{\substack{w \in W^{1, 2}_0(B(a, 2r)) \\ w \geq 1 \text{ on } F}} \int_{B(a, 2r)} A \nabla w \cdot \nabla w \, dX
	\\ \lesssim 
	\inf_{\substack{w \in W^{1, 2}_0(B(a, 2r)) \\ w \geq 1 \text{ on } F}} \int_{B(a, 2r)} \abs{\nabla w}^2 \, dX
	=
	\text{Cap}(F; B(a, 2r)).
	\end{multline}
	From \eqref{eq:cap1}, \eqref{eq:cap2}, and the fact that 
	\begin{equation*}
	\text{Cap}(B(a, r/N), B(a, 2r)) \approx_N r^{1-n} \approx \text{Cap}(\overline{B(a, r)}, B(a, 2r)),
	\end{equation*}
	we immediately get 
	\begin{equation*}
	\text{Cap}(F, B(a, 2r)) \gtrsim_N \varepsilon^2 \, \text{Cap}(\overline{B(a, r)}, B(a, 2r)),
	\end{equation*}
	which shows \eqref{eq:CDC} with constants depending on $N$ and $\varepsilon$, so independent on $a$ and $r$, as desired. This completes the proof of Theorem~\ref{th:necessity_CDC} in the case $\Omega$ is bounded or $\Omega$ is unbounded with $\pom$ unbounded.  	
	
	\subsection{Proof of Theorem~\ref{th:necessity_CDC} in the case $\Omega$ unbounded with $\pom$ being bounded.}
	Let $\Omega$ be an unbounded open set such that $\pom$ is bounded. As in the previous proof, for some given $a \in \pom$ and $0 < r < \diam(\pom)$, we want to establish the CDC estimate \eqref{eq:CDC} with constants uniform on $a$ and $r$. By Lemma~\ref{lem:regular}, we may assume that $\Omega$ is regular. Let us sketch how a slight modification of the argument in the previous proof gives the result also in this case. The main difference is that in the current scenario, the elliptic measure is not a probability. But in fact, in our points of interest, it resembles a probability, so the proof will run similarly. Let us give some more details.
	
	By translation, we may assume that $0 \in \pom$ and so $\pom \subset B(0, 3 \diam(\pom))$. We claim that there exists some $c^* > 0$ (depending only on $n$, the ellipticity of $L$, $\alpha$, and $C_\alpha$ in \eqref{eq:bounds_CDC}) such that 
	\begin{equation} \label{eq:claim_non_degeneracy}
	\omega_L^X(\pom) \geq c^*,
	\qquad 
	\text{ for any } X \in \Omega \cap B(0, 100\diam(\pom)).
	\end{equation}
	The choice of $\Omega \cap B(0, 100\diam(\pom))$ is easily motivated by the previous subsection: the only points in the proof where we pass to the complement when using the elliptic measure are \eqref{eq:contradiction} and \eqref{eq:measure_12}, and there the points involved clearly lie in $\Omega \cap B(0, 100\diam(\pom))$.
	
	Assuming this momentarily, let us see which steps in the previous proof need a change. Step 1 works without any change. In Step 2 we just need to change the choice of $\kappa$ so that \eqref{LXB} reads 
	\begin{equation*}
	\omega_L^X(\pom \setminus B(a, r)) \leq c^*/3,
	\qquad X \in \Omega \cap B(a, \kappa r).
	\end{equation*} 
	Then choose $c := \kappa (c_0^2 c^* / 3)^{1/(n-1)}$ so that \eqref{eq:contradiction} is replaced by the following
	\begin{multline*}
	\frac{c^*}{3} 
	\geq 
	\omega_L^Y(\pom \setminus B(a, r))
	=
	\omega_L^Y(\pom) - \omega_L^Y(\pom \cap B(a, r))
	\\
	\geq 
	c^* - \frac{\Gamma_L(Y, a)}{c_0(cr)^{1-n}}
	\geq 
	c^* - c_0^{-2}\left( \frac{\kappa}{c} \right)^{1-n}
	=
	\frac{2c^*}{3},
	\end{multline*}
	with which we get a contradiction in any case. Step 3 works also without further changes. 	
	In Step 4, we first need to change the choice of $d$ so that \eqref{eq:measure_12} reads $u_\psi(X) \geq c^*/2$. The rest of the proof 
	{up to \eqref{eq:greater_eps} proceeds in the same way, with constants depending on the harmless constant $c^*$. After that, we need to change our choice of $v$ so that $v(Y) := \omega_L^Y (\pom \setminus B(a, cr/2))$. We also define $\widehat{v}$ much as in \eqref{eq:v_hat} and observe that both $v$ and $\hat{v}$ are continuous in $\overline{\Omega'}$, decay at infinity by Lemma~\ref{lem:decay_unbounded} (because $\pom$ is bounded), and coincide in $\pom'$, where we recall that $\Omega':=\Omega  \setminus \overline{B(a,\rho)}$. Thus the maximum principle asserts that $v \equiv \widehat{v}$ in $\Omega'$. This means that, if $Y \in \partial B(a, r) \cap \Omega \subset \Omega'$, we have
	\begin{multline*}
	\omega^Y_L(\pom \setminus B(a, cr/2)) 
	= 
	\int_{\partial \Omega' } \omega_L^X(\pom \setminus B(a, cr/2)) \, d\omega^Y_{L, \Omega'} (X) 	
	\\ \geq 
	\omega^Y_{L, \Omega'} (\partial\Omega' ) \inf_{X\in \partial \Omega'} \omega^X_L(\pom \setminus B(a, cr/2)). 
	\end{multline*}
	Note that $\omega^X_L(\pom \setminus B(a, cr/2)) = 1$ for $X \in \partial \Omega \setminus B(a, \rho)$, and, by \eqref{eq:greater_eps},  $ \omega^X_L(\pom \setminus B(a, cr/2))\geq \varepsilon$ for $X \in \partial B(a, \rho) \cap \Omega$. This clearly implies that 
	\begin{equation*}
		\inf_{X\in \partial \Omega' } \omega^X_L(\pom \setminus B(a, cr/2)) 
		\geq 
		\varepsilon.
	\end{equation*}
	On the other hand, since $\Omega'=\Omega \setminus\overline{ B(a, \rho) }\subset \ree \setminus \overline{ B(a, \rho) }=:\Omega''$, by the maximum principle (its use is justified similarly as above) in $\Omega' $ we obtain, for $Y \in \partial B(a, r) \cap \Omega \subset \Omega'$,
	\begin{equation*}
	\omega^Y_{L, \Omega' } (\partial\Omega') 
	\geq 
	\omega^Y_{L, \Omega''} (\partial\Omega'').
	\end{equation*}	
	To continue, we next observe that $\Omega''=\ree \setminus \overline{ B(a, \rho) }$ satisfies the CDC with a constant independent on $\rho$, whence  there exists $\widehat{\alpha}$ (depending only on $n$ and ellipticity) such that $\dot{\mathscr{C}}^{\widehat{\alpha}}$-solvability holds for $L$ in $\Omega''$, see Theorem \ref{thm:Ca-ill}  and Section~\ref{sec:ill}. This means that we can then repeat \textit{mutatis mutandis} the argument leading to \eqref{eq:claim_non_degeneracy} in $\Omega''$ to conclude that for every $N\ge 100$ there exists $c_N > 0$ depending on ellipticity and dimension such that
	\begin{equation*} 
	\omega_{L,\Omega''}^X(\pom'') \geq c_N,
	\qquad 
	\text{ for any } X \in \Omega'' \cap B(a, N\diam(\pom'')).
	\end{equation*}
	As a consequence, recalling that $cr\le \rho < r$, we obtain	$\omega_{L,\Omega''}^Y(\pom'') \gtrsim 1$ for every $Y \in\partial B(a,r)$.
	Collecting all these observations we have shown that 
	\begin{equation*}
		\omega^Y_L(\pom \setminus B(a, cr/2)) \gtrsim \varepsilon, \qquad \text{ for } Y \in \partial B(a, r) \cap \Omega,
	\end{equation*}
	with harmless implicit constants. With this in hand we can then complete Step 4.} Finally, the argument in Step 5 can be carried out without any modification because we always work in bounded domains, where the elliptic measure is a probability. Therefore, the proof will be complete once we obtain  \eqref{eq:claim_non_degeneracy}.
	
	Let us next prove the claim \eqref{eq:claim_non_degeneracy}. Take $x_0, y_0 \in \pom$ so that $\abs{x_0 - y_0} = \diam(\pom)$ (these points must be different because $\diam(\pom) > 0$). Define $\psi(x) := \abs{x-y_0}^\alpha / \diam(\pom)^\alpha$ so that 
	\begin{align*}
	\psi(x_0) = 1, \quad 
	\psi(y_0) = 0, \quad \text{and}\quad 
	\psi \in \dot{\mathscr{C}}^\alpha(\pom) \quad \text{with}\quad 
	\|\psi\|_{\dot{\mathscr{C}}^\alpha(\pom)} \leq \diam(\pom)^{-\alpha}.  
	\end{align*} 
	Associated with this boundary datum $\psi$, we construct the solution as in \eqref{eq:ill-solution}: 
	\begin{equation*}
	u_1(X) = u_1(X; y_0, \psi) 
	= 
	\psi(y_0) + \int_\pom (\psi(y) - \psi(y_0)) d\omega_L^X(y),
	\end{equation*}
	which, by \eqref{eq:bounds_CDC}, satisfies
	\begin{equation}\label{4a-1}
	\|u_1\|_{\dot{\mathscr{C}}^\alpha(\Omega)} 
	\le C_{\alpha} 
	\|\psi\|_{\dot{\mathscr{C}}^\alpha(\pom)}
	\le C_{\alpha} 
	\diam(\pom)^{-\alpha}.
	\end{equation} 
	In a similar fashion, we can also define another solution
	\begin{equation*}
	u_2(X) = u_2(X; x_0, 1-\psi) 
	= 
	(1-\psi)(x_0) + \int_\pom [(1-\psi)(y) - (1-\psi)(x_0)] d\omega_L^X(y),
	\end{equation*}
	which satisfies 
	\begin{align}\label{4a-2}
	\|u_2\|_{\dot{\mathscr{C}}^\alpha(\Omega)}
	\le C_{\alpha} \|1-\psi\|_{\dot{\mathscr{C}}^\alpha(\pom)} 
	= C_{\alpha} \|\psi\|_{\dot{\mathscr{C}}^\alpha(\pom)}
	\le C_{\alpha} \diam(\pom)^{-\alpha}.
	\end{align} 
	Gathering \eqref{4a-1} and \eqref{4a-2}, we obtain that, if we define $u := u_1 + u_2$, then
	\begin{align}\label{wL-Ca}
	u \in \dot{\mathscr{C}}^\alpha(\Omega) 
	\quad\text{with}\quad 
	\|u\|_{\dot{\mathscr{C}}^\alpha(\Omega)} 
	\leq M \diam(\pom)^{-\alpha}, 
	\end{align}
	where the constant $M$ is simply $2C_{\alpha}$. Note that, indeed, $u(X) = \omega_L^X(\pom)$ for any $X \in \Omega$.
	
	Given $X \in \Omega$ with $\delta(X) \leq (2M)^{-1/\alpha} \diam(\pom) =: r_0$, we pick $\widehat{x} \in \pom$ so that $|X-\widehat{x}| = \delta(X)$. Then by \eqref{wL-Ca} and the regularity of $\Omega$, 
	\begin{equation*}
	|\omega_L^X(\pom) - 1|
	=
	|u(X) - u(\widehat{x})|
	\leq 
	M \diam(\pom)^{-\alpha} |X-\widehat{x}|^\alpha 
	\leq 
	\frac12,
	\end{equation*}
	and as a consequence,  
	\begin{equation} \label{eq:non_degenerate_close}
	\omega_L^X(\pom) \geq \frac12 
	\qquad 
	\text{ for any } X \in \Omega \text{ with } \delta(X) \leq r_0.
	\end{equation}	 
	Hence, an application of the maximum principle in $\Omega' := \{X \in \Omega : \delta(X) > r_0 \}$, much as in Step 2 above (using Lemma~\ref{lem:decay_unbounded} to show that $\omega_L^X(\pom) = \omega_L^X(\pom \cap B(0, 3\diam(\pom))) \longrightarrow 0$ as $X \to \infty$), shows that 
	\begin{equation*}
	\omega_L^X (\pom) \geq \frac{\Gamma_L(X, 0)}{2c_0^{-1}r_0^{1-n}},
	\qquad X \in \Omega', 
	\end{equation*}
	which together with \eqref{eq:fund_sol} gives that for $X \in \Omega' \cap B(0, 100\diam(\pom))$, 
	\begin{equation}\label{eq:far}
	\omega_L^X (\pom) 
	\geq 
	c_0^2\frac{(100\diam(\pom))^{1-n}}{2r_0^{1-n}} 
	=
	c_0^2 \frac{[100(2M)^{1/\alpha}]^{1-n}}{2}
	=:
	c^* > 0.
	\end{equation}
	Therefore, \eqref{eq:claim_non_degeneracy} is a consequence of \eqref{eq:non_degenerate_close} and \eqref{eq:far}. 
	
	\begin{remark}\label{remark:equivalences}
	{
	From our proofs we can see that for any regular open set $\Omega$ and for any elliptic operator $L$, solutions of the $\dot{\mathscr{C}}^\alpha$-Dirichlet problem defined by \eqref{eq:well-solution} or \eqref{eq:ill-solution} satisfy the estimate \eqref{eq:bounds_CDC}, if and only if, non-negative solutions with vanishing boundary data enjoy boundary H\"older continuity (as in Lemma~\ref{lem:Holder}) with \emph{almost} the same exponent $\alpha$. 
	
	To see this, take $\alpha \leq \alpha_1$, where $\alpha_1$ is the exponent in the De Giorgi/Nash estimates (cf.~Lemma~\ref{lem:DGN}). First note that $\dot{\mathscr{C}}^\alpha$-solvability along with \eqref{eq:bounds_CDC} implies $\mathrm{LEMD}(\alpha)$ by Steps 1-4 in Section \ref{section:CDC-proof} (noting again that some cases are trivial as explained in the footnote in Remark~\ref{rmk:LEMDbdrHolder}), so by Remark~\ref{rmk:LEMDbdrHolder} we obtain boundary H\"older continuity (as in Lemma~\ref{lem:Holder}) with exponent $\alpha$. 
	
	Conversely, if boundary H\"older continuity (as in Lemma~\ref{lem:Holder}) with exponent $\alpha$ holds, then Lemma~\ref{lem:8r4r} is available (see also Remark~\ref{rmk:bdrHolderdecay2}), and then the machinery in Sections 3 and 4 runs to obtain $\dot{\mathscr{C}}^{\alpha'}$-solvability with bounds \eqref{eq:bounds_CDC}, for any $\alpha'< \alpha$. To conclude, we remark that it is in general not possible to take $\alpha'=\alpha$ by the example in \cite[p. 3]{Mazya}.
	
}

	\end{remark}

	\subsection{Adaptation of the proof for growth functions} \label{subsec:CDC_growth}
	
	Let us provide a brief sketch of how the preceding proof would work with an arbitrary growth function $\varphi \in \mathcal{G}_\beta$.
	
	First, we should replace the quotients $\big(\frac{|X-x|}{r} \big)^\alpha$ in Definition~\ref{def:HMD} by $\frac{\varphi(|X-x|)}{\varphi(r)}$, and in Step 1 we will use accordingly $\phi(\xi) = \min \big\{ \frac{\varphi(|\xi - x|)}{\varphi(\rho)}, 1  \big\}$, which belongs to $\dot{\mathscr{C}}^{\varphi}$ be Lemma \ref{lem:pro-phi}. Then, to finish Step 1, one should also use the fact that the first integral condition in \eqref{eq:ass-phi} implies 
	\begin{align*}
	\varphi(\lambda^{-1} t) \le \frac{1}{\ln \lambda} \int_{\lambda^{-1} t}^t \varphi(s) \, \frac{ds}{s} \leq \frac{C_\varphi}{\ln \lambda} \varphi(t) \quad\text{ for all $t>0$ and $\lambda>1$},
	\end{align*} 
	and the second integral condition in \eqref{eq:ass-phi} gives the subadditivity of $\varphi$ (see Lemma \ref{lem:pro-phi} part \eqref{list-22}): 
	\begin{equation*}
	\varphi(t_1+t_2) \leq 2 C_\varphi \, (\varphi(t_1)+\varphi(t_2)), \quad\text{ for all } \, t_1, t_2 > 0.
	\end{equation*}
	These estimates and the adaptation of Definition~\ref{def:HMD} enable our previous proof to be adapted to the $\dot{\mathscr{C}}^{\varphi}$ setting. Further details are left to the reader.

	
	\section{Applications to other boundary value problems} \label{sec:applications}
	
	In this section, we will explore some other Dirichlet boundary value problems whose well-posedness, under suitable underlying geometrical conditions, can be shown as a consequence of that of the $\dot{\mathscr{C}}^\alpha$-Dirichlet problem that we have extensively studied along the previous sections.
	
	We first introduce the following definition: 
	
	\begin{definition}[{\bf Non-tangential limit}]\label{def:nt} 
		Given $\theta > 1$ and $y \in \pom$, we say that $\restr{u}{\pom}^{{\rm nt. \, lim}, \theta} (y) = f(y)$ if 
		\begin{equation*}
		\lim_{\substack{X \in \Gamma^{\theta} (y) \\ X \to y}} u(X) = f(y).
		\end{equation*}
		Here and elsewhere, $\Gamma^\theta(y) := \{ X \in \Omega : |X-y| < \theta \delta(X) \}$ denotes the cone with vertex $y \in \pom$ and aperture $\theta>1$.
	\end{definition}
	{In this section, we shall impose the boundary value condition
	\[ \restr{u}{\pom}^{{\rm nt. \, lim}, \theta} = f, \, \quad   \sigma\text{-a.e. on } \pom, \]
	for function $f$ in suitable spaces. Compared to \eqref{eq:problem}, we impose a weaker boundary value condition here because the functions $u$ and/or $f$ are only defined in Sobolev or $L^p$ spaces a priori, and it does not make sense to impose a pointwise boundary condition. A posteriori we shall indeed prove that $u=f$ on $\pom$.
	}

	\subsection{The Dirichlet problem with solutions satisfying a fractional Carleson measure condition}\label{sec:Car}
	
	For the first application of the results in the previous sections, we introduce some spaces: 
	
	\begin{definition} 
		Let $\varphi$ be a growth function. Define a (generalized) Carleson type $\varphi$-H\"{o}lder space on an open set $\Omega \subset \R^{n+1}$ as 
		\[
		\mathscr{C}^{*,\varphi}(\Omega) :=\big\{u \in W^{1,2}_{\loc}(\Omega): \|u\|_{\mathscr{C}^{*,\varphi}(\Omega)}
		<\infty \big\},
		\]
		where $\| \cdot \|_{\mathscr{C}^{*,\varphi}(\Omega)}$ stands for
		\begin{equation*}
		\|u\|_{\mathscr{C}^{*,\varphi}(\Omega)}
		:=
		\sup_{\substack{x \in \pom \\ 0<r<\infty}} 
		\frac{1}{\varphi(r)} \bigg(\frac{1}{r^n} \iint_{B(x, r) \cap \Omega} |\nabla u(X)|^2 \delta(X) \, dX \bigg)^{\frac12},
		\end{equation*}
		where here and elsewhere $\delta(X) := \dist(X, \pom)$.
	\end{definition}
	\begin{remark} When $\varphi(r) \equiv 1$, the quantity above is usually referred to as a ``Carleson measure condition'', and it is used to study the absolute continuity of harmonic/elliptic measure (cf. \cite{CHMT, HKMP, HLM, HMMTZ}). 	
	\end{remark}
	
	In very general frameworks, this space is related to $\dot{\mathscr{C}}^\varphi(\Omega)$, as the following result shows.
	
	\begin{lemma}\label{lem:CC}
		Let $\Omega \subset \ree$, $n \ge 2$, be an open set with ADR boundary. Let $L=-\div(A\nabla)$ be a real (non-necessarily symmetric) uniformly elliptic operator. Let $\varphi \in \mathcal{G}_\beta$ for some $\beta \in (0, 1)$.
		Then there exists $C$ (depending only on dimension, the ADR constant, the ellipticity constant of $L$, and $C_\varphi$) such that for any weak solution $u$ of $Lu=0$ in $\Omega$, 
		\begin{equation*}
		\|u\|_{\mathscr{C}^{*,\varphi}(\Omega)} \le C \|u\|_{\dcv(\Omega)}. 
		\end{equation*}
	\end{lemma}
	
	\begin{proof}
		Let $\W=\W(\Omega)$ denote the collection of (closed) dyadic Whitney cubes of $\Omega$, so that the cubes in $\W$ form a pairwise non-overlapping covering of $\Omega$, which satisfy 
		\begin{equation}\label{eq:W1}
		4 \diam(I) \le \dist(4I, \pom) \le \dist(I, \pom) \le 40 \diam(I), \quad\forall I \in \W, 
		\end{equation}
		and 
		\begin{equation*}
		\diam(I_1) \approx \diam(I_2), \quad\text{ whenever $I_1$ and $I_2$ touch}. 
		\end{equation*}
		Take an arbitrary ball $B=B(x, r)$ with $x \in \pom$ and $0<r<\infty$. Set $\W_B:=\{I \in \W: I \cap B \neq \emptyset\}$. For any $I \in \W_B$, we choose $Y_I \in I \cap B$. Then Caccioppoli's inequality applied to $u-u(Y_I)$ gives 
		\begin{multline}\label{eq:BOm}
		\iint_{B \cap \Omega} |\nabla u(Y)|^2 \delta(Y)\, dY 
		\lesssim \sum_{I \in \W_B} \ell(I) \iint_{I} |\nabla (u-u(Y_I))(Y)|^2 \, dY
		\\ 
		\lesssim \sum_{I \in \W_B} \ell(I)^{-1} \iint_{\frac32 I} |u(Y)-u(Y_I)|^2 \, dY
		\lesssim \|u\|_{\dcv(\Omega)}^2 \sum_{I \in \W_B} \ell(I)^n \varphi(\ell(I))^2,
		\end{multline}
		where we have used the doubling property of $\varphi$ (cf. Lemma~\ref{lem:pro-phi} part \eqref{list-2}).
		
		Pick $y_I \in \pom$ so that $\delta(Y_I)=|Y_I-y_I|$. Set $B_I:=B(y_I, \ell(I))$ and $\Delta_I:=B_I \cap \pom$. For each $I \in \W_B$, it follows from \eqref{eq:W1} that 
		\begin{align}\label{eq:ineqs_box_1} 
		\diam(I) \le \dist(I, \pom) \le\delta(Y_I) \le \diam(I) + \dist(I, \pom) \le 41 \diam(I), 
		\end{align}
		and 
		\begin{align}\label{eq:ineqs_box_2} 
		\ell(I) \leq \diam(I) \leq \dist(I, \pom) \leq \delta(Y_I) \le |Y_I-x|<r. 
		\end{align}
		Note that 
		\begin{align*}
		|y_I-x| \le |y_I-Y_I| + |Y_I-x| = \delta(Y_I) + |Y_I-x| \leq 2|Y_I-x|<2r,  
		\end{align*}
		which implies $B_I \subset 3B$. Moreover, we claim that we even have some bounded overlap property, in the sense that for any $k \in \Z$ with $2^{-k}<r$, there holds 
		\begin{align}\label{eq:IWB-overlap}
		\sum_{I \in \W_B: \ell(I)=2^{-k}} \mathbf{1}_{\Delta_I} \le C_n \, \mathbf{1}_{3\Delta}. 
		\end{align}
		Indeed, if $\Delta_I \cap \Delta_J \neq \emptyset$ with $I, J \in \W_B$ and $\ell(I)=\ell(J)$, then we pick $z_{I,J} \in \Delta_I \cap \Delta_J$ and obtain 
		\begin{multline*}
		\dist(I, J) \le |Y_I-Y_J| \le |Y_I-y_I| + |y_I-z_{I,J}| + |z_{I,J}-y_J| + |y_J-Y_J|
		\\
		\le \delta(Y_I) + \ell(I) + \ell(J) + \delta(Y_J) 
		\le 84\sqrt{n+1} \, \ell(I). 
		\end{multline*}
		Thus, given $I \in \W_B$ we clearly have that 
		\[
		\#\{J \in \W_B: \Delta_J \cap \Delta_I \neq \emptyset, \, \ell(J)=\ell(I) \} \le C_n. 
		\]
		This shows \eqref{eq:IWB-overlap}.

		On the other hand, if $\diam(\pom) \le 2^{-k}<r$, we have for every $I \in \W_B$ with $\ell(I)=2^{-k}$, 
		\[
		|Y_I-x| \le |Y_I-y_I| + |y_I-x| \le \delta(Y_I) + \diam(\pom) \le 41 \diam(I) + \diam(\pom)<42 \sqrt{n+1} \, 2^{-k}, 
		\]
		using \eqref{eq:ineqs_box_1} and therefore $I \subset B(x, 43\sqrt{n+1}\, 2^{-k})$. Thus, 
		\begin{multline}\label{eq:IWB-number}
		\#\{I \in \W_B: \ell(I)=2^{-k}\} 
		=2^{k(n+1)} \sum_{I \in \W_B: \ell(I)=2^{-k}} |I| 
		\\
		\qquad=2^{k(n+1)}\bigg|\bigcup_{I \in \W_B: \ell(I)=2^{-k}} I \bigg| 
		\le 2^{k(n+1)} |B(x, 43\sqrt{n+1} \, 2^{-k})| 
		= C_n, 
		\end{multline}
		whenever $\diam(\pom) \le 2^{-k}<r$. Then, recalling \eqref{eq:ineqs_box_2}, we use \eqref{eq:IWB-overlap} and \eqref{eq:IWB-number} to obtain that 
		\begin{align}\label{eq:IWB}
		\sum_{I \in \W_B} \ell(I)^n \varphi(\ell(I))^2  
		&= \sum_{\substack{I \in \W_B \\ \ell(I) < \diam(\pom)}} \ell(I)^n \varphi(\ell(I))^2 + \sum_{\substack{I \in \W_B \\ \ell(I) \geq \diam(\pom)}} \ell(I)^n \varphi(\ell(I))^2  
		\\
		&= \sum_{k: \, 2^{-k}<\min\{r, \diam(\pom)\}} \sum_{I \in \W_B:\ell(I)=2^{-k}} \ell(I)^n \varphi(\ell(I))^2 
		\nonumber\\
		&\qquad\qquad\quad+ \sum_{k: \, \diam(\pom) \le 2^{-k}<r} \sum_{I \in \W_B:\ell(I)=2^{-k}} \ell(I)^n \varphi(\ell(I))^2 
		\nonumber\\
		&\lesssim \sum_{k: \, 2^{-k}<r} \varphi(2^{-k})^2 \sum_{I \in \W_B:\ell(I)=2^{-k}} \sigma(\Delta_I) 
		\nonumber\\
		&\qquad\qquad\quad+ \sum_{k:\, \diam(\pom) \le 2^{-k}<r} r^n \varphi(2^{-k})^2 \# \{I \in \W_B:\ell(I)=2^{-k}\} 
		\nonumber\\
		&\lesssim (\sigma(3\Delta) + r^n) \bigg(\sum_{k:\, 2^{-k}<r} \varphi(2^{-k}) \bigg)^2 
		\nonumber\\
		&\lesssim r^n \bigg(\int_{0}^r \varphi(s) \frac{ds}{s}\bigg)^2  
		\nonumber\\
		&\lesssim r^n \varphi(r)^2, \nonumber
		\end{align}
		where in the last steps we have used the ADR property of $\pom$, and 
		\begin{multline} \label{eq:ssr}
		\sum_{k:\, 2^{-k}<r} \varphi(2^{-k}) 
		\approx
		\sum_{k:\, 2^{-k}<r} \varphi(2^{-k}) \int_{2^{-k}}^{2^{-k+1}} \frac{ds}{s} 
		\\
		\le  
		\sum_{k:\, 2^{-k}<r} \int_{2^{-k}}^{2^{-k+1}} \varphi(s) \frac{ds}{s} 
		= 
		\int_0^{2r} \varphi(s) \frac{ds}{s} 
		\lesssim \varphi(2r) 
		\lesssim \varphi(r), 
		\end{multline}
		since $\varphi$ is non-decreasing and doubling, and we have also used \eqref{eq:ass-phi}. As a consequence, the desired estimate immediately follows from \eqref{eq:BOm} and \eqref{eq:IWB}.  
	\end{proof}

	\subsubsection{Equivalence of spaces in 1-sided CAD}
	
	Before exploring the boundary value problem associated with the space $\mathscr{C}^{*, \varphi}(\Omega)$, let us first show that this space can in fact be identified with $\dot{\mathscr{C}}^{\varphi}(\Omega)$ under some geometrical assumptions on $\Omega$. To this end, we need some definitions and lemmas. 
	
	\begin{definition}[\bf Corkscrew condition]\label{def:cks}
		We say that an open set $\Omega\subset \ree$ satisfies the \textit{corkscrew condition} if for some uniform constant $0<c_0<1$ and for every $x\in \partial\Omega$ and $0<r<\diam(\partial\Omega)$, if we write $\Delta:=\Delta(x,r)$, there exists a ball $B(X_\Delta,c_0r)\subset B(x,r)\cap\Omega$.  The point $X_\Delta\subset \Omega$ is called a \textit{corkscrew point relative to} $\Delta$ (or relative to $B$). We note that  we may allow $r<C\diam(\pom)$ for any fixed $C$, simply by adjusting the constant $c_0$.
	\end{definition}
	
	\begin{definition}[\bf Harnack Chain condition] \label{def:hc}
		We say that an open set $\Omega \subset \ree$ satisfies the \textit{Harnack chain condition} if there are uniform constants $C_1,C_2>1$ so that for every pair of points $X, X'\in \Omega$
		there is a chain of balls $B_1, B_2, \dots, B_N\subset \Omega$ with $N \leq  C_1(2+\log_2^+ \Pi(X,X'))$,
		where
		\begin{equation}\label{cond:Lambda}
		\Pi(X,X'):=\frac{|X-X'|}{\min\{\delta(X), \delta(X')\}},
		\end{equation}
		such that $X\in B_1$, $X'\in B_N$, $B_k\cap B_{k+1}\neq\emptyset$ for each $1\le k\le N-1$, and 
		\begin{equation}\label{preHarnackball}
		C_2^{-1} \diam(B_k) \leq \dist(B_k,\partial\Omega) \leq C_2 \diam(B_k),\qquad
		1\le k\le N.
		\end{equation}
		The chain of balls is called a \textit{Harnack chain}.
	\end{definition}
	
	We note that in the context of the previous definition if $\Pi(X,X')\le 1$ we can trivially form the Harnack chain $B_1=B(X,3\delta(X)/5)$ and $B_2=B(X', 3\delta(X')/5)$ so that \eqref{preHarnackball} holds with $C_2=3$. Hence the Harnack chain condition is non-trivial only when $\Pi(X,X')> 1$.
	
	\begin{remark}\label{rem:HC-balls}
		If $B_1,\dots B_N$ is a Harnack chain connecting $X$ and $Y$ in $\Omega$, we clearly have that for every $1\le k\le N-1$, using that $B_k\cap B_{k+1}\neq\emptyset$,
		\[
		C_2^{-1} \diam(B_k) 
		\leq 
		\dist(B_k,\partial\Omega) 
		\le
		\diam(B_{k+1})+\dist(B_{k+1},\partial\Omega) 
		\le
		(1+C_2) \diam(B_{k+1})
		\]
		and
		\[
		C_2^{-1} \diam(B_{k+1}) 
		\leq 
		\dist(B_{k+1},\partial\Omega) 
		\le
		\diam(B_{k})+\dist(B_{k},\partial\Omega) 
		\le
		(1+C_2) \diam(B_{k}).
		\]
		Besides, there hold 
		\[
		C_2^{-1} \diam(B_1) 
		\leq 
		\dist(B_1,\partial\Omega) 
		\le
		\delta(X)
		\le
		\diam(B_1)+\dist(B_1,\partial\Omega) 
		\le
		(1+C_2) \diam(B_{1})
		\]
		and
		\[
		C_2^{-1} \diam(B_N) 
		\leq 
		\dist(B_N,\partial\Omega) 
		\le
		\delta(Y)
		\le
		\diam(B_N)+\dist(B_N,\partial\Omega) 
		\le
		(1+C_2) \diam(B_{N}).
		\]
		All these easily imply that for every $1\le j\le N$, 
		\begin{equation}\label{diam-compar:HC}
		(1+C_2)^{-N}C_2^{N-1}\max\{\delta(X),\delta(Y)\}
		\le
		\diam(B_j)\le (1+C_2)^{N-1}C_2^N\min\{\delta(X),\delta(Y)\}.
		\end{equation}
	\end{remark}
	
	\begin{definition}[\textbf{1-sided CAD}]  \label{def:NTA}
		We say that an open set $\Omega\subset \ree$ is a \textit{1-sided CAD} (chord-arc domain) if it satisfies the corkscrew and Harnack chain conditions, and has an ADR boundary.
	\end{definition}
	
	Let us state an auxiliary lemma that will be used later.
	
	\begin{lemma}\label{lemma:unbounded}
		Let $\Omega\subset\ree$, $n\ge1$, be an unbounded open set with bounded boundary and assume that it satisfies the corkscrew condition with constant $c_0\in(0,1)$. Then, for every $x\in\pom$ and $r\ge \diam(\pom)$ there exists a ball $B(X_{x,r}, c_0\,r/4)\subset B(x,r)\cap\Omega$.
	\end{lemma}
	
	\begin{proof}
		Fix $x\in\pom$ and let us first consider the case $ \diam(\pom)\le r\le 2 \diam(\pom)$. If we write $\Delta'=\Delta(x,\diam(\pom)/2)$, by the corkscrew condition we can find $X_{\Delta'}$ such that $B(X_{\Delta'},c_0\,\diam(\pom)/2)\subset B(x,\diam(\pom)/2)\cap\Omega$. If we set $X_{x,r}:=X_{\Delta'}$ then we get as desired
		\[
		B(X_{x,r}, c_0\,r/4)
		\subset
		B(X_{\Delta'},c_0\,\diam(\pom)/2)
		\subset 
		B(x,\diam(\pom)/2)\cap\Omega
		\subset
		B(x,r)\cap\Omega.
		\]
		Consider next the case $r>2 \diam(\pom)$ and pick an arbitrary $X_{x,r}\in\partial B(x,3r/4)$  so that $B(X_{x,r},r/4)\subset B(x,r)$. We claim that  $B(X_{x,r},r/4)\subset \Omega$. To see this, note that $x\in\pom$ implies that $\pom\subset \overline{B(x,\diam(\pom))}$ and hence $\pom\setminus \overline{B(x,\diam(\pom))}=\emptyset$. This means that either $\ree\setminus  \overline{B(x,\diam(\pom))}\subset \Omega$ or  $\ree\setminus  \overline{B(x,\diam(\pom))}\subset \ree\setminus\overline{\Omega}$. In the latter case we have $\overline{\Omega}\subset \overline{B(x,\diam(\pom))}$ which implies that $\Omega$ is bounded and this contradicts our hypothesis. Consequently, $\ree\setminus  \overline{B(x,\diam(\pom))}\subset \Omega$ and as a result $B(X_{x,r},r/4)\subset \ree\setminus  \overline{B(x,\diam(\pom))}\subset  \Omega$. All these, together with the fact that $c_0\in (0,1)$, imply that $B(X_{x,r},c_0r/4)\subset B(x,r)\cap\Omega$. 
	\end{proof}

	The main result in this subsection is formulated as follows.
	
	\begin{theorem}\label{thm:CC}
		Let $\Omega \subset \ree$, $n\ge 2$, be a 1-sided CAD, and let $L=-\div(A \nabla)$ be a real (non-necessarily symmetric) elliptic operator. Then for every $\beta \in (0, \alpha_1]$ (where $\alpha_1 \in (0,1)$ is the exponent in the De Giorgi/Nash estimates, see Lemma~\ref{lem:DGN}), and for every growth function $\varphi \in \mathcal{G}_\beta$, we have
		\begin{align}\label{CC-1}
		\|u\|_{\dot{\mathscr{C}}^{\varphi}(\Omega)} 
		\approx \|u\|_{\mathscr{C}^{*,\varphi}(\Omega)},
		\qquad
		\forall\, u\in W^{1,2}_\loc(\Omega),\ \text{$Lu=0$ in the weak sense in $\Omega$},
		\end{align}
		where the implicit constants depend only on $n$, the 1-sided CAD constants, the ellipticity constant of $L$, and $C_{\varphi}$. In particular, if we define 
		\[
		\mathscr{C}_L^{*,\varphi}(\Omega)
		:= \big\{u\in W^{1,2}_{\loc}(\Omega) \cap \mathscr{C}^{*,\varphi}(\Omega): Lu=0 \text{ in the weak sense in $\Omega$} \big\},
		\]
		then 
		\begin{align*}
		\dot{\mathscr{C}}_{L}^{\varphi}(\Omega)
		=\mathscr{C}_{L}^{*,\varphi}(\Omega)
		\quad\text{ and }\quad 
		\dot{\mathscr{C}}_L^{\varphi}(\Omega)/\re 
		=\mathscr{C}_L^{*,\varphi}(\Omega)/\re,
		\quad\text{with equivalent norms}.
		\end{align*}
	\end{theorem}

	To show Theorem \ref{thm:CC}, we establish two auxiliary results, the second one depending heavily on the Harnack chain condition. 
	
	\begin{lemma}\label{lemma:Calpha-HC-balls}
		Let $\Omega\subset\ree$, $n\ge 2$, be an open set, and let $L=-\div(A\nabla )$ be  a real (non-necessarily symmetric) uniformly elliptic operator. Let  $\Upsilon\ge 1$ and assume that the ball $B=B(X,R)\subset\Omega$ satisfies $\Upsilon^{-1}R\le \dist(B,\pom)\le \Upsilon R$. Then for every $\beta \in (0, \alpha_1]$ (where $\alpha_1$ is the exponent in the De Giorgi/Nash estimates, see Lemma \ref{lem:DGN}), for every growth function $\varphi \in \mathcal{G}_\beta$, and for every $u\in  W^{1,2}_{\loc}(\Omega)$ satisfying $Lu=0$ in the weak sense in $\Omega$ there holds 
		\begin{align*}
		\sup_{\substack{X,Y\in \overline{B}\\	X\neq Y}}\frac{|u(X)-u(Y)|}{\varphi(|X-Y|)}
		\lesssim \|u\|_{\mathscr{C}^{*, \varphi}(\Omega)}, 
		\end{align*}
		where the implicit constant depends only on $n$, the ellipticity of $L$, $\Upsilon$, and $C_{\varphi}$.  
	\end{lemma}

	\begin{proof}
		Assume first that $X,Y\in \overline{B}$ with $|X-Y|<R/(8\Upsilon)$. Write $B_X=B(X,R/(8\Upsilon))$ and note that
		\[
		\frac{R}{\Upsilon}
		\le
		\dist(B,\pom)
		\le
		\delta(X)
		\le
		\dist(4B_X,\pom)+\frac{R}{2\Upsilon},
		\]
		hence, $\dist(4B_X,\pom)\ge R/(2\Upsilon)$ and $4B_X\subset\Omega$. Also, if $\widehat{x}\in\pom$ is chosen so that $\dist(B,\pom)=\dist(B,\widehat{x})$, for every $Z \in 2 B_X$, we have
		\[
		|Z-\widehat{x}|
		\le
		|Z-X|+\diam(B)+\dist(B,\widehat{x})
		<
		(3+\Upsilon)R,
		\]
		thus $2B_X\subset B(\widehat{x},(3+\Upsilon)R)$. Using all these, Lemma~\ref{lem:DGN} applied in $B_X$ to the solution $u(\cdot)-u_{2B_X}$ (where $u_{2B_X}$ is the average of $u$ in $2B_X$), and Poincaré's inequality we conclude that (we allow constants depend on $\Upsilon$)
		\begin{align*}
		|u(X)-u(Y)|
		&\lesssim
		\bigg(\frac{|X-Y|}{R}\bigg)^{\alpha_1}
		\bigg(\frac{1}{|2B_X|}\iint_{2B_X}|u(Y)-u_{2B_X}|^2\,dY\bigg)^{\frac{1}{2}}
		\\
		&\lesssim 
		\bigg(\frac{|X-Y|}{R}\bigg)^{\alpha_1} 
		\bigg(\frac{1}{R^{n-1}} \iint_{2B_X}|\nabla u(Y)|^2\,dY\bigg)^{\frac{1}{2}}
		\\
		&\lesssim
		\bigg(\frac{|X-Y|}{R}\bigg)^{\alpha_1} 
		\bigg(\frac{1}{R^n} \iint_{B(\widehat{x},(3+\Upsilon)R)}|\nabla u(Y)|^2\,\delta(Y)dY\bigg)^{\frac{1}{2}}
		\\
		&\lesssim
		\bigg(\frac{|X-Y|}{R}\bigg)^{\alpha_1} \varphi((3+\Upsilon) R) 
		\|u\|_{\mathscr{C}^{*, \varphi}(\Omega)}
		\\
		&\lesssim
		\bigg(\frac{|X-Y|}{R}\bigg)^{\beta} \varphi(R) 
		\|u\|_{\mathscr{C}^{*, \varphi}(\Omega)}
		\\
		&\lesssim
		\varphi(|X-Y|) \|u\|_{\mathscr{C}^{*, \varphi}(\Omega)}, 
		\end{align*}
		where we have used that $|X-Y|<R$, $\beta \le \alpha_1$, and Lemma \ref{lem:pro-phi}.
		
		To consider the case $X,Y\in \overline{B}$ with $|X-Y| \ge R/(8\Upsilon)$, we can easily find a sequence $X_0=X, X_1, X_2, \dots, X_N=Y$ with  $|X_j-X_{j+1}|<R/(8\Upsilon)$ for $0\le j\le N-1$ and with $N$ depending just on $\Upsilon$. We can then apply what we have just obtained to $X_j$ and $X_{j+1}$ and eventually conclude the desired estimate.
	\end{proof}

	\begin{lemma}\label{lemma:Calpha-boundary}
		Let $\Omega\subset\mathbb{R}^{n+1}$, $n\ge 2$, be a 1-sided CAD and let $L=-\div(A\nabla )$ be  a real (non-necessarily symmetric) uniformly elliptic operator. Given $c_1\in (0,1)$, $\beta \in (0, \alpha_1]$ (where $\alpha_1$ is the exponent in the De Giorgi/Nash estimates, see Lemma~\ref{lem:DGN}), and a growth function $\varphi \in \mathcal{G}_\beta$, there exists a constant $C$ (depending only on $n$, the 1-sided CAD constants, the ellipticity constants of $L$, $c_1$, and $C_{\varphi}$) such that for every $x\in\pom$, $0<r\le s<\infty$, and $X$, $Y\in\Omega$ so that $B(X,c_1 r)\subset B(x,r)\cap\Omega$ and $B(Y,c_1 s)\subset B(x,s)\cap\Omega$, and for every $u\in  W^{1,2}_{\loc}(\Omega)$ satisfying $Lu=0$ in the weak sense in $\Omega$  there holds
		\begin{equation*}
		|u(X)-u(Y)| \le C \, \varphi(s) \, \|u\|_{\mathscr{C}^{*, \varphi}(\Omega)}. 
		\end{equation*} 
	\end{lemma}
	
	\begin{proof}
		Note that if $\|u\|_{\mathscr{C}^{*, \varphi}(\Omega)}=\infty$ there is nothing to prove. Besides, if $\|u\|_{\mathscr{C}^{*, \varphi}(\Omega)}=0$ then $u$ is constant and the desired inequality is also trivial. Thus, we may assume that $0<\|u\|_{\mathscr{C}^{*, \varphi}(\Omega)}<\infty$, and in that case we can replace $u$ by $u/\|u\|_{\mathscr{C}^{*, \varphi}(\Omega)}$ and eventually assume that $\|u\|_{\mathscr{C}^{*, \varphi}(\Omega)}=1$. Let $k_0\ge 0$ be such that $2^{-k_0-1}s< r \le 2^{-k_0}s$. 
		
		We first work under the restriction $s<2\diam(\pom)/c_1$ (if $\pom$ is unbounded, this imposes no restriction). Using that $\Omega$ satisfies the corkscrew condition, for every $0\le k\le k_0$ we pick $X_k := X_{\Delta(x,2^{-k}s)}$ (cf.~Definition~\ref{def:cks}) so that 
		\begin{align}\label{BX}
		B(X_k, c_0 2^{-k}s) \subset B(x, 2^{-k}s) \cap \Omega.
		\end{align}
		Note that for every $0\le k\le k_0$ we have $c_0 2^{-k}s\le \delta(X_k)<  2^{-k}s$. Also, 
		\[
		2^{-1}c_1\delta(X_{k_0})<
		2^{-k_0-1}c_1s
		<
		c_1r
		\le
		\delta(X)
		<
		r 
		\le 
		2^{-k_0}s
		\le
		c_0^{-1}\delta(X_{k_0})
		\]
		and
		\[
		c_1\delta(X_0)<
		c_1s
		\le 
		\delta(Y)
		<
		s
		\le
		c_0^{-1}\delta(X_{0}). 
		\]
		Moreover, 
		\begin{align*}
		|Y-X_{0}| &\le |Y-x|+|x-X_0| < 2s, 
		\\
		|X-X_{k_0}| &\le |X-x|+|x-X_{k_0}|< r+ 2^{-{k_0}}s < 3r,
		\end{align*}
		and
		\[
		|X_{k-1}-X_{k}|
		\le
		|X_{k-1}-x|+|x-X_{k}|
		<
		2^{-k+1}s+ 2^{-k}s
		<
		2^{-k+2}s,
		\qquad 
		1\le k\le k_0.
		\]
		All this show that with the notation introduced in \eqref{cond:Lambda}
		\[
		\Pi(X, X_{k_0}),\ \Pi(Y, X_0),\ \Pi(X_{k-1}, X_{k})\lesssim 1, \qquad 
		1\le k\le k_0.
		\]
		Set $X_{k_0+1}:=X$ and $X_{-1}:=Y$ and fix an arbitrary $0\le k\le k_0+1$. Use that $\Omega$ satisfies the Harnack chain condition to construct a chain of balls $B_1^k,\dots, B_{N_k}^k$ joining $X_{k-1}$ and $X_{k}$, satisfying \eqref{preHarnackball}, and with $N_k\lesssim (2+\log_2^+\Pi(X_k, X_{k-1}))\lesssim 1$. Pick $Y_j^k\in  B_j^k\cap B_{j+1}^{k}$ for $1\le j\le N_k-1$ and note that by Lemma~\ref{lemma:Calpha-HC-balls} and \eqref{diam-compar:HC} in Remark~\ref{rem:HC-balls} we obtain
		\begin{align*}
		|u(X_{k-1})-u(X_{k})|
		&\le
		|u(X_{k-1})-u(Y_{1}^k)|
		+
		\sum_{j=1}^{N_k-1} |u(Y_{j}^k)-u(Y_{j+1}^k)|
		+
		|u(Y_{N_k}^k)-u(X_{k})|
		\\
		&\lesssim 
		\varphi(|X_{k-1}-Y_{1}^k|) 
		+
		\sum_{j=1}^{N_k-1} \varphi(|Y_{j}^k-Y_{j+1}^k|) 
		+
		\varphi(|Y_{N_k}^k-X_{k}|) 
		\\
		&\le
		(N_{k}+1)\sup_{0\le j\le N_k} \varphi(\diam(B_{j}^k)) 
		\\
		&\lesssim
		(N_{k}+1) \varphi\big((1+C_2)^{N_k-1} C_2^{N_k}\min\{\delta(X_{k-1}),\delta(X_k)\}\big)
		\\
		&\lesssim \varphi(2^{-k}s),
		\end{align*}
		where we have used that $N_k\lesssim 1$ and Lemma \ref{lem:pro-phi}. Then by telescoping, \eqref{eq:ass-phi}, and invoking again Lemma \ref{lem:pro-phi}, we conclude that
		\begin{multline*}
		|u(Y)-u(X)|
		= |u(X_{-1})-u(X_{k_0+1})|
		\le \sum_{k=0}^{k_0+1} |u(X_{k-1})-u(X_{k})|
		\\
		\lesssim \sum_{k=0}^{k_0+1} \varphi(2^{-k}s)
		\lesssim \sum_{k=0}^{\infty} \int_{2^{-k} s}^{2^{-k+1} s} \varphi(t) \frac{dt}{t}  
		\lesssim \int_{0}^{2s} \varphi(t) \frac{dt}{t}  
		\lesssim \varphi(2s)
		\lesssim \varphi(s). 
		\end{multline*}
		This is the desired estimate under the assumption that $s<2\diam(\pom)/c_1$.

		Let us next consider the case $s\ge 2\diam(\pom)/c_1$, which is meaningful only when $\pom$ is bounded. Note that 
		$B(Y,c_1 s)\subset B(x,s)\cap\Omega$ implies that $2\diam(\pom)\le c_1 s\le\diam(\Omega)$, hence $\Omega$ must be unbounded (otherwise, $\diam(\pom)=\diam(\Omega)$). This means that we are in the situation $\Omega$ unbounded with $\pom$ bounded. For any $0 \le k \le k_0$ so that $2^{-k}s<\diam(\pom)$, we use that $\Omega$ satisfies the corkscrew condition to find $X_k:=X_{\Delta(x,2^{-k}s)}$. If $0 \le k \le k_0$ is so that $2^{-k}s\ge\diam(\pom)$, we invoke Lemma~\ref{lemma:unbounded} and let $X_k:=X_{x,2^{-k}s}$. These choices guarantee that $B(X_k,c_0 2^{-k}s/4)\subset B(x,2^{-k}s)\cap\Omega$, which is analogous  to \eqref{BX}. We can then repeat the previous argument replacing $c_0$ by $c_0/4$ and conclude the desired estimate. Further details are left to the reader. 
	\end{proof}

	\begin{proof}[\textbf{Proof of Theorem~\ref{thm:CC}}]
		
		In view of Lemma \ref{lem:CC}, it suffices to show the upper bound in \eqref{CC-1}. First note that if $\|u\|_{\mathscr{C}^{*, \varphi}(\Omega)}=\infty$ there is nothing to prove. Besides,  if $\|u\|_{\mathscr{C}^{*, \varphi}(\Omega)}=0$ then $u$ is constant and the desired inequality is also trivial. Thus, as before, we may assume that $\|u\|_{\mathscr{C}^{*, \varphi}(\Omega)}=1$.
		
		Fix $X, Y \in \Omega$ and $0< \beta \le \alpha_1$. We may assume that $\delta(X)\le \delta(Y)$. We first consider the case $|X-Y|<\max\{\delta(X),\delta(Y)\}/2=\delta(Y)/2$. Let $B_Y= B(Y,\delta(Y)/2)$ so that $X\in B_Y$ and $\dist(B_Y,\pom) = \delta(Y)/2$.   We can then apply Lemma~\ref{lemma:Calpha-HC-balls} with $B=B_Y$ and $\Upsilon=1$ to see that
		\[
		|u(X)-u(Y)| \lesssim \varphi(|X-Y|). 
		\]
		
		To continue, we next assume that $|X-Y|\ge \max\{\delta(X),\delta(Y)\}/2=\delta(Y)/2$. Pick $\widehat{x}$, $\widehat{y}\in\pom$ so that $\delta(X)=|X-\widehat{x}|$ and $\delta(Y)=|Y-\widehat{y}|$. If $|X-Y|<2\diam(\pom)$ we choose $\widetilde{X}:=X_{\Delta(\widehat{x},4|X-Y|)}$ and $\widetilde{Y}:=X_{\Delta(\widehat{y},4|X-Y|)}$ (here we use the corkscrew condition, see Definition~\ref{def:cks}). On the other hand, if  $|X-Y|\ge 2\diam(\pom)$ we then have that $\pom$ is necessarily bounded and $2\diam(\pom) \le |X-Y|\le \diam(\Omega)$, hence $\Omega$ must be unbounded (otherwise, $\diam(\Omega)=\diam(\pom)$). We can then use Lemma~\ref{lemma:unbounded} to find $\widetilde{X}:=X_{\widehat{x},4|X-Y|}$ and $\widetilde{Y} := X_{\widehat{y}, 4|X-Y|}$. In either scenario and by construction the following hold:
		\[
		B(X, \delta(X))\subset B(\widehat{x}, 2\delta(X))\cap\Omega, \quad 
		B(\widetilde{X}, c_0|X-Y|)\subset B(\widehat{x}, 4|X-Y|)\cap\Omega,
		\]
		and
		\[
		B(Y, \delta(Y))\subset B(\widehat{y}, 2\delta(Y))\cap\Omega, 
		\quad\text{ and }\quad 
		B(\widetilde{Y}, c_0|X-Y|)\subset B(\widehat{y}, 4|X-Y|)\cap\Omega.
		\]
		Observing that $2\delta(X)\le 2\delta(Y)\le 4|X-Y|$, we can then invoke Lemma~\ref{lemma:Calpha-boundary} with $c_1=c_0/4$ to  obtain 
		\begin{equation}\label{34qf3f3af}
		|u(X)-u(\widetilde{X})|
		\lesssim \varphi(|X-Y|) 
		\qquad\text{and}\qquad
		|u(Y)-u(\widetilde{Y})|
		\lesssim \varphi(|X-Y|).
		\end{equation}
		On the other hand we note that $\delta(\widetilde{X})\ge c_0|X-Y|$, $\delta(\widetilde{Y})\ge c_0|X-Y|$, and 
		\begin{multline*}
		|\widetilde{X}-\widetilde{Y}|
		\le
		|X-\widehat{x}|+|\widehat{x}-\widetilde{X}|+|X-Y|+|Y-\widehat{y}|+|\widehat{y}-\widetilde{Y}|
		\\ 
		\le \delta(X)+9|X-Y|+\delta(Y)
		\le 13|X-Y|. 
		\end{multline*}
		Thus, $\Pi(\widetilde{X}, \widetilde{Y})\le 13 c_0^{-1}$. We can then form a Harnack chain $B_1, \ldots, B_N$ joining $\widetilde{X}$ and $\widetilde{Y}$ with $N\lesssim (2+\log_2^+ \Pi(\widetilde{X}, \widetilde{Y}))\lesssim 1$. Pick $Z_j \in  B_j\cap B_{j+1}$ for $1\le j\le N$ and note that by Lemma~\ref{lemma:Calpha-HC-balls} and \eqref{diam-compar:HC} in Remark~\ref{rem:HC-balls} we obtain
		\begin{align*}
		|u(\widetilde{X})-u(\widetilde{Y})|
		&\le
		|u(\widetilde{X})-u(Z_1)|
		+
		\sum_{j=1}^{N-1} |u(Z_{j})-u(Z_{j+1})|
		+
		|u(Z_{N})-u(\widetilde{Y})|
		\\
		&\lesssim
		\varphi(|\widetilde{X}-Z_{1}|)
		+
		\sum_{j=1}^{N-1} \varphi(|Z_{j}-Z_{j+1}|)
		+ \varphi(|Z_{N}-\widetilde{Y}|) 
		\\
		& \le
		(N+1)\sup_{1 \le j \le N} \varphi(\diam(B_{j})) 
		\\
		&\lesssim
		(N+1) \varphi\big((1+C_2)^{N-1} C_2^{N}\min\{\delta(\widetilde{X}),\delta(\widetilde{Y})\}\big)
		\\
		&\lesssim
		\varphi \big(\min\big\{|\widetilde{X}-\widehat{x}|,|\widetilde{Y}-\widehat{x}|\big\} \big)
		\\
		&\le 
		\varphi(4|X-Y|)
		\\
		&\lesssim
		\varphi(|X-Y|),
		\end{align*}
		where we have used that $N\lesssim 1$ and Lemma \ref{lem:pro-phi}. This and \eqref{34qf3f3af} readily imply the desired estimate:
		\begin{equation*}
		|u(X)-u(Y)|
		\le |u(X)-u(\widetilde{X})|+|u(\widetilde{X})-u(\widetilde{Y})|+|u(\widetilde{Y})-u(Y)|
		\lesssim \varphi(|X-Y|).
		\end{equation*}
		This shows the upper bound in \eqref{CC-1} and completes the proof of Theorem \ref{thm:CC}.
	\end{proof}


	\subsubsection{Existence of solutions without connectivity}
	
	After working in 1-sided CAD to obtain Theorem~\ref{thm:CC}, let us go back to a more general setting, where we assume no connectivity on $\Omega$, and consider the Dirichlet boundary value problem associated with the $\mathscr{C}^{*, \varphi}$ space. 	
	
	We may now state an existence result for the Dirichlet boundary value problem associated with the $\mathscr{C}^{*, \varphi}$ space. Again, we remark the fact that no connectivity is assumed on $\Omega$.	
	
	\begin{theorem} \label{thm:CaCa_existence_general}
		Let $\Omega \subset \R^{n+1}$, $n \geq 2$, be an open set with ADR boundary so that either {\bf $\Omega$ is bounded or $\Omega$ is unbounded with $\pom$ being unbounded}.  Let $L = -\div (A\nabla)$ be a real (non-necessarily symmetric) elliptic operator, and let $\omega_L$ be the associated elliptic measure. Then there exists $\beta \in (0, 1)$ (depending only on $n$, the ADR constant, and the ellipticity constant of $L$) such that for any $\theta > 1$ and for any growth function $\varphi \in \mathcal{G}_\beta$, the function
		\begin{equation}\label{eq:solution_carleson}
		u(X) = \int_\pom f(y) \, d\omega_L^X(y), \qquad X \in \Omega
		\end{equation}
		is a solution to the $\mathscr{C}^{*, \varphi}$-Dirichlet problem 
		\begin{equation} \label{CaCa-problem_general}
		\begin{cases}
		u \in W^{1, 2}_{\loc} (\Omega), \\
		Lu = 0 \text{ in the weak sense in } \Omega, \\
		u \in \mathscr{C}^{*, \varphi}(\Omega), \\
		\restr{u}{\pom}^{{\rm nt. \, lim}, \theta} = f \in \dcv(\pom), \,  \sigma\text{-a.e. on } \pom. 
		\end{cases}
		\end{equation}
		Moreover, $u \in \mathscr{C} (\overline{\Omega})$ and satisfies $\restr{u}{\pom} = f$, and there is a constant $C$ (depending on $n$, the ADR constant, the ellipticity constant of $L$, and $C_\varphi$) such that 
		\begin{equation*}
		\|u\|_{\mathscr{C}^{*, \varphi}(\Omega)}
		\leq C \|f\|_{\dcv(\pom)}.
		\end{equation*}
	\end{theorem}

	\begin{proof} 
		Since $\pom$ is ADR, it follows from \eqref{eq:ADR-CDC} that $\Omega$ satisfies the CDC. Then, from Theorem~\ref{thm:well} we see that, given $f \in \dcv(\pom)$, the function $u$ defined in \eqref{eq:solution_carleson} solves the problem~\eqref{eq:C-varphi}. Moreover, $u \in \dcv(\Omega) \cap \mathscr{C} (\overline{\Omega})$ satisfies $\restr{u}{\pom} = f$, and 
		\begin{equation*}
		\|u\|_{\dot{\mathscr{C}}^\varphi(\Omega)}
		\lesssim \|f\|_{\dot{\mathscr{C}}^\varphi(\pom)}.
		\end{equation*}
		On the other hand, invoking Lemma~\ref{lem:CC} we infer that $u \in \mathscr{C}^{*, \varphi}(\Omega)$ and 
		\begin{equation*}
		\|u\|_{\mathscr{C}^{*, \varphi}(\Omega)} \lesssim \|u\|_{\dcv(\Omega)} \lesssim \|f\|_{\dcv(\pom)}. 
		\end{equation*}
		This completes the proof. 
	\end{proof}

	\begin{remark}
	We observe that from the previous result one obtains $u \in \mathscr{C}(\overline{\Omega})$. This implies that $\restr{u}{\pom}^{{\rm nt. \, lim}, \theta} = f$ everywhere on $\pom$ in the sense of Definition of \ref{def:nt}, for any $\theta > 1$. This and the fact that none of the constants involved in the result depend on $\theta$ show that the precise value of $\theta$ does not play any relevant role in the theorem.
	\end{remark}


	\subsubsection{A sufficient condition for uniqueness} \label{subsec:carleson_uniqueness}
	
	In this subsection we will show that, under a very mild connectivity assumption, the solutions to the $\mathscr{C}^{*, \varphi}$ problem \eqref{CaCa-problem_general} are in fact unique. Before stating the result, let us recall a definition.
	
	\begin{definition}[{\bf Carrot path}]\label{def:Carrot}
		Let $\Omega \subset \ree$ be an open set. Given a point $X \in \Omega$, and a point $y \in \pom$, we say that a connected rectifiable path $\gamma=\gamma(y, X)$, with endpoints $y$ and $X$, is a \textit{carrot path} (more precisely, a $\lambda$-carrot path) connecting $y$ to $X$, if $\gamma \setminus \{y\} \subset \Omega$ and for some $\lambda \in (0,1)$ and for all $Z \in \gamma$, there holds $\lambda \ell(\gamma(y, Z)) \le \delta(Z)$. 
	\end{definition}
	
	\begin{proposition}\label{prop:CaCa_uniqueness_general}
		Let $\Omega \subset \R^{n+1}$, $n \geq 2$, be an open set with ADR boundary so that either {\bf $\Omega$ is bounded or $\Omega$ is unbounded with $\pom$ being unbounded}. Let $L = -\div (A\nabla)$ be a real (non-necessarily symmetric) elliptic operator, and let $\omega_L$ be the associated elliptic measure. Given $\theta > 1$, assume that
		{ 
		\begin{equation}\label{eq:path}
		\begin{array}{c}
		\text{there is an open neighborhood $U$ of $\pom$ in $\Omega$ so that } 
		\\
		\text{ for every $X\in U$ there exists some $\lambda = \lambda(X) \in (0, 1)$ satisfying $\lambda^{-1} < \theta$ and }  
		\\
		\text{$\sigma \Big{(} \big{\{} y \in \pom \cap B(X, 2\delta(X)) : \text{there exists a $\lambda$-carrot path joining $y$ and $X$} \big{\}} \Big{)} > 0$}.
		\end{array}
		\end{equation}
		}
		Then there exists $\beta \in (0, 1)$ (depending only on $n$, the ADR constant, and on the ellipticity constant of $L$) such that for any growth function $\varphi \in \mathcal{G}_\beta$, the solution given by \eqref{eq:solution_carleson} is the unique solution to the problem~\eqref{CaCa-problem_general}, and hence the problem is well-posed.	
	\end{proposition}

	\begin{remark}
		{
		The assumption $\lambda^{-1}< \theta$ in \eqref{eq:path} is to guarantee that the $\lambda$-carrot path joining $y$ and $X$ is contained in the non-tangential cone $\Gamma^\theta(y)$, a set where we have information on $u$. 
		We also remark that the condition \eqref{eq:path} is much weaker than the usual weak local John condition, since we only require the collection of boundary points which are non-tangentially accessible to have positive measure, rather than imposing that such a set takes a fixed portion of the entire set $\pom \cap B(X, 2\delta(X))$. See \cite{AHMMT} where the weak local John condition plays a crucial role.
		}
	\end{remark}
	
	\begin{proof}
		Suppose that $u_1$ and $u_2$ are two solutions of \eqref{CaCa-problem_general}. Write $u:=u_1-u_2$ which solves \eqref{CaCa-problem_general} with $f \equiv 0$. We are going to show that $u \equiv 0$. If $\|u\|_{\mathscr{C}^{*, \varphi}(\Omega)} = 0$, then $u$ is a constant on $\pom$. This together with the fact that $\restr{u}{\pom}^{{\rm nt. \, lim}, \theta}=0$ implies that $u \equiv 0$ in $\Omega$. On the other hand, if $\|u\|_{\mathscr{C}^{*, \varphi}(\Omega)} \neq 0$, by homogeneity we may assume that $\|u\|_{\mathscr{C}^{*, \varphi}(\Omega)} = 1$. We want to show that $u = 0$, and we proceed in a series of steps, which we state as lemmas.
		
		\begin{lemma} \label{lem:bound_U}
			In the current scenario, there exists a constant $C$ (depending on $n$, the ADR constant, the ellipticity constant of $L$, $C_{\varphi}$ in \eqref{eq:ass-phi}, and $\theta$) such that 
			\begin{equation*}
			u(X) \leq C \, \varphi(\delta(X)) \, \text{ for all } X \in U. 
			\end{equation*}
		\end{lemma}
		
		\begin{proof}
			Fix $X \in U$. Using \eqref{eq:path} and the fact that $\restr{u}{\pom}^{{\rm nt. \, lim}, \theta} = 0$, $\sigma\text{-a.e. on } \pom$, there exist $y \in B(X, 2\delta(X)) \cap \pom$ such that $\restr{u}{\pom}^{{\rm nt. \, lim}, \theta}(y)=0$ and a $\lambda$-carrot path $\gamma(y, X)$ connecting $y$ and $X$.
			{
			The assumption $\lambda^{-1}<\theta$ guarantees that $\gamma(y,X) \subset \Gamma^\theta(y)$. Indeed, for any $Z\in \gamma(y,X)$, by Definition \ref{def:Carrot} we have
			\[ \delta(Z) \geq \lambda \ell(\gamma(y,Z)) \geq \lambda |Z-y |. \]
			Thus $|Z-y| \leq \lambda^{-1} \delta(Z) < \theta \delta(Z)$, and $Z\in \Gamma^\theta(y)$. 
			
			Let $\rho>1$ be fixed such that $1-1/\rho < \lambda/3$. We define the sequence $\{X_k\}_{k \geq 0} \subset \gamma(y, X)$ as the points which satisfy $\ell_k:= \ell(\gamma(y, X_k)) = \rho^{-k} \ell(\gamma(y, X))=:\rho^{-k} \ell$, so concretely $X_0=X$. Note that $\delta(X_k) \ge \lambda \ell_k $. We also define $B_k := B(X_k, \frac{\lambda}{3} \ell_k)$ and $\widetilde{B}_k := 2 B_k$, which are clearly Whitney regions inside $\Omega$. Observe that
			\[ |X_{k+1}-X_k| \leq \ell(\gamma(X_{k+1}, X_k)) = \ell_k - \ell_{k+1} = (1-1/\rho) \ell_k < \frac{\lambda}{3} \ell_k, \]
			by our assumption on $\rho$. Thus $X_{k+1} \in B_k$. 
			
			Using Lemma \ref{lem:DGN} (applied to the function $u-u_{\widetilde{B}_k}$), Poincar\'e's inequality, the fact that $\delta(X_k) \geq \lambda \ell_k$ and $\widetilde{B}_k \subset B(y, 2\ell_k)$, $\|u\|_{\dcstar(\pom)} = 1$, and $\lambda^{-1} < \theta$, we have that
			\begin{multline*}
				\sup_{B_k} |u-u_{\widetilde{B}_k}| \lesssim \left(\bariint_{\widetilde{B}_k} |u - u_{\widetilde{B}_k}|^2 \, dZ \right)^{\frac12} 
				\\ \lesssim \lambda \ell_k \, \left(\bariint_{\widetilde{B}_k}|\nabla u|^2 \, dZ \right)^{\frac12}
				\lesssim \left(\frac{1}{(\lambda\ell_k)^n } \iint_{\widetilde{B}_k} |\nabla u|^2 \delta(Z) \, dZ \right)^{\frac12} 
				\\ \leq \left(\frac{\theta^n}{(\ell_k)^n } \iint_{B(y, 2 \ell_k)} |\nabla u|^2 \delta(Z) \, dZ \right)^{\frac12}
				\lesssim_\theta \varphi(2\ell_k) \lesssim \varphi(\ell_k).
			\end{multline*} 
			In the last inequality we have used the doubling property of $\varphi$ (see Lemma~\ref{lem:pro-phi}). Since both $X_k, X_{k+1} \in B_k$, the previous estimate implies
			\[ |u(X_k) - u(X_{k+1})| \leq |u(X_k) - u_{\widetilde{B}_k}| + |u(X_{k+1}) - u_{\widetilde{B}_k}| \leq 2 \sup_{B_k} |u - u_{\widetilde{B}_k}| \lesssim \varphi(\ell_k), \]
			for every $k\geq 0$. Thus,
			\begin{equation}\label{eq:telescope}
				|u(X)| \leq \sum_{k=0}^N |u(X_k) - u(X_{k+1})| + |u(X_{N+1})| \lesssim \sum_{k=0}^N \varphi(\ell_k) + |u(X_{N+1})|. 
			\end{equation} 
			For the first term, using \eqref{eq:ass-phi}, we obtain:
			\begin{equation}\label{tmp:telescopevarphi}
			\sum_{k=0}^\infty \varphi(\ell_k)	\lesssim 
			\varphi(\ell) + \sum_{k=1}^\infty \int_{\rho^{-k} \ell}^{\rho^{-k+1}\ell} \varphi(s) \frac{ds}{s} =
			\varphi(\ell) + \int_0^{\ell} \varphi(s) \frac{ds}{s}
			\lesssim 
			\varphi(\ell).
			\end{equation}
			For the second term, since $\{X_k\} \subset \Gamma^\theta(y)$ and $X_k \to y$, the assumption $\restr{u}{\pom}^{{\rm nt. \, lim}, \theta}(y)=0$ gives $|u(X_{N+1})| \to 0$ as $N \to \infty$. 
			On the other hand, since $\delta(X) = \delta(X_0) \geq \lambda \ell$, by the monotonicity of $\varphi$ and its doubling property (see Lemma \ref{lem:pro-phi}), we also have
			\[ \varphi(\ell) = \varphi(\lambda^{-1} \cdot \lambda \ell)
			\leq \varphi(\theta \lambda \ell) \lesssim_\theta \varphi(\lambda \ell) \leq \varphi(\delta(X)).\]
			Combined with \eqref{eq:telescope} and \eqref{tmp:telescopevarphi}, we conclude that
			\[ |u(X)| \lesssim_\theta \varphi(\delta(X)), \] 
			where the implicit constant only depends on $\theta$, so in particular it is independent of $X$.	
		}
		\end{proof}	
		{
		Clearly Lemma \ref{lem:bound_U} implies that $u \in \mathscr{C}(\overline{\Omega})$ and $u=0$ on $\pom$. When $\Omega$ is bounded, the maximum principle guarantees that $u=0$ in $\Omega$, thus the proof of the uniqueness of solutions is complete when $\Omega$ is bounded. When $\Omega$ is unbounded with $\pom$ unbounded, we need the following global estimate.
		}	
		
		
		\begin{lemma} \label{lem:bound_Omega}
			In the current scenario, there exists a constant $C$ (depending on $n$, the ADR constant, the ellipticity constant of $L$, $C_{\varphi}$ in \eqref{eq:ass-phi}, and $\theta$) such that 
			\begin{equation}\label{eq:u-bound}
			u(X) \leq C \varphi(\delta(X)) \, \text{ for all } X \in \Omega.
			\end{equation}
		\end{lemma}
		
		\begin{proof}
			{
			The proof is very similar to that of Lemma \ref{lem:bound_U}, but since we have already proven that $u\in \mathscr{C}(\overline{\Omega})$ and $u=0$ on $\pom$, we can choose a more straightforward \emph{carrot path}, that is the line segment to the closest point at the boundary.
			
			Fix $X \in \Omega$. Take $\widehat{x} \in \pom$ such that $\delta(X) = |X-\widehat{x}|$. Clearly the line segment connecting $X$ to $\widehat{x}$, denoted as $\gamma$, is contained in $\Omega$, and every point $Z\in \gamma$ satisfies that $|Z-\widehat{x}| = \delta(Z)$. 
			Define the sequence of points $X_0 = X$ and $X_{k+1} = (X_k + \widehat{x}) /2$ for each $k \ge 0$.
			Since $U$ is an open neighborhood of $\pom$ in $\Omega$, there exists some $\varepsilon = \varepsilon(\widehat{x})$ such that $B(\widehat{x}, \varepsilon) \cap \Omega \subset U$. Pick $k_0 \in \NN$ so that $\delta(X_{k_0}) = \delta(X) / 2^{k_0} \leq \varepsilon_0 / 2$, which yields $X_{k_0} \subset B(\widehat{x}, \varepsilon) \cap \Omega \subset U$. Hence by Lemma \ref{lem:bound_U}, we have that 
			\[ |u(X_{k_0})| \leq C \varphi(\delta(X_{k_0})) \leq C\varphi(\delta(X)). \]
			Similarly as \eqref{eq:telescope} and \eqref{tmp:telescopevarphi}, we have that
			\begin{align*}
				|u(X)| \leq \sum_{k=0 }^{k_0-1} |u(X_k) - u(X_{k+1})| + |u(X_{k_0})| \lesssim \varphi(\delta(X)) + |u(X_{k_0})| \lesssim \varphi(\delta(X)),
			\end{align*}
			with a constant bigger than $C$ in Lemma \ref{lem:bound_U}.}
		\end{proof}

		To conclude the proof of Proposition \ref{prop:CaCa_uniqueness_general} for the case when $\Omega$ is unbounded with $\pom$ unbounded, we can proceed much as in the proof of Theorem~\ref{thm:well}. Fix $X_0 \in \Omega$ and pick $\widehat{x}_0 \in \pom$ such that $|X_0-\widehat{x}_0| = \delta(X_0)$. Since  \eqref{eq:u-bound}  works for any point in $\Omega$, the monotonicity of $\varphi$ yields 
		\begin{equation*}
		\sup_{B(\widehat{x}_0, R) \cap \overline{\Omega}} |u|
		\lesssim 
		\varphi(R), \qquad\forall \, R > 0.
		\end{equation*}
		With this and Boundary Hölder continuity from Lemma~\ref{lem:Holder} (since $u$ is a weak solution and $\restr{u}{\pom} \equiv 0$ by Lemma~\ref{lem:bound_Omega}) we get, for $R$ big enough so that $X_0 \in B(\widehat{x}_0, R/2)$:
		\begin{equation*}
		|u(X_0)| 
		\lesssim 
		\bigg( \frac{|X_0 - \widehat{x}_0|}{R} \bigg)^{\alpha_2} \sup_{B(\widehat{x}_0, R) \cap \overline{\Omega}} |u| 
		\lesssim 
		\delta(X_0)^{\alpha_2} \frac{\varphi(R)}{R^{\alpha_2}}
		\lesssim 
		\delta(X_0)^{\beta} \frac{\varphi(R)}{R^{\beta}}
		\underset{R \to \infty}{\longrightarrow} 
		0 
		\end{equation*}
		by Lemma~\ref{lem:pro-phi} and recalling that $\beta \leq \alpha_2$. Since $X_0$ was chosen arbitrarily, this shows that $u \equiv 0$ in $\Omega$, and the proof is complete. 
	\end{proof}


	\subsection{The Morrey-Campanato Dirichlet problem}\label{sec:Morrey}
	
	For our last application, we introduce the Morrey-Campanato spaces. The point of introducing these spaces is their connection with the $\dot{\mathscr{C}}^{\alpha}$ spaces, from which we will also be able to reduce their associated Dirichlet problems to the study of a $\mathscr{C}^{*, \alpha}$ problem.
	
		Throughout the section we assume the surface measure $\sigma := \restr{\mathcal{H}^n}{\pom}$ is doubling. 
	
	\begin{definition}[{\bf Generalized Morrey-Campanato spaces}] 
		Let $\Omega \subset \ree$ be an open set, and assume that $\sigma$ is doubling. Given a growth function $\varphi$ along with some integrability exponent $p \in [1, \infty)$, the associated \textit{generalized Morrey-Campanato space} $\E^{\varphi, p}(\pom)$ on $\pom$ is defined as the collection of all functions $f \in L_{\loc}^1(\pom, d\sigma)$ satisfying 
		\begin{align*}
		 \|f\|_{\E^{\varphi, p}(\pom)}
		:=
		\sup_{\substack{x \in \pom \\ r > 0}}  
		\frac{1}{\varphi(r)} \bigg(\fint_{\Delta(x, r)} |f-f_{\Delta(x, r)}|^p \, d\sigma \bigg)^{\frac1p}
		< \infty,
		\end{align*}
		where we recall that $\Delta(x, r) := B(x, r) \cap \pom$ and $f_{\Delta(x, r)} := \fint_{\Delta(x, r)} f \, d\sigma$.
	\end{definition}
	
	For each $p \in [1, \infty)$, $r \in (0, \infty)$, and $f \in L^1_{\loc}(\pom)$, we define the $L^p$-based mean oscillations of $f$ at a given scale $r$ as 
	\begin{align*}
	\osc_p(f; r) := \sup_{\substack{x \in \pom \\ 0 < s \leq r}} \bigg(\fint_{\Delta(x, s)} |f-f_{\Delta(x, s)}|^p \, d\sigma \bigg)^{\frac1p}. 
	\end{align*}
	
	\begin{lemma}\label{lem:osc}
		Let $\Omega \subset \ree$ be an open set and assume that $\sigma$ is doubling measure. Let $\varphi \in \mathcal{G}_\beta$ for some $\beta \in (0, 1)$. For each $f \in L^1_{\loc}(\pom, d\sigma)$, the following properties hold: 
		\begin{list}{{\rm (\theenumi)}}{\usecounter{enumi}\leftmargin=1cm \labelwidth=1cm \itemsep=0.2cm \topsep=.2cm \renewcommand{\theenumi}{\alph{enumi}}} 
			
			\item\label{list:osc-1} Fix $p \in [1, \infty)$. Then, as a function of $r$, the quantity $\osc_p(f; r)$ is non-decreasing in $r$. 
			
			\item\label{list:osc-3} Given $p \in [1, \infty)$, the function $f$ belongs to $\Ephip(\pom)$ if and only if $\osc_p(f; r)/\varphi(r)$, as a function in $r$, is bounded on $(0, \infty)$. More specifically, 
			\[
			\|f\|_{\Ephip(\pom)} = \sup_{r>0} \frac{\osc_p(f; r)}{\varphi(r)}. 
			\]
			
			\item\label{list:osc-5} If $\Delta(x, r) \subset \Delta(y, s)$ for some $x, y \in \pom$ and $r,s>0$, we have 
			\begin{equation*}
			|f_{\Delta(x, r)} - f_{\Delta(y, s)}| 
			\leq 
			\frac{\sigma(\Delta(y, s))}{\sigma(\Delta(x, r))} \osc_p(f; s),\quad 1\le p < \infty.
			\end{equation*}
			If in addition $f \in \Ephip(\pom)$, we deduce
			\begin{equation*}
			|f_{\Delta(x, r)} - f_{\Delta(y, s)}| 
			\lesssim  \varphi(s) \|f\|_{\Ephip(\pom)}, \quad 0<s \le 2r.
			\end{equation*}

			\item\label{list:osc-6} For any $0 < r < s$ and $x \in \pom$, 
			\begin{equation*}
			|f_{\Delta(x, r)} - f_{\Delta(x, s)}| 
			\lesssim \int_0^{4s} \osc_p(f; t) \frac{dt}{t}, \quad 1\le p < \infty.
			\end{equation*} 
			If in addition $f \in \Ephip(\pom)$, we have 
			\begin{equation*}
			|f_{\Delta(x, r)} - f_{\Delta(x, s)}| 
			\lesssim \varphi(s) \|f\|_{\Ephip(\pom)}, \quad 1\le p < \infty.
			\end{equation*}
		\end{list}
	\end{lemma}
	
	\begin{proof}
		Parts \eqref{list:osc-1} and \eqref{list:osc-3} are trivial. To show \eqref{list:osc-5}, we use Jensen's inequality: 
		\begin{multline*}
		|f_{\Delta(x, r)} - f_{\Delta(y, s)}| 
		\leq 
		\fint_{\Delta(x, r)} |f - f_{\Delta(y, s)}| d\sigma 
		=
		\frac{\sigma(\Delta(y, s))}{\sigma(\Delta(x, r))} \fint_{\Delta(y, s)} |f - f_{\Delta(y, s)}| d\sigma 
		\\ 
		\le 
		\frac{\sigma(\Delta(y, s))}{\sigma(\Delta(x, r))} \bigg(\fint_{\Delta(y, s)} |f - f_{\Delta(y, s)}|^p d\sigma \bigg)^{\frac1p}
		\leq 
		\frac{\sigma(\Delta(y, s))}{\sigma(\Delta(x, r))} \osc_p(f; s). 
		\end{multline*}
		And if $f \in \Ephip(\pom)$ and $0<s \le 2r$, noting that $\Delta(x, r) \subset \Delta(y, s) \subset \Delta(x, 2s) \subset \Delta(x, 4r)$, we use the doubling property of $\sigma$ and \eqref{list:osc-3} to get 
		\begin{equation*}
		|f_{\Delta(x, r)} - f_{\Delta(y, s)}| 
		\leq \frac{\sigma(\Delta(x, 4r))}{\sigma(\Delta(x, r))} \osc_p(f; s) 
		\lesssim \varphi(s)  \|f\|_{\Ephip(\pom)}. 
		\end{equation*}
		In turn, the first part of \eqref{list:osc-6} follows by iterating of \eqref{list:osc-5}. Indeed, let $k \geq 0$ be such that $2^k r < s \leq 2^{k+1} r$. Denoting $\Delta_\tau := \Delta(x, \tau)$ for any $\tau > 0$, since $\sigma$ is doubling and $\osc_p(f; \cdot)$ is non-decreasing from \eqref{list:osc-1}, a similar computation to the one in \eqref{eq:ssr} leads to
		\begin{multline*}
		|f_{\Delta_r} - f_{\Delta_s}| 
		\leq 
		\sum_{j=0}^k \big|f_{\Delta_{2^jr}} - f_{\Delta_{2^{j+1}r}}| + \big|f_{\Delta_{2^{k+1} r}} - f_{\Delta_s} \big| 
		\\ 
		\leq 
		\sum_{j=0}^k \frac{\sigma(\Delta_{2^{j+1}r})}{\sigma(\Delta_{2^jr})} \osc_p(f; 2^{j+1}r) 
		+ \frac{\sigma(\Delta_{2^{k+1}r})}{\sigma(\Delta_s)} \osc_p(f; 2^{k+1}r)
		\\ 
		\lesssim 
		\sum_{j=0}^k \osc_p(f; 2^{j+1}r)
		\lesssim  
		\int_{2r}^{2^{k+2}r} \osc_p(f; t) \frac{dt}{t} 
		\leq 
		\int_0^{4s} \osc_p(f; t) \frac{dt}{t}.  
		\end{multline*}
		If in addition $f \in \Ephip(\pom)$, we further obtain, using \eqref{eq:ass-phi} and the doubling property of $\varphi$:
		\begin{align*}
		|f_{\Delta_r} - f_{\Delta_s}| 
		\leq 
		\int_0^{4s} \varphi(t) \frac{dt}{t} \|f\|_{\Ephip(\pom)} 
		\lesssim 
		\varphi(4s) \|f\|_{\Ephip(\pom)}
		\lesssim 
		\varphi(s) \|f\|_{\Ephip(\pom)},
		\end{align*}
		which completes the proof. 
	\end{proof}

	The following proposition establishes the equivalences of the spaces we have been dealing with. This has been shown in \cite[Theorem 4]{MS}, but our proof is self-contained and simpler. 
	
	\begin{proposition}\label{prop:holder_campanato}
		Let $\Omega \subset \ree$ be an open set and assume that $\sigma$ is a doubling measure. Let $\varphi \in \mathcal{G}_\beta$ for some $\beta \in (0, 1)$. Then the spaces $\dcv(\pom)$ and $\Ephip(\pom)$ are the same, with equivalent norms. (In the inclusion $\Ephip(\pom) \hookrightarrow \dcv(\pom)$ we understand that for any $f \in \Ephip(\pom)$, we can modify $f$ in a set of zero $\sigma$-measure so that $f \in \dcv(\pom)$).
	\end{proposition}
	
	\begin{proof}
		The fact that $\|f\|_{\Ephip(\pom)} \lesssim \|f\|_{\dcv(\pom)}$ follows easily from the definitions, hence we  only prove the converse inequality. First, Lemma \ref{lem:osc} part \eqref{list:osc-6} implies that for each fixed $x \in \pom$, the sequence $\{f_{\Delta(x, r)}\}_r$ is Cauchy for each fixed $x \in \pom$. Hence, the limit exists and we can define $g(x) := \lim_{r \to 0^+} f_{\Delta(x, r)}$. By Lebesgue Differentiation Theorem (which we may apply because $\sigma$ is a doubling measure and $f \in L^1_\loc(\pom, d\sigma)$), we obtain that $f = g$ $\sigma$-a.e. in $\pom$.
		
		Fix $x, y \in \pom$ and call $s := |x-y|$. By definition of $g$, we have
		\begin{equation} \label{eq:difference_averages}
		|g(x) - g(y)| 
		=
		\lim_{r \to 0} |f_{\Delta(x, r)} - f_{\Delta(y, r)}|.
		\end{equation}		
		Now compute, for $r$ small enough, say $2r < s$,
		\begin{multline} \label{eq:comparison_averages}
		|f_{\Delta(x, r)} - f_{\Delta(y, r)}| 
		\leq 
		|f_{\Delta(x, r)} - f_{\Delta(x, s)}| + |f_{\Delta(x, s)} - f_{\Delta(y, 2s)}| + |f_{\Delta(y, 2s)} - f_{\Delta(y, r)}| 
		\\ 
		\lesssim 
		\varphi(2s) \|f\|_{\Ephip(\pom)}
		\lesssim 
		\varphi(s) \|f\|_{\Ephip(\pom)}, 
		\end{multline}
		where we have applied Lemma \ref{lem:osc} part \eqref{list:osc-6} to the first and last terms, and part \eqref{list:osc-5} to the second one since $\Delta(x, s) \subset \Delta(y, 2s)$ by definition of $s$. Plugging \eqref{eq:comparison_averages} into \eqref{eq:difference_averages} yields $\|g\|_{\dcv(\pom)} \lesssim \|f\|_{\Ephip(\pom)}$, as desired.
	\end{proof}

	With all these tools in hand, we finally consider the Dirichlet problem in generalized Morrey-Campanato spaces.
	
	\begin{theorem} \label{thm:CE_existence_general}
		Let $\Omega \subset \R^{n+1}$, $n \geq 2$, be an open set with ADR boundary so that either {\bf $\Omega$ is bounded or $\Omega$ is unbounded with $\pom$ being unbounded}. Let $L = -\div (A\nabla)$ be a real (non-necessarily symmetric) elliptic operator, and let $\omega_L$ be the associated elliptic measure. Then there exists $\beta \in (0, 1)$ (depending only on $n$, the ADR constant, and the ellipticity constant of $L$) such that for any $\theta>1$ and for any growth function $\varphi \in \mathcal{G}_\beta$, the Morrey-Campanato Dirichlet problem 		
		\begin{equation} \label{eq:Morrey_general}
		\begin{cases}
		u \in W^{1, 2}_{\loc} (\Omega), \\
		Lu = 0 \text{ in the weak sense in } \Omega, \\
		u \in \mathscr{C}^{*,\varphi}(\Omega), \\
		\restr{u}{\pom}^{{\rm nt. \, lim}, \theta} = f \in \Ephip(\pom), \, \sigma \text{-a.e. on } \pom, 
		\end{cases}
		\end{equation}
		has a solution. Moreover, $u \in \mathscr{C} (\overline{\Omega})$ satisfies $\restr{u}{\pom} = f$ $\sigma$-a.e. in $\pom$, and there is a constant $C$ (depending on $n$, the ADR constant, the ellipticity constant of $L$, and $C_\varphi$) such that 
		\begin{align*}
		\|u\|_{\mathscr{C}^{*, \varphi}(\Omega)} \le C \|f\|_{\mathscr{E}^{\varphi, p}(\pom)}. 
		\end{align*}		
		Furthermore, if $\omega_L \ll \sigma$ (in the sense that $\omega_L^X \ll \sigma$ for all $X \in \Omega$), such a solution can be represented as 
		\begin{equation}\label{CE-solution_general}
		u(X) = \int_\pom f(y) \, d\omega_L^X(y), \qquad X \in \Omega.
		\end{equation}
	\end{theorem}	
	\begin{proof}
		Since $\pom$ is ADR, the surface measure $\sigma$ is doubling. Therefore, Proposition~\ref{prop:holder_campanato} implies that there exists $g \in \dcv(\pom)$ such that $f = g$ $\sigma$-a.e. in $\pom$, with $\|f\|_{\Ephip(\pom)} \approx \|g\|_{\dcv(\pom)}$. By Theorem~\ref{thm:CaCa_existence_general}, the function 
		\begin{equation*}
		\widetilde{u}(X) = \int_\pom g(y) \, d\omega_L^X(y), \qquad X \in \Omega
		\end{equation*}
		is also a solution to the problem \eqref{CaCa-problem_general}. Since $f = g$ $\sigma$-a.e. in $\pom$, this means that $\widetilde{u}$ is a solution to the problem \eqref{eq:Morrey_general}. Moreover, Theorem~\ref{thm:CaCa_existence_general} also yields $u \in \mathscr{C}(\overline{\Omega})$ with $\restr{u}{\pom} = g$ in $\pom$, and 
		\begin{equation*}
		\|u\|_{\mathscr{C}^{*, \varphi}(\Omega)}
		\lesssim 
		\|g\|_{\mathscr{E}^{\varphi, p}(\pom)}
		\approx
		\|f\|_{\mathscr{E}^{\varphi, p}(\pom)}
		\end{equation*}
		
		For the second part, assume that $\omega_L \ll \sigma$ (in the sense that $\omega_L^X \ll \sigma$ for all $X \in \Omega$). Since $f = g$ $\sigma$-a.e. on $\pom$, we have $f = g$ $\omega_L^X$-a.e. on $\pom$ for all $X \in \Omega$. This easily implies that
		\begin{equation*}
		\int_\pom f \,d\omega_L^X
		=
		\int_\pom g \,d\omega_L^X,
		\end{equation*}
		so $u(X)$, as defined by \eqref{CE-solution_general}, coincides with $\widetilde{u}(X)$ for every $X \in \Omega$. This finishes the proof.
	\end{proof}

	\begin{proposition}
		Let $\Omega \subset \R^{n+1}$, $n \geq 2$, be an open set with ADR boundary so that either {\bf $\Omega$ is bounded or $\Omega$ is unbounded with $\pom$ being unbounded}. Let $L = -\div (A\nabla)$ be a real (non-necessarily symmetric) elliptic operator, and let $\omega_L$ be the associated elliptic measure. Given $\theta > 1$, assume that \eqref{eq:path} holds. Then there exists $\beta \in (0, 1)$ (depending only on $n$, the ADR constant, and the ellipticity constant of $L$) such that for any growth function $\varphi \in \mathcal{G}_\beta$, the problem~\eqref{eq:Morrey_general} has a unique solution, and hence it is well-posed.	
		
	\end{proposition}
	\begin{proof}
		As in the proof of Theorem~\ref{thm:CE_existence_general}, matters can be reduced to the proof of Proposition~\ref{prop:CaCa_uniqueness_general} by means of Proposition~\ref{prop:holder_campanato}. {In particular, note that the proof of Lemma \ref{lem:bound_U} only uses the fact that the boundary datum vanishes $\sigma$-a.e. on $\pom$.} Details are left to the reader.
	\end{proof}


\end{document}